\documentclass[a4paper]{amsart} 
 
\usepackage{times,txfonts}  
\usepackage{verbatim} 
\usepackage{amssymb} 
\usepackage{amsbsy} 
\usepackage{amscd} 
\usepackage{amsmath} 
\usepackage{amsthm} 

\newtheorem{theorem}{Theorem} 
\newtheorem{problem}{Problem} 
\newtheorem{lemma}{Lemma} 
\newtheorem{proposition}{Proposition} 
\newtheorem{claim}{Claim} 
\theoremstyle{remark} 
\newtheorem{remark}{Remark} 
\theoremstyle{definition} 
\newtheorem{defn}{Definition} 
 
\newcommand{\on}[1]{\mathop{\mathrm{#1}}\nolimits} 
\newcommand{\inv}{\on{inv}} 
 
\newcommand{\LPO}[1]{\thicklines \put( 45, 45){\line(1, 1){ 10}} 
\put( 45, 55){\line(1,-1){ 10}}\put( 57, 38){#1}\thinlines} 
\newcommand{\CPO}[1]{\thicklines \put( 95, 45){\line(1, 1){ 10}} 
\put( 95, 55){\line(1,-1){ 10}}\put(107, 38){#1}\thinlines} 
\newcommand{\RPO}[1]{\thicklines \put(145, 45){\line(1, 1){ 10}} 
\put(145, 55){\line(1,-1){10}}\put(157, 38){#1}\thinlines} 
\newcommand{\VLL}[2]{\put( 35, 90){#1}\put( 55, 90){#2} 
\put( 50, 10){\line(0, 1){ 80}}} 
\newcommand{\VCL}[2]{\put( 85, 90){#1}\put(105, 90){#2} 
\put(100, 10){\line(0, 1){ 80}}} 
\newcommand{\VRL}[2]{\put(135, 90){#1}\put(155, 90){#2} 
\put(150, 10){\line(0, 1){ 80}}} 
\newcommand{\RDL}[2]{\put(135, 90){#1}\put(155, 90){#2} 
\put(150, 10){\line(0, 1){ 8}}\put(150, 22){\line(0, 1){ 6}} 
\put(150, 32){\line(0, 1){ 6}}\put(150, 42){\line(0, 1){ 6}} 
\put(150, 52){\line(0, 1){ 6}}\put(150, 62){\line(0, 1){ 6}} 
\put(150, 72){\line(0, 1){ 6}}\put(150, 82){\line(0, 1){ 8}}} 
\newcommand{\HOL}[2]{\put(195, 55){#1}\put(190, 40){#2} 
\put(  0, 50){\line( 1,0){ 200}}} 
\newcommand{\HDL}[2]{\put(195, 55){#1}\put(190, 40){#2} 
\put(  0, 50){\line( 1,0){ 8}}\put( 12, 50){\line( 1,0){ 6}} 
\put( 22, 50){\line( 1,0){ 6}}\put( 32, 50){\line( 1,0){ 6}} 
\put( 42, 50){\line( 1,0){ 6}}\put( 52, 50){\line( 1,0){ 6}} 
\put( 62, 50){\line( 1,0){ 6}}\put( 72, 50){\line( 1,0){ 6}} 
\put( 82, 50){\line( 1,0){ 6}}\put( 92, 50){\line( 1,0){ 6}} 
\put(102, 50){\line( 1,0){ 6}}\put(112, 50){\line( 1,0){ 6}} 
\put(122, 50){\line( 1,0){ 6}}\put(132, 50){\line( 1,0){ 6}} 
\put(142, 50){\line( 1,0){ 6}}\put(152, 50){\line( 1,0){ 6}} 
\put(162, 50){\line( 1,0){ 6}}\put(172, 50){\line( 1,0){ 6}} 
\put(182, 50){\line( 1,0){ 6}}\put(192, 50){\line( 1,0){ 8}} 
} 
\newenvironment{ZU}[1]{\begin{center}\begin{picture}(220,100) 
{#1}\end{picture}\end{center}} 
 
\title[Resolution of idealistic filtration in dimension 3]{ 
Resolution of singularities of an idealistic filtration in dimension 3 after Benito-Villamayor 
} 
 
\author[H. Kawanoue and K. Matsuki]{Hiraku Kawanoue and Kenji Matsuki} 
 
\dedicatory{\footnotesize Dedicated to Prof.~Shigefumi Mori 
on the occasion of his 60-th birthday}

\address{{\rm Hiraku Kawanoue} \\ 
Research Institute for Mathematical Sciences \\ 
Kyoto University, Kyoto 606-8502} 
\email{kawanoue@kurims.kyoto-u.ac.jp} 
 
\address{{\rm Kenji Matsuki} \\ 
Department of Mathematics, Purdue University \\ 
150 N. University Street, West Lafayette, IN 47907-2067 
} 
\email{kmatsuki@math.purdue.edu} 
 
 
 
\begin{document} 
\begin{abstract} 
We establish an algorithm for resolution of singularities of an idealistic filtration in dimension 3 (a local version) in positive characteristic, incorporating the method recently developed by Benito-Villamayor into our framework. 
Although (a global version of) our algorithm only implies embedded resolution of surfaces in a smooth ambient space of dimension 3, a classical result known before, we introduce some new invariant which effectively measures how much the singularities are improved in the process of our algorithm and which strictly decreases after each blow up. 
This is in contrast to the well-known Abhyankar-Moh pathology of the increase of the residual order under blow up and the phenomenon of the ``Kangaroo'' points observed by Hauser. 
\end{abstract} 
 
\maketitle 
 
\setcounter{tocdepth}{2} 
\tableofcontents 
 
\begin{section}{Outline of the paper} 
The goals of this paper are two-fold. 
The first goal is to present the general mechanism of resolution of singularities (a local version) in the framework of the Idealistic Filtration Program in positive characteristic. 
The classical algorithm in characteristic zero works by induction on dimension based upon the notion of a hypersurface of maximal contact. 
Our algorithm in positive characteristic works by induction on the invariant ``$\sigma$'' 
based upon the notion of a leading generator system (cf. \cite{K}). 
Roughly speaking, the general mechanism splits into two parts; the first part is to reduce the problem in the general case to the one in the so-called ``monomial case'', and the second part is to solve the problem in the monomial case. 
 The first part of the general mechanism works in arbitrary dimension. The second part is quite subtle and difficult in positive characteristic (while it is easy in the classical setting in characteristic zero). 
The second goal is to establish the algorithm in dimension 3, by actually solving the problem of resolution of singularities in the monomial case. 
We incorporate the method recently developed by Benito-Villamayor \cite{BV3} into our framework. 
We introduce some new invariant, which effectively measures how much the singularities are improved in the process of our algorithm and which 
strictly decreases after each blow up. 
This is in clear contrast to the well-known Abhyankar-Moh pathology 
of the increase of the residual order (cf. \cite{MR935710}) and the phenomenon of 
the ``Kangaroo'' points observed by Hauser (cf. \cite{Hau10}) and 
others. 
We note that the algorithm by Benito-Villamayor 
(cf. \cite{BV1}\cite{BV2}\cite{BV3}), which works by induction 
on dimension based upon the notion of a generic projection, is different from our algorithm, and that even the setting of the monomial case is different from ours by definition. 
It is something of a 
surprise that we can share the ``same'' method when our and their approaches are different. 
Establishing the algorithm in dimension 4 or above remains as an open problem. 
 
We remark that (a global version, which will be published elsewhere, of) 
our algorithm in dimension 3 only yields embedded resolution 
of surfaces in a nonsingular ambient 3-fold over over an algebraically closed field of any characteristic. 
This is a classical result known for more than 50 years since the time of Abhyankar, Hironaka and others (cf. \cite{Abh66}\cite{Cutkosky}\cite{Hau97}\cite{Hir67}). 
The front line of research, thanks to the works of Abhyankar, Cossart-Piltant, and Cutkosky (cf. \cite{Abh66}\cite{CP08}\cite{CP09}\cite{CJS}\cite{Cutkosky}), goes way beyond, establishing resolution of singularities of a 3-fold over an algebraically closed field of positive characteristic. 
We believe, however, that the method of our paper should provide a first step toward the open problem of embedded resolution of singularities of 3-folds in a nonsingular ambient 4-fold in positive characteristic, and also in higher dimensional cases. 
 
The outline of the paper goes as follows. 
 
After \S 1, which describes the outline of the paper, we give a brief overview of the contents of the paper in \S 2. 
In \S 3, we present a quick review on the algorithm in characteristic zero. 
Our algorithm in positive characteristic is modeled upon the classical algorithm in characteristic zero. 
The review is given in such a way that the reader can see the similarities and differences between our algorithm and the classical algorithm through an easy and accessible comparison. 
In \S 4, we present the general mechanism of our algorithm for resolution of singularities in positive characteristic. 
In \S 5, we present a solution to the problem of resolution of singularities 
in the monomial case in dimension 3, thus completing 
our algorithm as a whole in dimension 3. 
 
We assume, throughout the entire paper, that the base field is an algebraically closed field $k = \overline{k}$ of characteristic zero $\mathrm{char}(k) = 0$ or  positive characteristic $\mathrm{char}(k) = p > 0$. 
Therefore, there is no danger in \emph{not} distinguishing 
the two notions ``smooth over $k$'' and ``regular'' (meaning that every local ring of a variety is a regular local ring), and in using the word ``nonsingular'' as a synonym. 
In the case where the base field $k$ is perfect, our algorithm over its algebraic closure $\overline{k}$ is invariant under the action of the Galois group $\mathrm{Gal}(\overline{k}/k)$ and hence descends to the algorithm over the original base field $k$. 
The case where the base field is not perfect will be investigated elsewhere. 
 
\smallskip 
 
\textbf{Acknowledgments}: We would like to express our hearty thanks 
to Profs. A. Benito and O. Villamayor for explaining their results to us with great patience. 
We owe \S 5 of this paper to their ideas, which inspired us to come up with further improvements. We would also like to thank Profs. R. Blanco, A. Bravo, S. Encinas, and Dr. M.L. Garc\'ia for their warm hospitality during our stay in Madrid, Spain, in May of 2011. 
The first author would like to thank Prof. Hauser in Vienna, Austria, and Dr. Bernd Schober in Regensburg, Germany. 
The conversation with them led to a crucial development of the paper. 
The second author would like to thank Prof. Y. Lee for the hospitality during his stay in Seoul, Korea, and Prof. Y. Kawamata at Tokyo University, Japan, for the consistent and warm encouragement. 
He would also like to thank Profs. D. Arapura, W. Heinzer, J. Lipman, T.T. Moh, B. Ulrich, J. W{\l}odarczyk at Purdue University in Indiana, U.S.A., for their advice and many invaluable comments. 
Our indebtedness to the people at RIMS in Kyoto, Japan, is more than a word could express. 
The deep appreciation goes to Profs. S. Mori and S. Mukai, who have been supporting us both mathematically and personally throughout our entire project toward resolution of singularities in positive characteristic. 
Prof. N. Nakayama's help was essential for the completion of our paper. 
We would like to thank the referee for giving us many helpful comments to improve the paper. 
The first author is partially supported by the Grant-in-Aid for 
Young Scientists (B) and by the Sumitomo Foundation. 
The second author is partially supported by the NSA grant H98230-10-1-0172. 
\end{section} 
\begin{section}{Overview} 
The problem of resolution of singularities in its simplest form is stated as follows: 
\begin{problem}[Resolution of singularities] Given an algebraic variety 
$X$ over $k$, find a proper birational map 
$X \overset{\pi}\leftarrow \widetilde{X}$ from 
a nonsingular variety $\widetilde{X}$. 
\end{problem} 
The above problem is reduced to the following problem of embedded resolution of singularities, if our solution to 
the latter is \emph{functorial} in the sense that 
it is stable under the pull-back by smooth morphisms. 
\begin{problem}[Embedded resolution of singularities] Given an algebraic variety $X$, embedded as a closed subvariety in a smooth 
ambient variety $W$ over $k$, i.e., $X \subset W$, find a sequence 
starting from $(X_0\subset W_0)=(X\subset W)$ 
$$ 
(X_0 \subset W_0) 
\leftarrow 
\!\cdots 
\leftarrow (X_i \subset W_i) 
\overset{\pi_{i+1}}\leftarrow (X_{i+1} \subset W_{i+1}) 
\leftarrow 
\!\cdots 
\leftarrow (X_l \subset W_l), 
$$ 
where 
$W_i \leftarrow W_{i+1}$ is a blow up with smooth center $C_i \subset W_i$ 
which does not contain the strict transform $X_i$ of $X_0$ in year $i$ , i.e., 
$C_i \not\supset X_i$, and which is transversal to the exceptional 
divisor $D_i$, i.e., $C_i \pitchfork D_i$ (maybe contained in $D_i$), 
such that the last strict transform $X_l$ is nonsingular and transversal 
to $D_l$. 
\end{problem} 
 
Note that in the above formulation we do \emph{not} require that 
the center is contained in the singular locus of the strict transform $C_i \subset \mathrm{Sing}(X_i)$ or even that the center is contained in the strict transform $C_i \subset X_i$. 
 
Note also that we say $C_i$ is transversal to $D_i$, denoted by 
$C_i \pitchfork D_i$, if at any closed point $P \in W_i$ there 
exist a regular system of parameters $(x_1, \ldots, x_d)$ at $P$, 
taken from ${\mathfrak m}_P \subset {\mathcal O}_{W,P}$, and subsets 
$A, B \subset \{1, \ldots, d = \dim W_i\}$ such that 
$C_i = \bigcap_{\alpha \in A}\{x_{\alpha} = 0\}$ and 
$D_i = \bigcup_{\beta \in B}\{x_{\beta} = 0\}$ in a neighborhood of $P$. 
 
The implication 
``Solution to Problem 2 (in a functorial way) 
$\Rightarrow$ Solution to Problem 1'' 
can be seen as follows: Given an algebraic variety $X$, 
decompose it into the union of affine open subvarieties 
$X = \bigcup_{\lambda \in \Lambda}X_{\lambda}$ with embeddings 
$X_{\lambda} \subset W_{\lambda} = {\mathbb A}^{n_{\lambda}}$. 
Take embedded resolutions of singularities  $(X_{\lambda} \subset W_{\lambda}) \leftarrow (\widetilde{X_{\lambda}} \subset \widetilde{W_{\lambda}})$. 
We have only to see that the $\widetilde{X_{\lambda}}$'s 
patch together by the \emph{functoriality} to obtain 
resolution of singularities $X = \bigcup_{\lambda \in \Lambda}X_{\lambda} \leftarrow \widetilde{X} = \bigcup_{\lambda \in \Lambda}\widetilde{X_{\lambda}}$. 
For the detail of the proof, we refer the reader to \cite{W}. 
 
How can we solve Problem 2 (in a functorial way so that we can solve Problem 1 also) ? We would like to use ``induction'' on dimension. 
However, the inductive scheme to approach Problem 2 is not clear as stated, at least not obvious. 
We will present in \S 3 the reformulation by Hironaka where in characteristic zero the inductive scheme on dimension is more transparent and used classically, and in \S 4 another reformulation in our framework where in positive characteristic the inductive scheme on the invariant ``$\sigma$'' emerges. 
 
We give an overview of the contents of our paper in the following. 
 
\S 3 is devoted to a quick review on the algorithm 
in characteristic zero. 
 
In \S 3.1, we give the precise statement of the reformulation by 
Hironaka (following the language used by Villamayor). 
In short, 
the reformulation turns the problem of embedded resolution of 
singularities into a game of reducing the order of an ideal on 
a nonsingular ambient variety. 
We start from a triplet of data 
$(W, ({\mathcal I},a), E)$, where $W$ is a nonsingular variety 
over $k$, $({\mathcal I},a)$ is 
a pair consisting of a \emph{nonzero} coherent ideal sheaf 
${\mathcal I}$ on $W$ and a fixed positive 
integer $a \in {\mathbb Z}_{> 0}$, and where $E$ is a simple normal 
crossing divisor on $W$, called a boundary (divisor). 
We define 
its singular locus to be 
$\mathrm{Sing}({\mathcal I},a) := \{P \in W \mid 
\mathrm{ord}_P({\mathcal I}) \geq a\}$. 
Then we are required 
to find a sequence of transformations (see \S 3.1 for the precise 
definition of a transformation) 
starting from 
$(W_0,({\mathcal I}_0,a),E_0)=(W,({\mathcal I},a),E)$ 
 
\begin{align*} 
(W_0,({\mathcal I}_0,a),E_0) & \leftarrow \!\cdots\! \leftarrow 
(W_i,({\mathcal I}_i,a),E_i)  \overset{\pi_{i+1}}\leftarrow 
(W_{i+1},({\mathcal I}_{i+1},a),E_{i+1}) \\ 
& \leftarrow 
\!\cdots\! \leftarrow (W_l,({\mathcal I}_l,a),E_l) 
\end{align*} 
such that $\mathrm{Sing}({\mathcal I}_l,a) = \emptyset$. 
That is to say, the order of the last ideal ${\mathcal I}_l$ is 
everywhere below $a$. 
We call such a sequence ``resolution of 
singularities for $(W, ({\mathcal I},a), E)$''. 
 
In this reformulation, the inductive scheme of the problem can 
be stated, though naive, simply as follows: Given 
a triplet of data 
$(W, ({\mathcal I},a), E)$, find another triplet of data 
$(H,({\mathcal J},b), F)$ with 
$H \underset{\text{closed}}\subset W$ and 
$\dim H = \dim W - 1$ such that constructing resolution of 
singularities for $(W, ({\mathcal I},a), E)$ is equivalent 
to constructing one for $(H,({\mathcal J},b),F)$, a property 
symbolically denoted by 
$$(W, ({\mathcal I},a), E) \underset{\mathrm{equivalent}} 
\sim (H,({\mathcal J},b),F).$$ 
 
In \S 3.2, we discuss the inductive scheme on dimension of 
the algorithm in characteristic zero. 
It starts with the 
following key inductive lemma: Given $(W, ({\mathcal I},a), E)$ 
assumed to be under a certain condition $(\star)$, the lemma 
constructs $(H,({\mathcal J},b),F)$ with 
$H \underset{\text{closed}}\subset W$ and $\dim H = \dim W - 1$ 
(locally around a fixed point $P \in W$) which satisfies 
one of the following two. 
 
\smallskip 
 
\indent{\rm (i)}\quad The ideal ${\mathcal J}$ is a zero sheaf, 
i.e., ${\mathcal J} \equiv 0$:  In this case, we take 
the transformation with center $H =  \mathrm{Sing}({\mathcal I},a)$ 
$$(W, ({\mathcal I},a), E) \leftarrow (\widetilde{W}, 
(\widetilde{\mathcal I},a), \widetilde{E}).$$ 
After the transformation, we have 
$\mathrm{Sing}(\widetilde{\mathcal I},a) = \emptyset$ 
and hence resolution of singularities is achieved. 
 
\indent{\rm (ii)}\quad The ideal ${\mathcal J}$ is \emph{not} a zero sheaf, 
i.e., ${\mathcal J} \not\equiv 0$:  In this case, we have 
$$(W, ({\mathcal I},a), E) \underset{\mathrm{equivalent}}\sim 
(H,({\mathcal J},b),F).$$ 
Therefore, by constructing resolution of singularities for 
$(H$, 
$({\mathcal J},b), F)$ by induction on dimension, we achieve 
resolution of singularities for $(W, ({\mathcal I},a), E)$. 
 
\smallskip 
 
The hypersurface ``$H$'' in the key inductive lemma 
is called \emph{a hypersurface of maximal contact}. 
 
There are three shortcomings of the above key inductive lemma when we look at the goal of establishing the algorithm in characteristic zero: 
 
\medskip 
 
\indent{\rm \textcircled{\footnotesize 1}}\quad 
we have to impose condition 
$(\star)$ on $(W,({\mathcal I},a),E)$, 
 
\indent{\rm \textcircled{\footnotesize 2}}\quad the construction of 
$(H,({\mathcal J},b),F)$ is only local, and 
 
\indent{\rm \textcircled{\footnotesize 3}}\quad the invariant 
``$\mathrm{ord}$'' may increase after a transformation, even 
though our ultimate goal is to reduce its value to be below 
the fixed level $a$. 
 
\medskip 
 
In order to overcome these shortcomings, we adopt the following mechanism. 
 
\medskip 
 
\indent{\rm (1)}\quad We introduce a pair of invariants 
$(\mathrm{w\text{-}ord},s)$ and its associated triplet of data $(W,({\mathcal K},\kappa),G)$, called the modification of the original triplet $(W,({\mathcal I},a),E)$, having the following properties. 
(We remark that, when we want to emphasize the dimension of the ambient variety, we add ``$\dim$'' to the pair of invariants as the first factor, making the pair into a triplet $(\dim, \mathrm{w\text{-}ord},s)$.) 
 
\smallskip 
 
\indent$\bullet$\quad The maximum locus of the pair of invariants coincides 
with the singular locus of the modification, i.e., 
$$\mathrm{MaxLocus}(\mathrm{w\text{-}ord},s) = \mathrm{Sing}({\mathcal K},\kappa).$$ 
Moreover, after each transformation, the value of the pair 
$(\mathrm{w\text{-}ord},s)$ never increases, 
and the locus where the value of the pair takes the same maximum value as the original one coincides with the singular locus of the transformation of the modification. 
(Note that the transformations of $(W,({\mathcal K},\kappa),G)$ induce those of $(W,({\mathcal I},a),E)$.)  This means that resolution of singularities for $(W,({\mathcal K},\kappa),G)$ implies the strict decrease of the (maximum) value of the pair $(\mathrm{w\text{-}ord},s)$. 
 
\indent$\bullet$\quad Even though the original triplet 
$(W,({\mathcal I},a),E)$ may not satisfy condition $(\star)$, the modification $(W,({\mathcal K},\kappa),G)$ does. 
 
\smallskip 
 
\indent{\rm (2)}\quad We apply the key inductive lemma to 
$(W,({\mathcal K},\kappa),G)$. 
We find a triplet $(H,({\mathcal J},b),F)$ with $H \underset{\text{closed}}\subset W$ and $\dim H = \dim W - 1$ such that we are either in Case (i), or in Case (ii) where resolution of singularities for $(H,({\mathcal J},b),F)$ implies the one for $(W,({\mathcal K},\kappa),G)$. 
 
\indent{\rm (3)}\quad In Case (i), by a single blow up with center $H$, 
we achieve resolution of singularities for $(W,({\mathcal K},\kappa),G)$. 
In Case (ii), by induction on dimension, we achieve resolution of singularities for $(H,({\mathcal J},b),F)$, hence for $(W,({\mathcal K},\kappa),G)$. 
In both cases, we have the strict 
decrease of the (maximum) value of the pair $(\mathrm{w\text{-}ord},s)$.

\indent{\rm (4)}\quad Repeatedly decreasing the value of the pair 
$(\mathrm{w\text{-}ord},s)$ this way, we reach the case where $\mathrm{w\text{-}ord} = 0$. 
The condition $\mathrm{w\text{-}ord} = 0$ is equivalent to saying 
that the ideal is generated by some monomial of the defining equations of the components in the boundary divisor $E$. 
Thus we call the case where 
$\mathrm{w\text{-}ord} = 0$ \emph{the monomial case}. 
 
\indent{\rm (5)}\quad Finally, we have only to construct 
resolution of singularities in the monomial case, which 
can be done easily in characteristic zero. 
 
\medskip 
 
The description of the inductive scheme above is only local, and hence 
we have overcome shortcomings \textcircled{\footnotesize 1} 
and \textcircled{\footnotesize 3} only so far. 
The way we overcome shortcoming \textcircled{\footnotesize 2} 
is discussed in \S 3.5 via the strand of invariants woven in \S 3.3. 
 
In \S 3.3, we present the inductive scheme explained in \S 3.2 in terms of weaving the strand of invariants ``$\inv_{\mathrm{classic}}$''. 
The strand ``$\inv_{\mathrm{classic}}$'' consists of the units of the form 
$(\dim H^j, \mathrm{w\text{-}ord}^j,s^j)$, the dimension of the hypersurface of maximal contact $H_j$ followed by the pair as described in \S 3.1 computed from the triplet $(H^j,({\mathcal J}^j,b^j),F^j)$ at the $j$-th stage, and ends either with $(\dim H^m,\infty)$ or with $(\dim H^m,0,\Gamma)$ depending on whether the last triplet $(H^m,({\mathcal J}^m,b^m),F^m)$ is in Case (i) of the key inductive lemma or it is in the monomial case. 
That is to say, ``$\inv_{\mathrm{classic}}$'' takes the following form 
\begin{align*} 
\text{``$\inv_{\mathrm{classic}}$''} 
=\ & 
(\dim H^0 = \dim W,\mathrm{w\text{-}ord}^0,s^0) \cdots 
(\dim H^j, \mathrm{w\text{-}ord}^j,s^j)  \\ 
&\cdots (\dim H^{m-1}, \mathrm{w\text{-}ord}^{m-1},s^{m-1}) 
\begin{cases} 
(\dim H^m, \infty), 
\ \text{or} 
\\ 
(\dim H^m,0,\Gamma). 
\end{cases} 
\end{align*} 
Note that we do not include the invariant $s$ in the last unit. 
 
We choose the center of blow up to be 
$H^m$ when $\mathrm{w\text{-}ord}^m = \infty$ 
according to Case (i) of the key inductive lemma, or 
the maximum locus of the invariant $\Gamma$ on $H^m$ 
when $\mathrm{w\text{-}ord}^m = 0$ according to 
the procedure in the monomial case (which is discussed in \S 3.4). 
This is equivalent to choosing the center of blow up to be 
the maximum locus of ``$\inv_{\mathrm{classic}}$''. 
This leads to the decrease of the (maximum) value of 
``${\inv_{\mathrm{classic}}}$''.  Showing that 
the (maximum) value of ``$\inv_{\mathrm{classic}}$'' 
can not decrease infinitely many times, we achieve 
resolution of singularities for $(W, ({\mathcal I},a), E)$. 
 
In \S 3.4, we discuss the procedure of constructing resolution of singularities in the monomial case, the only remaining task to complete the classical algorithm.  In characteristic zero, this can be done easily and purely from the combinatorial data obtained by looking at the monomial in consideration, manifested as the invariant $\Gamma$. 
 
The strand  ``$\inv_{\mathrm{classic}}$'' a priori depends on the 
choice of the hypersurfaces of maximal contact we take in the process 
of weaving, and it is a priori only locally defined.  However, 
the strand ``$\inv_{\mathrm{classic}}$''  is actually independent 
of the choice, and hence it is globally well-defined.  This can be 
shown classically by the so-called Hironaka's trick, or more recently 
by inserting W{\l}odarczyk's ``homogenization'' or 
the first author's ``differential saturation'' into 
the construction of the modification.  Therefore, the process 
of resolution of singularities, where we take the center of blow up 
to be the maximum locus of  ``$\inv_{\mathrm{classic}}$'', is also 
globally well-defined.  This is how we overcome shortcoming 
\textcircled{\footnotesize 2} of the key inductive lemma, 
accomplishing the globalization of the algorithm in \S 3.5. 
 
This completes the overview of \S 3. 
 
\medskip 
 
\S 4 is devoted to describing the general mechanism of our algorithm 
in positive characteristic, which is closely modeled upon the algorithm 
in characteristic zero explained in \S 3.  (We remark that our algorithm 
is also valid in characteristic zero.) 
 
In \S 4.1, we give the statement of a further reformulation 
(of Problem 2), which allows us to present the inductive structure 
on the invariant ``$\sigma$''.  We start from a triplet of data 
$(W, {\mathcal R}, E)$, where we replace the pair $({\mathcal I},a)$ 
in the classical triplet $(W,({\mathcal I},a),E)$ with ${\mathcal R}$.  
Here 
${\mathcal R} = \bigoplus_{a \in {\mathbb Z}_{\geq 0}}({\mathcal I}_a,a)$ 
represents an idealistic filtration of i.f.g.\! type (short for 
``integrally and finitely generated type''), i.e., a finitely 
generated and graded (by the nonnegative integer 
$a \in {\mathbb Z}_{\geq 0}$ called the level of the ideal 
${\mathcal I}_a$ in the first factor) ${\mathcal O}_W$-algebra 
satisfying the condition 
${\mathcal O}_W = {\mathcal I}_0 \supset {\mathcal I}_1 \supset 
{\mathcal I}_2 \supset \cdots \supset {\mathcal I}_a \supset \cdots$.  
We define its singular 
locus to be $\mathrm{Sing}({\mathcal R}) := \{P \in W \mid 
\mathrm{ord}_P({\mathcal I}_a) \geq a\quad 
\forall a \in {\mathbb Z}_{\geq 0}\}$. 
Then we are required to find a sequence of transformations 
(see \S 4.1 for the precise definition of a transformation) 
starting from 
$(W_0,{\mathcal R}_0,E_0)=(W,{\mathcal R},E)$ 
\begin{align*} 
(W_0,{\mathcal R}_0,E_0) & \leftarrow \cdots  \leftarrow 
(W_i,{\mathcal R}_i,E_i)  \overset{\pi_{i+1}}\leftarrow 
(W_{i+1},{\mathcal R}_{i+1},E_{i+1}) \\ 
& \leftarrow \cdots \leftarrow (W_l,{\mathcal R}_l,E_l) 
\end{align*} 
such that $\mathrm{Sing}({\mathcal R}_l,a) = \emptyset$. 
We call such a sequence ``resolution of singularities for 
$(W, {\mathcal R}, E)$''.  So far it is perfectly parallel 
to the story in characteristic zero, and there is nothing 
unique to the case in positive characteristic.  In fact, 
resolution of singularities for $(W,({\mathcal I},a),E)$ 
is equivalent to resolution of singularities for 
$(W,{\mathcal R},E)$ where the idealistic filtration of i.f.g.\! 
type is given by the formula ${\mathcal R} 
= \bigoplus_{n \in {\mathbb Z}_{\geq 0}}({\mathcal I}^{\lceil {n}/{a}\rceil},n)$.  
However, the reformulation allows us to introduce the invariant 
$\sigma$, which plays the key role in our inductive scheme. 
 
In \S 4.2, we discuss the inductive scheme of our algorithm on 
the invariant $\sigma$ in positive characteristic.  Here, unlike 
in characteristic zero, we have no key inductive lemma.  In fact, 
the statement of the key inductive lemma fails to hold in 
positive characteristic as is demonstrated by an example due to R. Narasimhan. 
That is to say, there is \emph{no} smooth Hypersurface of 
Maximal Contact 
(HMC for short) in general.  However, we can still introduce 
the notion of \emph{a Leading Generator System} (LGS for short), 
which should be considered as a collective substitute in positive 
characteristic for the notion of an HMC in characteristic zero.  
Our basic strategy is to follow the construction of the algorithm 
in characteristic zero, replacing an HMC (leading to the induction 
on dimension) with an LGS (leading to the induction on the invariant 
$\sigma$).  The description of the new inductive scheme goes as follows. 
 
\medskip 
 
\indent{\rm (1)}\quad We introduce a triplet of invariants 
$(\sigma, \widetilde{\mu},s)$. 
 
If $(\sigma,\widetilde{\mu},s) = (\sigma,\infty,0)$ or $(\sigma,0,0)$, 
then we do not construct the modification $(W',{\mathcal R}',E')$. 
 
\smallskip 
 
\indent$\bullet$\quad In case $(\sigma,\widetilde{\mu},s) = (\sigma,\infty,0)$, 
we blow up with center $C = \mathrm{Sing}({\mathcal R})$. 
The nonsingularity of $C$ is guaranteed by the Nonsingularity Principle (cf. \cite{K} \cite{KM}), while the transversality of $C$ to the boundary $E$ is guaranteed by the invariant $s = 0$.  After the blow up, the singular locus becomes empty, and resolution of singularities for $(W,{\mathcal R},E)$ is achieved. 
 
\indent$\bullet$\quad In case $(\sigma,\widetilde{\mu},s) = (\sigma,0,0)$, 
we are in the monomial case by definition, and we go to Step (5). 
 
\smallskip 
 
If $(\sigma, \widetilde{\mu},s) \neq (\sigma,\infty,0)$ or $(\sigma,0,0)$, 
then we construct its associated triplet of data $(W',{\mathcal R}',E')$, called the modification of the original triplet $(W,{\mathcal R},E)$, having the following properties.  (Actually the ambient space for the modification stays the same, i.e., $W' = W$.) 
 
\smallskip 
 
\indent$\bullet$\quad Resolution of singularities for $(W',{\mathcal R}',E')$ 
implies the strict decrease of the (maximum) value of the triplet $(\sigma, \widetilde{\mu},s)$.  Note that the value of the triplet never increases after transformations. 
 
\indent$\bullet$\quad Either the value of $\sigma$ strictly decreases, 
or the value of $\sigma$ stays the same but the number of the components in the boundary strictly drops.  In either case, we have $(\sigma,\# E) > (\sigma',\# E')$. 
 
\smallskip 
 
\indent{\rm (2)}\quad There is \emph{no} key inductive lemma in our new setting. 
 
\indent{\rm (3)}\quad When $(\sigma,\widetilde{\mu},s) \neq (\sigma,\infty,0)$ 
or $(\sigma,0,0)$ in (1), 
we achieve resolution of singularities for $(W',{\mathcal R}',E')$ 
by induction $(\sigma,\# E)$, which implies the strict decrease of 
the (maximum) value of the triplet $(\sigma, \widetilde{\mu},s)$. 
 
\indent{\rm (4)}\quad Repeatedly decreasing the value of the triplet 
$(\sigma,\widetilde{\mu},s)$, we reach the case where 
 $(\sigma,\widetilde{\mu},s) = (\sigma,\infty,0)$ or $(\sigma,0,0)$, 
the monomial case. 
 
\indent{\rm (5)}\quad Finally we have only to construct resolution 
of singularities in the monomial case in order to achieve resolution of singularities of the original triplet $(W, {\mathcal R}, E)$.  However, the problem of resolution of singularities in the monomial case in positive characteristic is quite subtle and very difficult.  We only provide a solution in dimension 3 in \S 5. 
 
\smallskip

In \S 4.3, we present the inductive scheme explained in \S 4.2 in terms of weaving the strand of invariants ``$\inv_{\mathrm{new}}$''.  The strand ``$\inv_{\mathrm{new}}$'' consists of the units of the form $(\sigma^j,\widetilde{\mu}^j,s^j)$ computed from the triplet $(W^j,{\mathcal R}^j,E^j)$ at the $j$-th stage, and ends either with $(\sigma^m,\infty,0)$ or with $(\sigma^m,0,0)$.  That is to say, ``$\inv_{\mathrm{new}}$'' takes the following form 
\begin{align*} 
\text{``$\inv_{\mathrm{new}}$''} 
=\ & 
(\sigma^0,\widetilde{\mu}^0,s^0) \cdots (\sigma^j,\widetilde{\mu}^j,s^j) 
\\ & 
\cdots 
(\sigma^{m-1},\widetilde{\mu}^{m-1},s^{m-1}) 
\begin{cases} 
(\sigma^m,\infty,0), 
\ \text{or}\\ 
(\sigma^m,0,0). 
\end{cases} 
\end{align*} 
Note that we \emph{do} compute and include the invariant $s$ 
in the last unit, in contrast to the weaving of ``$\inv_{\mathrm{classic}}$''. 
 
Resolution of singularities of the last $m$-th triplet $(W^m,{\mathcal R}^m,E^m)$ can be achieved by taking the transformation with center $\mathrm{Sing}({\mathcal R}^m)$ when $(\sigma^m,\widetilde{\mu}^m,s^m) = (\sigma^m,\infty,0)$ or by resolution of singularities in the monomial case when $(\sigma^m,\widetilde{\mu}^m,s^m) = (\sigma^m,0,0)$.  This leads to the decrease of the value of ``$\inv_{\mathrm{new}}$''.  Showing that the value of ``$\inv_{\mathrm{new}}$''  can not decrease infinitely many times, we (should) achieve resolution of singularities for $(W,{\mathcal R},E)$ (assuming that the problem of resolution of singularities in the monomial case is solved). 
 
There is one important remark to make.  In \S 4.2 we do \emph{not} claim 
that the maximum locus of the triplet $(\sigma, \widetilde{\mu},s)$ coincides with the singular locus $\mathrm{Sing}({\mathcal R}')$ of the modified idealistic filtration, and 
in \S 4.3 we do \emph{not} claim that 
the strand ``$\inv_{\mathrm{new}}$'' is a global invariant whose maximum locus gives the globally well-defined center of blow up for the algorithm.  The discussion of our new inductive scheme and the weaving of the new strand of invariants are, therefore, only local.  (See Remark 3 for the precise meaning of ``local''.)  In fact, ``$\inv_{\mathrm{new}}$'' as presented is independent of the choice of the LGS's we take in the process of weaving, just like ``$\inv_{\mathrm{classic}}$'' is independent of the choice of the HMC's we take in the process of weaving in characteristic zero.  Nevertheless, the gap between $\mathrm{MaxLocus}(\sigma, \widetilde{\mu},s)$ and $\mathrm{Sing}({\mathcal R}')$ may occur when $\widetilde{\mu} = 1$, and hence that the maximum locus of ``$\inv_{\mathrm{new}}$'' does not provide the global center of blow up as it is.  This calamity is unique to our setting, and never existed in the classical setting.  We can fix this calamity by making certain adjustments to the strand, the description of which is rather technical and will be published elsewhere.  Therefore, in this paper, we restrict ourselves to the local description, which, we believe, still captures the essence of the inductive scheme on the invariant $\sigma$. 
 
In \S 4.4, we mention briefly the reason why resolution of singularities in the monomial case is difficult in positive characteristic, while it is easy in characteristic zero. 
 
This completes the overview of \S 4. 
 
\bigskip 
 
\S 5 is devoted to the detailed discussion of resolution of singularities in the monomial case in dimension 3, incorporating the method recently developed by Benito-Villamayor into our framework with some improvements.  In fact, we introduce a new invariant, which strictly drops after each blow up and hence effectively shows the termination of the procedure.  This is the most subtle and difficult part of this paper. 
 
The invariant $\tau$ represents the number of the elements in the LGS, which takes the value $0,1,2,3$ in dimension 3 and which never decreases under transformations.  When $\tau = 0,2,3$, resolution of singularities in the monomial case is rather easy.  Therefore, we focus our attention to the case where $\tau = 1$ in \S 5. 
 
We are thus in the situation where (analytically) we have a unique element $(h,p^e)$ in the LGS  and where $h$ is of the following form with respect to a regular system of parameters $(x,y,z)$ at $P \in W$, taken from $\widehat{{\mathfrak m}_P} \subset \widehat{{\mathcal O}_{W,P}}$, 
$$h = z^{p^e} + a_1z^{p^e} + \cdots + a_{p^e} 
\quad\text{with}\quad 
a_i \in k[[x,y]]$$ 
satisfying $\mathrm{ord}_P(a_i) > i$ for $i = 1, \ldots, p^e$. 
We also have a monomial 
$$(x^{\alpha}y^{\beta},a) \in {\mathcal R}$$ 
of the defining equations $x$ and $y$ of the components of the boundary divisor $E$ (actually $E_{\text{young}}$) such that every element $(f,\lambda) \in {\mathcal R}$ is divisible by $(x^{{\alpha}}y^{{\beta}})^{\lambda/a}$ 
modulo $h$, a consequence of the condition $\widetilde{\mu} = 0$. 
 
A naive idea for resolution of singularities in the monomial case may be stated as follows: Carry out the algorithm for resolution of singularities of $(x^{\alpha}y^{\beta},a)$ on the hypersurface $\{z = 0\}$. 
 
A bad news is that the above naive idea does not work for the reason that $(z,1) \not\in {\mathcal R}$ in general, which has the following bad implications: 
 
\medskip 
 
\indent{\rm (1)}\quad Even though we see that the coefficients $a_i$ 
for $i = 1, \ldots, p^e - 1$ are under control (in the sense that $a_i$ is divisible by $(x^{{\alpha}}y^{{\beta}})^{i/a}$), the constant term $a_{p^e}$ is not well controlled.   (The idealistic filtration ${\mathcal R}$ 
in the monomial case is \emph{not} differentially saturated but 
only 
saturated for $\{{\partial^n}/{\partial z^n} \mid n \in {\mathbb Z}_{\geq 0}\}$ 
in general.  However, this is good enough to control 
the coefficients $a_i$ for $i = 1, \ldots, p^e - 1$.) 
This leads to the calamity that a candidate for the center determined by the naive idea may not even be contained in the singular locus. 
 
\indent{\rm (2)}\quad The hypersurface $\{z = 0\}$ may not be 
of maximal contact.  That is to say, after a blow up, its strict transform may no longer ``contact'' (contain) the singular locus at all. 
 
\medskip 
 
In \S 5.1, we introduce the process of  ``cleaning'' in order to eliminate the ``mess'' described in the bad news above.  The idea of ``cleaning'' can be seen already in the work of Abhyankar and Hironaka, in the definition of the ``residual order''.  Here we follow the process refined by Benito-Villamayor.  After cleaning, the invariant 
$$\mathrm{H}(P) := \min\left\{\frac1{p^e}{\mathrm{ord}_P(a_{p^e})}, 
\ \mu(P) = \frac{\alpha + \beta}{a}\right\}$$ 
is independent of the choice of $h$ or a regular system of parameters $(x,y,z)$.  The invariant $\mathrm{H}$ is well-defined not only at a closed point $P \in W$ but also at the generic point $\xi_{\{x = 0\}}$ (resp. $\xi_{\{y = 0\}}$) of a component $\{x = 0\}$ (resp. $\{y = 0\}$) of the boundary divisor $E$. 
 
In \S 5.2, we give the description of the procedure of resolution of singularities, depending on the description of the singular locus.  Since we are in the monomial case, we have $\mathrm{Sing}({\mathcal R}) \subset \{x = 0\} \cup \{y = 0\}$.  The description of the singular locus restricted to $\{x = 0\}$, locally around $P$, is given as follows according to the value of the invariant $\mathrm{H}$: 
$$\mathrm{Sing}({\mathcal R}) \cap \{x = 0\} = 
\begin{cases} 
\{z = x = 0\} &\text{if}\quad \mathrm{H}(\xi_{\{x = 0\}}) \geq 1 \\ 
P &\text{if}\quad \mathrm{H}(\xi_{\{x = 0\}}) < 1. 
\end{cases} 
$$ 
That is to say, by looking at the invariant $\mathrm{H}$, we can tell if the singular locus has dimension 1, where it is a nonsingular curve, or has dimension zero, where it is an isolated point.  We have a similar description of the singular locus restricted to $\{y = 0\}$. 
 
The procedure goes as follows: 
 
\medskip 
 
\indent{\rm Step 1.}\quad Check if $\dim \mathrm{Sing}({\mathcal R}) = 1$. 
If yes, then blow up the 1-dimension\-al components one by one.  Since the invariant $\mathrm{H}$ strictly deceases for the component of the boundary divisor involved in the blow up, this step comes to an end after finitely many times with the dimension of the singular locus dropping to $0$. 
 
\indent{\rm Step 2.}\quad Once $\dim \mathrm{Sing}({\mathcal R}) = 0$, 
blow up the isolated points in the singular locus. 
 
\indent{\rm Step 3.}\quad Go back to Step 1. 
 
\indent Repeat these steps. 
 
 \medskip 
 
\S 5.3 is devoted to showing termination of the procedure described in \S 5.2.  We closely follow the beautiful and delicate argument recently 
developed by Benito-Villamayor, which analyzes the behavior of the 
monomial, the invariant $\mathrm{H}$, and the newton polygon of the 
constant term $a_{p^e}$ under blow ups.  Their argument is 
an extension of the classical ideas of Abhyankar and Hironaka, 
but highly refined taking into consideration the condition 
that we are in the monomial case.  Benito-Villamayor also uses 
``a proof by contradiction'' in some part of their argument for 
termination of the procedure.  That is to say,  they derive 
a contradiction, assuming the existence of an infinite sequence of 
the procedure.  Therefore, their argument is not effective. 
They also use some ``stratification'' of the configuration of 
the boundary divisors.  We introduce a new and explicit invariant, 
which makes the termination argument effective and which 
allows us to eliminate the use of ``stratification'' from our argument. 
 
This completes the overview of \S 5, and hence the overview 
of the entire paper. 
\end{section} 
\begin{section}{A quick review on the algorithm in characteristic zero} 
The goal of this section is to give a quick review on the algorithm 
in characteristic zero, upon which our algorithm in positive characteristic 
is closely modeled.  For the overview of this section, 
we refer the reader to \S 2. 
\begin{subsection}{Reformulation of the problem by Hironaka} 
First, we present the reformulation of the problem by Hironaka.  The form of the presentation we use here is due to Villamayor. 
\begin{problem}[Hironaka's reformulation] Suppose we are given the triplet of data $(W,({\mathcal I},a),E)$, where 
$W$ is a nonsingular variety over $k$, 
$({\mathcal I},a)$ is a pair consisting of a coherent ideal sheaf 
$\mathcal I$ on $W$ and a positive integer $a \in {\mathbb Z}_{> 0}$, 
and $E$  is a simple normal crossing divisor on $W$. 
 
We define its singular locus to be 
$$\mathrm{Sing}({\mathcal I},a) := \{P \in W \mid \mathrm{ord}_P({\mathcal I}) \geq a\}.$$ 
We only consider a \emph{nonzero} ideal sheaf ${\mathcal I}$ 
for the resolution problem. 
 
Then construct a sequence of transformations 
starting from $(W_0,({\mathcal I}_0,a),E_0)=(W,({\mathcal I},a),E)$ 
\begin{align*} 
(W_0,({\mathcal I}_0,a),E_0) & \leftarrow \!\cdots\! \leftarrow 
(W_i,({\mathcal I}_i,a),E_i)  \overset{\pi_{i+1}}\leftarrow 
(W_{i+1},({\mathcal I}_{i+1},a),E_{i+1}) \\ 
& \leftarrow \!\cdots\! \leftarrow (W_l,({\mathcal I}_l,a),E_l) 
\end{align*} 
such that $\mathrm{Sing}({\mathcal I}_l,a) = \emptyset$. 
We call such a sequence ``resolution of singularities for $(W,({\mathcal I},a),E)$''. 
 
We note that the transformation 
$$(W_i,({\mathcal I}_i,a),E_i)  \overset{\pi_{i+1}}\leftarrow (W_{i+1},({\mathcal I}_{i+1},a),E_{i+1})$$ 
is required to satisfy the following conditions: 
\begin{enumerate} 
\item $W_i \overset{\pi_{i+1}}{\leftarrow} W_{i+1}$ is a blow up 
with smooth center $C_i \subset W_i$, 
\item $C_i \subset \mathrm{Sing}({\mathcal I}_i,a)$, and $C_i$ is transversal to $E_i$ 
(maybe contained in $E_i$), i.e., $C_i \pitchfork E_i$, 
\item ${\mathcal I}_{i+1} = {\mathcal I}(\pi_{i+1}^{-1}(C_i))^{- a} \cdot \pi_{i+1}^{-1}({\mathcal I}_i){\mathcal O}_{W_{i+1}}$, 
\item $E_{i+1} = E_i \cup \pi_{i+1}^{-1}(C_i)$. 
\end{enumerate} 
\end{problem} 
\begin{lemma} A solution to Hironaka's reformulation provides a solution to the problem of embedded resolution of singularities. 
\end{lemma} 
\begin{proof} 
Given $X \subset W$, set $(W, ({\mathcal I},a), E) = (W, ({\mathcal I}_X,1), \emptyset)$.  Then $\mathrm{Sing}({\mathcal I}_0,a) = \mathrm{Sing}({\mathcal I}_X,1) = X = X_0$.  Take resolution of singularities for $(W, ({\mathcal I},a), E)$.  Observe that, if $X_i$, the strict transform of $X$ in year $i$, is an irreducible component of $\mathrm{Sing}({\mathcal I}_i,a)$ and if $X_i$ is not an irreducible component of the center $C_i$ in year $i$ (and $X_j$ has not been an irreducible component of $C_j$ in year $j$ for $0 \leq j < i$), then $X_{i+1}$ is an irreducible component of $\mathrm{Sing}({\mathcal I}_{i+1},a)$ in year $(i+1)$.  Since $\mathrm{Sing}({\mathcal I}_l,a) = \emptyset$, we conclude that $X_m$ must be an irreducible component of $C_m$ for some $m < l$.  Since the center $C_m$ is nonsingular by requirement, so is $X_m$.  Therefore, the truncation of the sequence up to year $m$ provides a sequence for embedded resolution.  (The other requirements for embedded resolution as stated in Problem 2 follow automatically from the construction.) 
\end{proof} 
The inductive scheme for solving Hironaka's reformulation can be simply stated in the following naive form. 
 
\medskip 
 
\textbf{Naive Inductive Scheme}: \emph{ Given a triplet 
$(W,({\mathcal I},a), E)$, find another triplet $(H,({\mathcal J},b),F)$ with $H \underset{\text{closed}}\subset W$ 
and $\dim H = \dim W - 1$ such that the problem of constructing 
resolution of singularities for $(W,({\mathcal I},a), E)$ is equivalent to constructing one for $(H,({\mathcal J},b),F)$, i.e., } 
$$(W, ({\mathcal I},a), E) \underset{\mathrm{equivalent}}\sim (H,({\mathcal J},b),F).$$ 
 
As it is, the above inductive scheme is too naive to hold in general.  In \S 3.2, we first state the key inductive lemma, 
which realizes the naive inductive scheme under a certain extra condition called $(\star)$, 
and then discuss how to turn it into the real inductive scheme, which works in general without the extra condition. 
\end{subsection} 
\begin{subsection}{Inductive scheme on dimension} 
\begin{lemma}[Key Inductive Lemma] Given $(W,({\mathcal I},a),E)$ and a closed point $P \in \mathrm{Sing}({\mathcal I},a)$, suppose that the following condition $(\star)$ is satisfied: 
$$(\star) \quad 
\left\{\begin{array}{ccc} 
\mathrm{ord}_P({\mathcal I}) &=& a \\ 
E &=& \emptyset. 
\end{array}\right.$$ 
Then there exists $(H,({\mathcal J},b),F)$ with 
$H \underset{\text{closed}}\subset W$ and 
$\dim H = \dim W - 1$, in a neighborhood of $P$, 
which satisfies one of the following two. 
 
\medskip 
 
\indent{\rm (i)}\quad 
The ideal ${\mathcal J}$ is a zero sheaf, i.e., ${\mathcal J} \equiv 0$:  In this case, we take the transformation with center $H =  \mathrm{Sing}({\mathcal I},a)$ 
$$(W, ({\mathcal I},a), E) \leftarrow (\widetilde{W}, (\widetilde{\mathcal I},a), \widetilde{E}).$$ 
After the transformation, we have $\mathrm{Sing}(\widetilde{\mathcal I},a) = \emptyset$ and hence we achieve resolution of singularities for $(W, ({\mathcal I},a), E)$ (in a neighborhood of $P$). 
 
\indent{\rm (ii)}\quad 
The ideal ${\mathcal J}$ is \emph{not} a zero sheaf, 
i.e., ${\mathcal J} \not\equiv 0$:  In this case, we have 
$$(W, ({\mathcal I},a), E) \underset{\mathrm{equivalent}}\sim (H,({\mathcal J},b),F).$$ 
Therefore, constructing resolution of singularities for 
$(H,({\mathcal J},$ $b),F)$ by induction on dimension, we achieve 
 resolution of singularities for 
$(W, ({\mathcal I},a), E)$ (in a neighborhood of $P$). 
\end{lemma} 
\begin{proof} Take $f \in {\mathcal I}_P$ such that $\mathrm{ord}_P(f) = a$.  Then there exists a differential operator $\delta$ of $\deg \delta = a - 1$ such that $\mathrm{ord}_P(\delta f) = 1$.  We note that this is exactly the place where we use the ``characteristic zero'' condition. 
 
We have only to set 
$$\left\{\begin{array}{ccl} 
H &=& \{h = 0\} \quad\text{where}\quad h = \delta f, \\ 
({\mathcal J},b) &=& (\mathrm{Coeff}({\mathcal I})|_H,a!), \\ 
F &=& E|_H = \emptyset, 
\end{array}\right.$$ 
where the ``coefficient ideal'', denoted by $\mathrm{Coeff}({\mathcal I})$, is defined by the formula 
$$\mathrm{Coeff}({\mathcal I}) := 
\sum_{j = 0}^{a - 1} \{Di\!f\!f^j({\mathcal I})\}^{{a !}/({a - j})}$$ 
with $Di\!f\!f^j({\mathcal I})$ being the sheaf characterized at the stalk 
for a point $Q \in W$ by 
$$Di\!f\!f^j({\mathcal I})_Q = 
\{\theta(g) \mid g \in {\mathcal I}_Q,\ \theta 
\text{ : a differential operator of }\deg(\theta) \leq j\}.$$ 
We note that the condition $E = \emptyset$ is only used to guarantee 
that $H$ intersects $E$ transversally.  The condition on 
the boundary divisor can be weakened and $E$ can be non-empty, 
as long as we can find 
such $H = \{h = 0\}$, with $h \in Di\!f\!f^{a-1}({\mathcal I})_P$ and 
$\text{ord}_P(h) = 1$, that is transversal to $E$. 
We then call the condition $(\star)_{\text{weakened}}$. 
\end{proof} 
\begin{remark} We remark that the statement of the key inductive lemma 
fails to hold in positive characteristic, as the following 
example by R. Narasimhan shows: 
 
Consider $(W,({\mathcal I},a),E)$ defined by 
$$\left\{\begin{array}{ccl} 
W &=& {\mathbb A}^4 = \mathrm{Spec}\ k[x,y,z,w] \quad\text{with}\quad 
\mathrm{char}(k) = 2, \\ 
({\mathcal I},a) &=& ((f),2) \quad \text{with}\quad 
f = w^2 + x^3y + y^3z + z^7x,\\ 
E &=& \emptyset. 
\end{array}\right. 
$$ 
Then we have a curve $C$, parametrized by $t$, sitting inside of the singular locus 
$$C = \{x = t^{15}, y = t^{19}, z = t^7, w = t^{32}\} \subset \mathrm{Sing}({\mathcal I},a).$$ 
Observe that the curve $C$ has full embedding dimension at the origin ${\mathbb O}$, i.e., 
$$\mathrm{embedding\text{-}}\dim_{\mathbb O}C = 4 = \dim W.$$ 
Therefore, there exists \emph{no} smooth hypersurface $H$ 
which contains the singular locus $\mathrm{Sing}({\mathcal I},a)$, and 
hence there exists \emph{no} smooth hypersurface of maximal contact. 
\end{remark} 
 
We list the shortcomings of the key inductive lemma toward establishing the algorithm for resolution of singularities for $(W,({\mathcal I},a),E)$ in general. 
 
\medskip 
 
\textbf{List of shortcomings of the key inductive lemma} 
 
\medskip 
 
\indent{\rm \textcircled{\footnotesize 1}}\quad We have to impose 
condition $(\star)$ on $(W,({\mathcal I},a),E)$. 
 
\indent{\rm \textcircled{\footnotesize 2}}\quad We construct 
$(H,({\mathcal J},b),F)$ only locally, and hence the resolution process is only locally constructed by induction. 
 
\indent{\rm \textcircled{\footnotesize 3}}\quad The invariant 
``$\mathrm{ord}$'' may strictly increase under transformations, even though our ultimate goal is to reduce the invariant ``$\mathrm{ord}$'' to be below the fixed level $a$. 
 
\medskip 
 
Now we describe the mechanism which overcomes all of the shortcomings above in one stroke, turning the key inductive lemma into the real inductive structure in characteristic zero. 
 
\medskip 
 
\textbf{Mechanism to overcome the shortcomings in the list} 
 
\medskip 
 
\noindent\underline{\rm Outline of the mechanism} 
 
\medskip 
 
\indent{\rm (1)}\quad Given $(W,({\mathcal I},a),E)$ (Precisely speaking, 
the triplet sits in the middle of the sequence, say in year ``$i$'', 
for resolution of singularities.  However, we omit 
the subscript ``$(\ )_i$'' indicating the year for simplicity 
of the notation.), we introduce a pair of invariants 
$(\mathrm{w\text{-}ord},s)$ and its associated triplet of data 
$(W,({\mathcal K},\kappa),G)$, called the modification of 
the original triplet, having the following properties. 
 
\medskip 
 
\indent$\bullet$\quad The maximum locus of the pair 
$(\mathrm{w\text{-}ord},s)$, which is an upper semi-continuous function, coincides with the singular locus $\mathrm{Sing}({\mathcal K},\kappa)$ of the modification, i.e., 
$$\mathrm{MaxLocus}(\mathrm{w\text{-}ord},s) = \mathrm{Sing}({\mathcal K},\kappa).$$ 
Moreover, after each transformation, the value of the pair 
$(\mathrm{w\text{-}ord},s)$ never increases, and the locus where the value of the pair takes the same maximum value as the original one coincides with the singular locus of the transformation of the modification.  (Note that the transformations of $(W,({\mathcal K},\kappa),G)$ induce those of $(W,({\mathcal I},a),E)$.)  This means that resolution of singularities for $(W,({\mathcal K},\kappa),G)$ implies the strict decrease of the (maximum) value of the pair $(\mathrm{w\text{-}ord},s)$. 
 
\indent$\bullet$\quad Even though the original triplet 
$(W,({\mathcal I},a),E)$ may not satisfy condition $(\star)$ (or condition $(\star)_{\mathrm{weakened}}$), the modification $(W,({\mathcal K},\kappa),G)$ satisfies $(\star)_{\mathrm{weakened}}$. 
 
\medskip 
 
\indent{\rm (2)}\quad We apply the key inductive lemma to 
$(W,({\mathcal K},\kappa),G)$.  We find a triplet $(H,({\mathcal J},b),F)$ with $H \underset{\text{closed}}\subset W$ 
and $\dim H = \dim W - 1$ such that we are either in Case (i), 
or in Case (ii) where resolution of singularities for $(H,({\mathcal J},b),F)$ implies the one for $(W,({\mathcal K},\kappa),G)$.  (We note that $G = E_{\mathrm{new}}$ may not be empty.  However, we can still find a hypersurface of maximal contact $H$ which is transversal to $G = E_{\mathrm{new}}$.  Therefore, the triplet  $(H,({\mathcal J},b),F)$ with $F = G|_H$ works.) 
 
\indent{\rm (3)}\quad In Case (i), by a simple blow up with center $H$, we 
accomplish resolution of singularities for $(W,({\mathcal K},\kappa),G)$.  In Case (ii), by induction on dimension, we achieve resolution of singularities for $(H,({\mathcal J},b),F)$, hence for $(W,({\mathcal K},\kappa),G)$.  In both cases, we have the strict decrease of the (maximum) value of the pair $(\mathrm{w\text{-}ord},s)$. 
 
\indent{\rm (4)}\quad By repeating the above procedures (1), (2), (3) and 
decreasing the value of the pair $(\mathrm{w\text{-}ord},s)$, we reach the stage where $\mathrm{w\text{-}ord} = 0$, i.e., $\overline{\mathcal I} = {\mathcal O}_W$.  This means that we are in the monomial case, where the ideal ${\mathcal I}$ is generated by some monomial of the defining equations of the components in $E_{\mathrm{young}} \subset E$. 
 
\indent{\rm (5)}\quad Finally, we have only to construct 
resolution of singularities for the triplet in the monomial case, which can be done easily in characteristic zero. 
 
\medskip 
 
We give the detailed and explicit description of the pair $(\mathrm{w\text{-}ord},s)$, the triplets $(W,({\mathcal K},\kappa),G)$ and $(H,({\mathcal J},b),F)$ in the following: 
 
\medskip 
 
\noindent\underline{\rm Description of the pair ($\mathrm{w\text{-}ord},s)$} 
 
\medskip 
 
\indent$\circ$\quad 
$\mathrm{w\text{-}ord}$: It is the so-called 
(normalized) weak order.  It is the order of $\overline{{\mathcal I}}$ (divided by the level $a$), where $\overline{{\mathcal I}}$ is obtained from ${\mathcal I}$ by dividing it as much as possible by the defining equations of the components in $E_{\mathrm{young}}$.  The symbol $E_{\mathrm{young}}$ refers to the union of the exceptional divisors created after the process of resolution of singularities began. 
 
\smallskip 
 
\indent$\circ$\quad 
$s$: It is the number of the components 
in $E_{\mathrm{old}} = E \setminus E_{\mathrm{new}}$.  The symbol $E_{\mathrm{new}}$ refers to the union of the exceptional divisors created after the time when the current value of $\mathrm{w\text{-}ord}$ first started. 
 
\medskip 
 
\noindent\underline{\rm Description of the triplet $(W,({\mathcal K},\kappa),G)$} 
 
$$ 
W = W, \quad 
({\mathcal K},\kappa) = \mathrm{Bdry} 
\left(\mathrm{Comp}({\mathcal I},a)\right), 
\quad\text{and}\quad 
G = E \setminus E_{\mathrm{old}} = E_{\mathrm{new}}, 
$$ 
where 
 
\indent$\bullet$\quad 
$\mathrm{Comp}({\mathcal I},a)$ is either 
the transformation of the one in the previous year 
if $\mathrm{w\text{-}ord}$ stays the same, or 
$$({\mathcal I}^M + \overline{\mathcal I}^a, M \cdot a) 
\quad\text{with}\quad M =\mathrm{w\text{-}ord} \cdot a 
$$ 
if $\mathrm{w\text{-}ord}$ strictly decreases, 
and where 
 
\indent$\bullet$\quad 
$\mathrm{Bdry}\left(\mathrm{Comp}({\mathcal I},a)\right)$ 
is either 
the transformation of the one in the previous year 
if $(\mathrm{w\text{-}ord},s)$ stays the same, or 
$$({\mathcal C} + 
(\sum_{D \subset E_{\mathrm{old}}}{\mathcal I}(D))^c, c) 
\quad\text{where}\quad 
({\mathcal C},c) = \mathrm{Comp}({\mathcal I},a),$$ 
if $(\mathrm{w\text{-}ord},s)$ strictly decreases. 
 
\medskip 
 
We note that we have $E_{\mathrm{new}} \subset E_{\mathrm{young}}$ and that they may not be equal in general.  We also note that, if the value of the pair $(\mathrm{w\text{-}ord},s)$ stays the same as in the previous year, then $({\mathcal K},\kappa)$ is the transformation of the one in the previous year.  We remark that the symbols ``$\mathrm{Comp}$'' and ``$\mathrm{Bdry}$'' represent the ``Companion'' modification and the ``Boundary'' modification, respectively. 
 
\medskip 
 
\noindent\underline{\rm Description of the triplet $(H,({\mathcal J},b),F)$ } 
 
\medskip 
 
\noindent\emph{Case$\colon 
$ The value of the pair $(\mathrm{w\text{-}ord},s)$ 
stays the same as in the previous year}. 
 
In this case, we simply take $(H,({\mathcal J},b),F)$ to be 
the transformation of the one in the previous year under blow up. 
 
\medskip 
 
\noindent\emph{Case$\colon 
$ The value of the pair $(\mathrm{w\text{-}ord},s)$ 
strictly decreases}. 
 
In this case, we construct $(H,({\mathcal J},b),F)$ as follows. 
$$\left\{\begin{array}{ccl} 
H &=& \text{the\ strict\ transform\ of\ }H_{i_{\mathrm{old}}},\\ 
({\mathcal J},b) &=& (\text{Coeff}({\mathcal K})|_H,\kappa !),\\ 
F &=& G|_H, 
\end{array}\right. 
$$ 
where $H_{i_\mathrm{old}}$ is taken in the following way: 
We go back to the year $\iota := i_{\mathrm{old}}$ 
when the current value of $\mathrm{w\text{-}ord}$ first started. 
Let $({\mathcal C}_{\iota},c_{\iota}) 
= \mathrm{Comp}({\mathcal I}_{\iota},a)$ 
be the companion modification constructed in year $\iota$. 
We take $f_{\iota} \in \left({\mathcal C}_{\iota}\right)_{P_{\iota}}$ and a differential operator $\delta_{\iota}$ of $\deg \delta_{\iota} = c_{\iota} - 1$ such that $\mathrm{ord}_{P_{\iota}}(f_{\iota}) = c_{\iota}$ and $\mathrm{ord}_{P_{\iota}}(\delta_{\iota} f_{\iota}) = 1$. 
We set $H_{i_{\mathrm{old}}}= 
H_{\iota} = \{\delta_{\iota} f_{\iota} = 0\}$. 
 
\medskip 
 
This completes the discussion of the mechanism to achieve resolution of singularities for $(W,({\mathcal I},a),E)$ in general.  We note, however, that we only overcome shortcomings \textcircled{\footnotesize 1} and \textcircled{\footnotesize 3} on the list, since the description so far is only local.  We discuss in \S 3.5 how to overcome shortcoming \textcircled{\footnotesize 2} and how to globalize the procedure via the strand of invariants woven in the next section. 
\end{subsection} 
\begin{subsection}{Weaving of the classical strand of invariants 
``$\inv_{\mathrm{classic}}$''} 
In \S 3.3, we interpret the inductive scheme explained in \S 3.2 in terms of weaving the strand of invariants ``$\inv_{\mathrm{classic}}$'', whose maximum locus (with respect to the lexicographical order) determines the center of blow up for the algorithm for resolution of singularities in characteristic zero. 
 
We weave the strand of invariants ``$\inv_{\mathrm{classic}}$''
 consisting of the units of the form $(\dim H^j, \allowbreak 
\mathrm{w\text{-}ord}^j,s^j)$, 
computed from the modifications 
$(H^j,$ $({\mathcal J}^j,b^j),F^j)$ constructed simultaneously along the weaving process.  We note that we are adding ``$\dim H$'' to the pair $(\mathrm{w\text{-}ord},s)$ as the first factor of the unit, in order to emphasize the role of the dimension in the inductive scheme. 
 
\medskip 
 
\textbf{Weaving Process} 
 
\medskip 
 
We describe the weaving process inductively. 
 
Note that constructing a sequence for resolution of singularities by blowing up is referred to as ``proceeding in the vertical direction'' passing from one year to the next, indicated by the subscript ``$i$'', while weaving the strand and constructing the modifications passing from one stage to the next, indicated by the superscript ``$j$'', is referred to ``proceeding in the horizontal direction'' staying in a fixed year. 
 
\medskip 
 
Suppose we have already woven the strands and constructed the modifications up to year $(i-1)$. 
 
Now we are in year $i$ (looking at the neighborhood of a closed point 
$P_i \in W_i$). 
 
We start with $(W_i,({\mathcal I}_i,a),E_i) = (H_i^0,({\mathcal J}_i^0,b_i^0),F_i^0)$, just renaming the transformation $(W_i,({\mathcal I}_i,a),E_i)$ in year $i$ of the resolution sequence as the $0$-th stage modification $(H_i^0,({\mathcal J}_i^0,b_i^0),F_i^0)$ in year $i$. 
 
Suppose that we have already woven the strand up to the $(j-1)$-th unit 
\begin{align*} 
\left(\inv_{\mathrm{classic}}\right)_i^{\leq j-1} 
=\ & 
(\dim H_i^0, \mathrm{w\text{-}ord}_i^0,s_i^0) 
(\dim H_i^1, \mathrm{w\text{-}ord}_i^1,s_i^1) 
 \\ 
& 
\cdots 
(\dim H_i^{j-1}, \mathrm{w\text{-}ord}_i^{j-1},s_i^{j-1}) 
\end{align*} 
and that we have also constructed the modifications up to the $j$-th one 
$$(H_i^0,({\mathcal J}_i^0,b_i^0),F_i^0), 
\cdots, (H_i^{j-1},({\mathcal J}_i^{j-1},b_i^{j-1}),F_i^{j-1}), 
(H_i^j,({\mathcal J}_i^j,b_i^j),F_i^j).$$ 
 
Our task is to compute the $j$-th unit $(\dim H_i^j,\mathrm{w\text{-}ord}_i^j,s_i^j)$ and construct the $(j+1)$-th modification $(H_i^{j+1},({\mathcal J}_i^{j+1},b_i^{j+1}),F_i^{j+1})$ (unless the weaving process is over at the $j$-th stage). 
 
\medskip 
 
\noindent\underline{\rm Computation of the j-th unit 
$(\dim H_i^j,\mathrm{w\text{-}ord}_i^j,s_i^j)$} 
 
\medskip 
 
\indent$\circ\ \dim H_i^j$:\quad We just remark that we insert 
this first factor in characteristic zero only to emphasize the role of the dimension, which corresponds to the role of the invariant $\sigma$ in our algorithm in positive characteristic. 
 
\medskip 
 
\indent$\circ\ \mathrm{w\text{-}ord}_i^j$:\quad  We compute the second factor 
as follows. 
$$\mathrm{w\text{-}ord}_i^j = 
\begin{cases} 
\infty &\text{if}\quad {\mathcal J}_i^j \equiv 0, \\ 
\mathrm{ord}(\overline{{\mathcal J}_i^j}) 
/{b_i^j} 
&\text{if}\quad {\mathcal J}_i^j \not\equiv 0, 
\end{cases} 
$$ 
where the ideal $\overline{{\mathcal J}_i^j}$ is obtained from ${\mathcal J}_i^j$ by dividing the latter as much as possible by the defining ideals of the components in $(F_i^j)_{\mathrm{young}}$, i.e., 
$$\overline{{\mathcal J}_i^j} 
=  \left(\prod_{D \in (F_i^j)_{\mathrm{young}}}{\mathcal I}(D)^{- \mathrm{ord}_{\eta(D)}({\mathcal J}_i^j)}\right) \cdot {\mathcal J}_i^j$$ 
where $\eta(D)$ is the generic point of $D$ 
and where $(F_i^j)_{\mathrm{young}} (\subset F_i^j)$ is the union 
of the exceptional divisors created after the year when the value $\left(\inv_{\mathrm{classic}}\right)_i^{\leq j-1}(\dim H_i^j)$ first started. 
 
We note that, if $\mathrm{w\text{-}ord}_i^j = \infty$ or $0$, 
we declare that the $(j = m)$-th unit is the last one, and we stop the weaving process at the $m$-th stage in year $i$.  When $\mathrm{w\text{-}ord}_i^j = \infty$, 
the third factor is \emph{not} included. 
When $\mathrm{w\text{-}ord}_i^j = 0$, we are in the monomial case and we insert the invariant $\Gamma$ as the third factor instead of the invariant $s$. 
 
\medskip 
 
\indent$\circ\ s_i^j$:\quad It is the number of the components 
in $(F_i^j)_{\mathrm{old}} = F_i^j \setminus (F_i^j)_{\mathrm{new}}$, where $(F_i^j)_{\mathrm{new}} (\subset F_i^j)$ is the union of the exceptional divisors created after the year when the value $\left(\inv_{\mathrm{classic}}\right)_i^{\leq j-1}(\dim H_i^j, \mathrm{w\text{-}ord}_i^j)$ first started. 
We note that the third factor $s$ is \emph{only} included 
if $\mathrm{w\text{-}ord}_i^j \neq \infty$ or $0$. 
 
\medskip 
 
At the end, the weaving process of the strand comes to an end in a fixed year $i$, with $\left(\inv_{\mathrm{classic}}\right)_i$ taking the following form 
\begin{align*} 
\lefteqn{ 
\left({\inv_{\mathrm{classic}}}\right)_i = 
(\dim H_i^0,\mathrm{w\text{-}ord}_i^0,s_i^0) 
\cdots (\dim H_i^j, \mathrm{w\text{-}ord}_i^j,s_i^j) 
}\\ 
& \qquad 
\cdots (\dim H_i^{m-1}, \mathrm{w\text{-}ord}_i^{m-1},  s_i^{m-1}) 
\begin{cases} 
(\dim H_i^m, \mathrm{w\text{-}ord}_i^m = \infty), \ \text{or}\\ 
(\dim H_i^m, \mathrm{w\text{-}ord}_i^m = 0, \Gamma). 
\end{cases} 
\end{align*} 
 
\textbf{Termination in the horizontal direction} 
 
\medskip 
 
We note that termination of the weaving process in the horizontal direction is a consequence of the fact that going from the $j$-th unit to the $(j+1)$-th unit we have $\dim H_i^j > \dim H_i^{j+1}$ and that the dimension obviously satisfies the descending chain condition. 
 
\medskip 
 
\noindent\underline{\rm Construction of the (j+1)-th modification 
$(H_i^{j+1},({\mathcal J}_i^{j+1},b_i^{j+1}),F_i^{j+1})$} 
 
\medskip 
 
We note that we construct the $(j+1)$-th modification only when 
$\mathrm{w\text{-}ord}_i^j \neq \infty$ or $0$. 
 
\noindent\emph{Case$\colon 
\left({\inv_{\mathrm{classic}}}\right)_i^{\leq j} 
= \left({\inv_{\mathrm{classic}}}\right)_{i-1}^j$}. 
 
In this case, we simply take 
$(H_i^{j+1},({\mathcal J}_i^{j+1},b_i^{j+1}),F_i^{j+1})$ 
to be the transformation of 
$(H_{i-1}^{j+1},({\mathcal J}_{i-1}^{j+1}, b_{i-1}^{j+1}),F_{i-1}^{j+1})$ 
under the blow up. 
 
\noindent\emph{Case$\colon 
\left({\inv_{\mathrm{classic}}}\right)_i^{\leq j} 
< \left({\inv_{\mathrm{classic}}}\right)_{i-1}^j$}. 
 
In this case, we follow the construction described in the 
mechanism discussed in \S 3.2. 
 
Starting from $(H_i^j,({\mathcal J}_i^j,b_i^j),F_i^j)$, 
we firstly construct 
$(H_i^j,({\mathcal K}_i^j,\kappa_i^j), G_i^j)$, 
whose description is given below. 
$$ 
H_i^j 
= H_i^j, 
\ 
({\mathcal K}_i^j,\kappa_i^j) 
= \mathrm{Bdry}({\mathcal C}_i^j,c_i^j), \ \text{and} 
\ 
G_i^j = F_i^j \setminus (F_i^j)_{\text{old}} = (F_i^j)_{\text{new}}, 
$$ 
where 
$$ 
({\mathcal C}_i^j,c_i^j) = \mathrm{Comp}({\mathcal J}_i^j,b_i^j), 
\quad 
\mathrm{Bdry}({\mathcal C}_i^j,c_i^j)  = ({\mathcal C}_i^j 
+ (\sum_{D \subset (F_i^j)_{\mathrm{old}}}{\mathcal I}(D))^{c_i^j}, c_i^j), 
$$ 
and, denoting 
$\left({\inv_{\mathrm{classic}}}\right)_i^{\leq j-1}(\dim H_i^j, 
\mathrm{w\text{-}ord}_i^j)$ by $\alpha_i^j$, 
we set 
$\mathrm{Comp}({\mathcal J}_i^j, b_i^j)$ 
to be either the transformation of 
$\mathrm{Comp}({\mathcal J}_{i-1}^j,b_{i-1}^j)$ 
if $\alpha_i^j=\alpha_{i-1}^j$, or 
$$ 
({{\mathcal J}_i^j}^{M_i^j} + \overline{{\mathcal J}_i^j}^{b_i^j}, 
M_i^j \cdot b_i^j) 
\quad\text{with}\quad 
M_i^j =\mathrm{w\text{-}ord}_i^j \cdot b_i^j 
$$ 
if $\alpha_i^j<\alpha_{i-1}^j$. 
 
\medskip 
 
Then we secondly construct 
$(H_i^{j+1},({\mathcal J}_i^{j+1},b_i^{j+1}),F_i^{j+1})$ as follows. 
$$\left\{\begin{array}{ccl} 
H_i^{j+1} &=& \text{the\ strict\ transform\ of\ }H_{i_{\mathrm{old}}^j}^{j+1},\\ 
({\mathcal J}_i^{j+1},b_i^{j+1}) &=& (\text{Coeff}({\mathcal K}_i^j)|_{H_i^{j+1}},(\kappa_i^j)!) \\ 
F_i^{j+1} &=& G_i^j|_{H_i^{j+1}} = (F_i^j)_{\text{new}}|_{H_i^{j+1}}, 
\end{array}\right. 
$$ 
where $H_{i_{\mathrm{old}}^j}^{j+1}$ is taken in the following way: 
We go back to the year $\iota := i_{\mathrm{old}}^j$ 
when the current value of $\mathrm{inv}_i^{\leq j}(\dim H_i^j,\mathrm{w\text{-}ord}_i^j)$ first started. 
Let $({\mathcal C}_{\iota}^j,c_{\iota}^j) = \mathrm{Comp}({\mathcal I}_{\iota}^j,b_{\iota}^j)$ be the companion modification constructed in year $\iota$. 
We take $f_{\iota}^j \in 
\left({\mathcal C}_{\iota}^j\right)_{P_{\iota}}$ 
and a differential operator $\delta_{\iota}^j$ of $\deg \delta_{\iota}^j = c_{\iota}^j - 1$ such that $\mathrm{ord}_{P_{\iota}}(f_{\iota}^j) = c_{\iota}^j$ and $\mathrm{ord}_{P_{\iota}}(\delta_{\iota}^j f_{\iota}^j) = 1$.  We set 
$H_{i_{\mathrm{old}}^j}^{j+1} 
= H_{\iota}^{j+1} 
= \{\delta_{\iota}^j f_{\iota}^j = 0\}$. 
 
\medskip 
 
\textbf{Summary of the algorithm in $\operatorname{char}(k)\!=\!0$ 
in terms of ``$\inv_{\mathrm{classic}}$\!\!''} 
 
\medskip 
 
We start with $(W,({\mathcal I},a),E) = (W_0,({\mathcal I}_0,a),E_0)$. 
 
Suppose we have already constructed the resolution sequence up to year $i$ 
$$(W,({\mathcal I},a),E) = (W_0,({\mathcal I}_0,a),E_0) \leftarrow 
\cdots \leftarrow (W_i,({\mathcal I}_i,a),E_i).$$ 
We weave the strand of invariants in year $i$ described as above 
\begin{align*} 
\lefteqn{ 
\left({\inv_{\mathrm{classic}}}\right)_i 
= 
(\dim H_i^0,\mathrm{w\text{-}ord}_i^0,s_i^0) 
\cdots 
(\dim H_i^j, \mathrm{w\text{-}ord}_i^j,s_i^j) 
}\\ 
& 
\cdots 
(\dim H_i^{m-1}, \mathrm{w\text{-}ord}_i^{m-1},s_i^{m-1}) 
\begin{cases} 
(\dim H_i^m, \mathrm{w\text{-}ord}_i^m = \infty),\ \text{or}\\ 
(\dim H_i^m, \mathrm{w\text{-}ord}_i^m = 0, \Gamma). 
\end{cases} 
\end{align*} 
There are two cases according to the form of the last unit: 
 
\medskip 
 
\noindent\emph{Case$\colon 
\mathrm{w\text{-}ord}_i^m = \infty$}. 
 
In this case, we take the center of blow up in year $i$ for the transformation to be the last hypersurface of maximal contact $H_i^m$. 
 
\noindent\emph{Case$\colon 
\mathrm{w\text{-}ord}_i^m = 0$}. 
 
In this case, we follow the procedure specified for resolution of singularities in the monomial case in \S 3.4, and take the center of blow up in year $i$ for the transformation to be the maximum locus of the invariant $\Gamma$ on $H_i^m$. 
 
\medskip 
 
In both cases, the center of blow up coincides with the maximum locus of the strand $\left({\inv_{\mathrm{classic}}}\right)_i$.

\pagebreak[2] 
 
\textbf{Termination in the vertical direction} 
 
\medskip 
 
The value of the strand ``$\inv_{\mathrm{classic}}$'' never increases 
after the blow 
up described as above. 
By construction, the maximum locus of 
$(\inv_{\mathrm{classic}})^{\leq j}$ coincides with the singular locus of 
$(H^{j+1},({\mathcal J}^{j+1},b^{j+1}), 
F^{j+1})$. 
Therefore, resolution of singularities for 
$(H^{j+1},({\mathcal J}^{j+1},b^{j+1}), 
F^{j+1})$ 
implies the strict decrease of $(\inv_{\mathrm{classic}})^{\leq j}$. 
In particular, in the first 
case, $(\inv_{\mathrm{classic}})^{\leq m-1}$ strictly decreases. 
In the second case, either 
$(\inv_{\mathrm{classic}})^{\leq m-1}$ strictly 
decreases, or while $(\inv_{\mathrm{classic}})^{\leq m-1}$ may remain 
the same (and hence so does 
$(\inv_{\mathrm{classic}})^{\leq m-1} 
(\dim H_i^m,\mathrm{w\text{-}ord}_i^m = 0)$), 
the invariant $\Gamma$ strictly decreases. 
After all, we conclude that the value of the strand strictly decreases after each blow up, i.e., we have $\left({\inv_{\mathrm{classic}}}\right)_i > \left({\inv_{\mathrm{classic}}}\right)_{i+1}$. 
Now we claim that the value of the strand ``$\inv_{\mathrm{classic}}$'' can not decrease infinitely many times.  In fact, suppose by induction we have shown that the value of $\left({\inv_{\mathrm{classic}}}\right)^{\leq t-1}$ can not decrease infinitely many times.  Then after some year, the value of $\left({\inv_{\mathrm{classic}}}\right)^{\leq t-1}$ stabilizes.  This in turn implies that the value $b^t$, the second factor of the pair in the $t$-th modification, stays the same, say $b$. 
Now the value of $\mathrm{w\text{-}ord}^t$, having the fixed denominator $b$, 
can not decrease infinitely many times, and neither can the value $s^t$ being the nonnegative integer.  Therefore, we conclude that the value of $\left({\inv_{\mathrm{classic}}}\right)^{\leq t}$ can not decrease infinitely many times.  As the value of $t$ increases, the value of the dimension decreases by one.  Since obviously the value of the dimension satisfies the descending chain condition, the increase of the value of $t$ stops after finitely many times. Finally, therefore, we conclude that the value of the strand ``$\inv_{\mathrm{classic}}$'' can not decrease infinitely many times. 
 
Therefore, the algorithm terminates after finitely many years, 
achieving resolution of singularities for $(W,({\mathcal I},a),E)$. 
\end{subsection} 
\begin{subsection}{The monomial case in characteristic zero} 
The purpose of \S 3.4 is to discuss how to construct resolution of singularities for $(W,({\mathcal I},a),E)$ which is in the monomial case.  (Precisely speaking, the triplet sits in the middle of the sequence, say in year ``$i$'', for resolution of singularities.  However, we omit the subscript ``$(\ )_i$'' indicating the year for simplicity of the notation.) 
 
Recall that, in the monomial case, ${\mathcal I}$ is a monomial of the ideals defining the components of $E_{\text{young}} = \bigcup_{t = 1}^eD_t \subset E$ (See \S 3.3 for the definition of $E_{\text{young}}$.), i.e., 
$${\mathcal I} = \prod_{t = 1}^e{\mathcal I}(D_t)^{c_t} 
\quad\text{with}\quad c_t \in {\mathbb Z}_{\geq 0}.$$ 
 
\pagebreak[2] 
 
\textbf{Invariant ``$\Gamma$''} 
 
\begin{defn}[Invariant ``$\Gamma$''] Let the situation be as above.  We define the invariant $\Gamma = (\Gamma_1,\Gamma_2,\Gamma_3)$ at $P \in \text{Sing}({\mathcal I},a)$ in the following way: 
\begin{align*} 
\Gamma_1 &= - \min \{n \mid \exists\ t_1, \ldots\!, t_n 
\text{ s.t. } c_{t_1} + \cdots + c_{t_n} \geq a, 
P \in \!D_{t_1} \!\cap\!\cdots \cap D_{t_n}\}, \\ 
\Gamma_2 &=\phantom{-}\max \{(c_{t_1} + \cdots + c_{t_n})/a \mid 
n = -\Gamma_1, P \in D_{t_1} \cap \cdots \cap D_{t_n}\}, \\ 
\Gamma_3 &=\phantom{-}\max\left\{ 
(t_1, \ldots, t_n) \mid\begin{array}{l} 
n = -\Gamma_1, 
\ 
c_{t_1} + \cdots + c_{t_n} =a\Gamma_2, 
\\ 
P \in D_{t_1} \cap \cdots \cap D_{t_n}, 
\ 
t_1 < \cdots < t_n 
\end{array}\right\}. 
\end{align*} 
\end{defn} 
 
It is immediate to see the following properties of the invariant ``$\Gamma$''. 
 
\medskip 
 
\indent{\rm (1)}\quad The invariant $\Gamma$ is an upper semi-continuous function. 
 
\indent{\rm (2)}\quad The maximum locus $\text{MaxLocus}(\Gamma)$ is nonsingular, since it is the intersection of some components in $E_{\mathrm{young}} \subset E$, a simple normal crossing divisor. 
 
\medskip 
 
\textbf{Procedure and its termination} 
 
\medskip 
 
Now take the transformation with center $C = \text{MaxLocus}(\Gamma)$ 
$$(W,({\mathcal I},a),E) \overset{\pi}\leftarrow (W',({\mathcal I}',a),E')$$ 
where $E' = E \cup \pi^{-1}(C)$ and $E'_{\text{young}} = E_{\text{young}} \cup \pi^{-1}(C) = \bigcup_{t = 1}^{e+1}D_t$ with $D_{e+1} = \pi^{-1}(C)$.  Then it is easy to see that $(W',({\mathcal I}',a),E')$ is again in the monomial case and that the invariant $\Gamma$ strictly decreases, i.e., 
$$\Gamma > \Gamma'.$$ 
Since the value of $\Gamma$ can not decrease infinitely many times, this procedure must terminate after finitely many years, achieving resolution of singularities for $(W,({\mathcal I},a),E)$ in the monomial case. 
 
This completes the discussion on how to construct resolution of singularities for $(W,({\mathcal I},a),E)$ in the monomial case. 
\end{subsection} 
\begin{subsection}{Globalization} 
The strand  ``$\inv_{\mathrm{classic}}$'' a priori depends on the choice of the hypersurfaces of maximal contact we take in the process of weaving, and it is a priori only locally defined.  However, the strand ``$\inv_{\mathrm{classic}}$''  is actually independent of the choice, and hence it is globally well-defined.  This can be shown classically by the so-called Hironaka's trick, or more recently by incorporating W{\l}odarczyk's ``homogenization'' (cf.\cite{W}) or the first author's ``differential saturation'' (cf.\cite{K}) in the construction of the modification.  Therefore, the process of resolution of singularities, where we take the center of blow up to be the maximum locus of  ``$\inv_{\mathrm{classic}}$'', is also globally well-defined.  This is how we overcome shortcoming (2) of the key inductive lemma. 
\end{subsection} 
 
\medskip 
 
This finishes the quick review on the classical algorithm in characteristic zero. 
\end{section} 
\begin{section}{Our algorithm in positive characteristic} 
The goal of this section is to discuss the general mechanism of our algorithm in positive characteristic, which is modeled closely upon the classical algorithm in characteristic zero reviewed in \S 3. 
\begin{subsection}{Reformulation of the problem in our setting} 
First, we present the reformulation of the problem in our setting. 
\begin{problem}[Reformulation in terms of 
an idealistic filtration (cf. \cite{K} \cite{KM})] 
Suppose we are given the triplet of data $(W,{\mathcal R},E)$ 
such that 
$W$ is a nonsingular variety over $k$, 
${\mathcal R}  =  \bigoplus_{a \in {\mathbb Z}_{\geq 0}}({\mathcal I}_a,a)$ 
is an idealistic filtration of i.f.g.\! type (resp. of i.g. type), 
i.e., a finitely generated graded ${\mathcal O}_W\text{-algebra}$ 
(resp. a graded ${\mathcal O}_W$-algebra) 
satisfying the condition 
${\mathcal O}_W = {\mathcal I}_0 \supset {\mathcal I}_1 \supset 
{\mathcal I}_2 \cdots  \supset {\mathcal I}_a \supset \cdots$, 
where ``$a$'' in the second factor specifies 
the ``level'' of the ideal ${\mathcal I}_a$ in the first factor, 
and $E$ is a simple normal crossing divisor on $W$. 
 
We define its singular locus to be 
$$\mathrm{Sing}({\mathcal R}) := \{P \in W \mid 
\mathrm{ord}_P({\mathcal I}_a) \geq a\  \quad 
\forall a \in {\mathbb Z}_{\geq 0}\}.$$ 
(We note that we only consider ${\mathcal R}$ with 
${\mathcal I}_1 \neq 0$ for the resolution problem.) 
 
Then construct a sequence of transformations 
starting from 
$(W_0,{\mathcal R}_0, E_0) =(W,{\mathcal R},E)$ 
\begin{align*} 
(W_0,{\mathcal R}_0,E_0) & \leftarrow \!\cdots\! \leftarrow 
(W_i,{\mathcal R}_i,E_i)  \overset{\pi_{i+1}}\leftarrow 
(W_{i+1},{\mathcal R}_{i+1},E_{i+1}) \\ 
& \leftarrow \!\cdots\! \leftarrow (W_l,{\mathcal R}_l,E_l) 
\end{align*} 
such that $\mathrm{Sing}({\mathcal R}_l) = \emptyset$. 
 
We call such a sequence ``resolution of singularities for $(W,{\mathcal R},E)$''. 
 
We note that the transformation 
$$(W_i,{\mathcal R}_i,E_i)  \overset{\pi_{i+1}}\leftarrow (W_{i+1},{\mathcal R}_{i+1},E_{i+1})$$ 
is required to satisfy the following conditions: 
 
\medskip 
\indent{\rm (1)}\quad 
$W_i \overset{\pi_{i+1}}{\leftarrow} W_{i+1}$ is a blow up with smooth center $C_i \subset W_i$, 
 
\indent{\rm (2)}\quad 
$C_i \subset \mathrm{Sing}({\mathcal R}_i)$, and $C_i$ is transversal 
to $E_i$ (maybe contained in $E_i$), i.e., $C \pitchfork E_i$, 
 
\indent{\rm (3)}\quad 
${\mathcal R}_{i+1} = {\mathcal G}(\bigcup_{a \in {\mathbb Z}_{> 0}}({\mathcal J}_{a,i+1},a))$, where 
$${\mathcal J}_{a, i+1} = {\mathcal I}(\pi_{i+1}^{-1}(C_i))^{- a} \cdot \pi_{i+1}^{-1}({\mathcal I}_{a,i}){\mathcal O}_{W_{i+1}} \quad\text{for}\quad a \in {\mathbb Z}_{\geq 0},$$ 
i.e., ${\mathcal R}_{i+1}$ is the smallest idealistic filtration of i.f.g.\! type containing $({\mathcal J}_{a,i+1},a)$ for all $a \in {\mathbb Z}_{\geq 0}$ 
(We note that ${\mathcal I}_{a,i+1} \supset {\mathcal J}_{a,i+1}$ but 
they may not be equal in general.), 
 
\indent{\rm (4)}\quad 
$E_{i+1} = E_i \cup \pi_{i+1}^{-1}(C_i)$. 
\end{problem} 
\begin{remark}[Local version] Problem 4 is the``global'' version of the problem of resolution of singularities for the triplet of data $(W,{\mathcal R},E)$.  In the following, we formulate its local version: Starting from a closed point $P \in \mathrm{Sing}({\mathcal R}) \subset W$ and its neighborhood, we have a sequence of closed  points and their neighborhoods 
$$         P_0 \in \mathrm{Sing}({\mathcal R}_0) \!\subset\! W_0 
\leftarrow P_1 \in \mathrm{Sing}({\mathcal R}_1) \!\subset\! W_1 
\leftarrow \!\cdots\! 
\leftarrow P_i \in \mathrm{Sing}({\mathcal R}_i) \!\subset\! W_i 
$$ 
in the resolution sequence, 
where $W=W_0$, ${\mathcal R}={\mathcal R}_0$ and $P=P_0$. 
After we choose the center $P_i \in C_i \subset \mathrm{Sing}({\mathcal R}_i) \subset W_i$ and take the transformation $W_i \overset{\pi_{i+1}}\leftarrow W_{i+1}$ to extend the resolution sequence, the ``devil'' tries to choose a closed point $P_{i+1} \in \pi_i^{-1}(P_i) \cap  \mathrm{Sing}({\mathcal R}_{i+1}) \subset W_i$.  If $\pi_i^{-1}(P_i) \cap  \mathrm{Sing}({\mathcal R}_{i+1}) = \emptyset$, then the devil loses.  Our task is to provide a prescription on how to choose the center so that, no matter how the devil makes his choice, he will end up losing.  That is to say, the prescription should guarantee that we ultimately reach year $i = l-1$ so that, with the choice of center $C_{l-1}$, after the blow up we have $\pi_l^{-1}(P_{l-1}) \cap \mathrm{Sing}({\mathcal R}_l) = \emptyset$. 
 
\medskip 
 
\textbf{Our algorithm discussed in this paper is exclusively for this local version} of the problem of resolution of singularities for the triplet of data $(W,{\mathcal R},E)$.  The adjustments we have to make to our algorithm in order to solve the global version of the problem will be published elsewhere. 
\end{remark} 
 
The notion of ``the differential saturation'' ${\mathcal D}{\mathcal R}$ (of an idealistic filtration ${\mathcal R}$ of i.f.g.\! type) plays an important role in our algorithm. 
\begin{defn}\label{DefDR} 
Let ${\mathcal R}$ be an idealistic filtration of i.f.g.\! type.  We define its differential saturation ${\mathcal D}{\mathcal R}$ (at the level of the stalk for a point  $P \in W$) as follows: 
$${\mathcal D}{\mathcal R}_P := {\mathcal G}\left(\left\{\left(\delta f, \max\{a - \deg \delta, 0\}\right) \mid (f,a) \in {\mathcal R}_P, 
\ \delta\colon\text{a diff. op.}\right\}\right), 
$$ 
where the symbol ${\mathcal G}(S)$ denotes ``an idealistic filtration of i.g. type generated by the set $S$'', i.e., the smallest idealistic filtration of i.g. type containing $S$. 
\end{defn} 
\begin{remark} \ 
 
\indent{\rm (1)}\quad 
It follows immediately from the generalized product rule (cf. \cite{K}) 
that ${\mathcal D}{\mathcal R}$ is again an idealistic filtration of 
i.f.g.\! type, ${\mathcal R} \subset {\mathcal D}{\mathcal R}$, and 
that it is differentially saturated, i.e., 
${\mathcal D}({\mathcal D}{\mathcal R}) = {\mathcal D}{\mathcal R}$. 
 
\indent{\rm (2)}\quad The problem of constructing resolution of 
singularities for 
$(W, {\mathcal R}, E)$ is equivalent to the one for 
$(W,{\mathcal D}{\mathcal R},E)$, i.e., 
$$(W,{\mathcal R},E) \underset{\mathrm{equivalent\ to}} 
\sim (W,{\mathcal D}{\mathcal R},E).$$ 
\end{remark} 
\end{subsection} 
\begin{subsection}{Inductive scheme on the invariant ``$\sigma$''} 
The classical algorithm in characteristic zero works 
by induction on dimension, based upon the notion of a \emph{smooth} 
hypersurface of maximal contact, as reviewed in \S 3.  Narasimhan's example (cf. Remark 1 in \S 3) tells us, however, that there is no hope of 
finding a \emph{smooth} hypersurface of maximal contact in positive 
characteristic.  The following proposition gives rise to the notion of ``a leading generator system'' (called an LGS for short), which we consider as a collective substitute in positive characteristic for the notion of a hypersurface of maximal contact (called an HMC for short) in characteristic zero.  Our algorithm in positive characteristic works 
by induction on the invariant ``$\sigma$'', based upon 
the notion of an LGS.  Roughly speaking, introducing the notion of an LGS corresponds 
to considering \emph{singular} hypersurfaces of maximal contact. 
 
\bigskip 
 
\textbf{Definition of the invariant ``$\sigma$''} 
\begin{proposition}[cf. \cite{K}] Let ${\mathcal R} 
= \bigoplus_{a \in {\mathbb Z}_{\geq 0}}({\mathcal I}_a,a)$ 
be an idealistic filtration of i.f.g.\! type.  Assume that 
${\mathcal R}$ is differentially saturated, i.e., 
${\mathcal R} = {\mathcal D}{\mathcal R}$.  Fix 
a closed point $P \in \mathrm{Sing}({\mathcal R}) \subset W$. 
 
Consider the leading algebra $L_P({\mathcal R})$ 
$$L_P({\mathcal R}) := \bigoplus_{a \in {\mathbb Z}_{\geq 0}}\{f \mathrm{\ mod\ } {\mathfrak m}_P^{a + 1} \mid (f,a) \in {\mathcal R}_P, f \in {\mathfrak m}_P^a\} 
\!\subset \bigoplus_{a \in {\mathbb Z}_{\geq 0}}{\mathfrak m}_P^a/{\mathfrak m}_P^{a+1}.$$ 
Then there exists a regular system of parameters 
$(x_1, \ldots\!, x_t, x_{t+1}, \ldots\!, x_d)$ at $P$ 
such that the leading algebra takes the following form: 
 
\medskip 
 
\noindent\emph{Case$\colon 
\mathrm{char}(k) = 0$} 
$$L_P({\mathcal R}) = k[x_1, \ldots, x_t] \subset k[x_1, \ldots, x_t, x_{t+1}, \ldots, x_d] = \bigoplus_{a \in {\mathbb Z}_{\geq 0}}{\mathfrak m}_P^a/{\mathfrak m}_P^{a+1}.$$ 
Moreover, we observe the following: if we take an element $(h_i,1) \in {\mathcal R}_P$ with $h_i \equiv x_i \mathrm{\ mod\ }{\mathfrak m}_P^2$ ($i = 1, \ldots, t$), then the hypersurface $\{h_i = 0\}$ is a hypersurface of maximal contact in the classical sense. 
 
\medskip 
 
\noindent\emph{Case$\colon 
\mathrm{char}(k) = p > 0$} 
$$L_P({\mathcal R}) = k[x_1^{p^{e_1}}, \ldots, x_t^{p^{e_t}}] 
\subset k[x_1, \ldots, x_t, x_{t+1}, \ldots, x_d] = 
\!\bigoplus_{a \in {\mathbb Z}_{\geq 0}} 
\!{\mathfrak m}_P^a/{\mathfrak m}_P^{a+1}$$ 
for some $0 \leq e_1 \leq \cdots \leq e_t$. 
\end{proposition} 
\begin{remark} The former \emph{Case$\colon\mathrm{char}(k) = 0$} 
in the above proposition can be regarded as a special case of 
the latter \emph{Case$\colon\mathrm{char}(k) = p > 0$}, 
by formally setting $p = \infty$, $0 = e_1 = \cdots = e_t$ and $\infty^0 = 1$, where all the $x^{p^e}$-terms with $e > 0$ become ``invisible'' as $p$ goes to $\infty$. 
\end{remark} 
\begin{defn}[Leading Generator System (cf. \cite{K} \cite{KM})]\label{DefLGS} 
Let the situation be as in the proposition above.  Take a subset 
${\mathbb H} =\{(h_{\alpha}, p^{e_{\alpha}})\}_{\alpha = 1}^t 
\subset {\mathcal R}_P$ with 
$$h_{\alpha} \equiv x_{\alpha}^{p^{e_{\alpha}}} \mathrm{\ mod\ }{\mathfrak m}_P^{p^{e_{\alpha}}+1} \quad\text{for}\quad{\alpha} = 1, \ldots, t.$$ 
We say that ${\mathbb H}$ is a leading generator system for ${\mathcal R}_P$ 
(called an LGS for short). 
A leading generator system for $\widehat{{\mathcal R}_P}$ is defined similarly. 
\end{defn} 
\begin{defn}[Invariant ``$\sigma$'' and ``$\tau$'' (cf. \cite{K} \cite{KM})] \label{DefST} 
Let the situation be as in the proposition above.  Then the invariants $\sigma$ and $\tau$ are defined by the following formulas 
$$\sigma(P) := (a_n)_{n \in {\mathbb Z}_{\geq 0}} 
\quad\text{where}\quad a_n = d - \#\{e_{\alpha} \mid e_{\alpha} \leq n\}$$ 
where the value set of the invariant $\sigma$ is given the lexicographical order, and 
$$\tau(P) := \# \text{ of the elements in an LGS} = \# {\mathbb H} = t.$$ 
(We note that the invariants $\sigma$ and $\tau$ are independent of the choice of a regular system of parameters or an LGS, and that it does not matter whether we compute them at the algebraic level or at the analytic level.)  The moral here is that the more $e_{\alpha}$'s we have at the lower level, the smaller the value of the invariant $\sigma$ is and hence we consider the better the LGS ${\mathbb H}$ is. 
\end{defn} 
 
\textbf{Basic strategy to establish our algorithm in 
positive characteristic}: \emph{Follow the construction 
of the algorithm in $\mathrm{char}(k) = 0$, replacing the notion of an HMC to use the induction on dimension by the notion of an LGS to use 
the induction on the invariant $\sigma$. } 
 
\medskip 
 
\textbf{Mechanism of the inductive scheme on the invariant $\sigma$} 
 
Given $(W,{\mathcal R},E)$ (Precisely speaking, the triplet sits in the middle of the sequence, say in year ``$i$'', for resolution of singularities.  However, we omit the subscript ``$(\ )_i$'' indicating 
the year for simplicity of the notation), 
we introduce a triplet of invariants $(\sigma,\widetilde{\mu},s)$ and its associated triplet of data $(W',{\mathcal R}',E')$.  Together they form the following mechanism to realize the inductive scheme on the invariant $\sigma$ 
(We note that there is \emph{no} Key Inductive Lemma in our setting): 
 
\medskip 
 
\noindent\underline{\rm Outline of the mechanism} 
 
\medskip 
 
\indent{\rm (1)}\quad 
We introduce a triplet of invariants $(\sigma, \widetilde{\mu},s)$. 
 
\noindent If $(\sigma,\widetilde{\mu},s) = (\sigma,\infty,0)$ or $(\sigma,0,0)$, then we do not construct the modification $(W',{\mathcal R}',E')$.

\indent$\bullet$\quad 
In case $(\sigma,\widetilde{\mu},s) = (\sigma,\infty,0)$, 
we blow up with center $C = \mathrm{Sing}({\mathcal R})$. 
The nonsingularity of $C$ is guaranteed by the 
Nonsingularity Principle (cf. \cite{K} \cite{KM}), 
while the transversality of $C$ to the boundary $E$ is 
guaranteed by the invariant $s = 0$.  After the blow up, 
the singular locus becomes empty, and resolution of 
singularities for $(W,{\mathcal R},E)$ is achieved. 
 
\indent$\bullet$\quad 
In case $(\sigma,\widetilde{\mu},s) = (\sigma,0,0)$, we are in the monomial case by definition, and we go to Step (5). 
 
\noindent If $(\sigma, \widetilde{\mu},s) \neq (\sigma,\infty,0)$ or $(\sigma,0,0)$, then we construct its associated triplet of data $(W',{\mathcal R}',E')$, called the modification of 
the original triplet $(W,{\mathcal R}$, $E)$, 
having the following properties.  (Actually the ambient space for the modification stays the same, i.e., $W' = W$.) 
 
\indent$\bullet$\quad 
Resolution of singularities for $(W',{\mathcal R}',E')$ implies the strict decrease of the (maximum) value of the triplet $(\sigma, \widetilde{\mu},s)$.  Note that the value of the triplet never increases after transformations. 
 
\indent$\bullet$\quad 
Either the value of $\sigma$ strictly decrease, or the value of $\sigma$ stays the same but the number of the components in the boundary strictly drops.  In either case, we have $(\sigma,\# E) > (\sigma',\# E')$. 
 
\indent{\rm (2)}\quad There is \emph{no} key inductive lemma in our new setting. 
 
\indent{\rm (3)}\quad 
When $(\sigma,\widetilde{\mu},s) \neq (\sigma,\infty,0)$ 
or $(\sigma,0,0)$ in (1), we achieve resolution of singularities 
for $(W',{\mathcal R}',E')$ by induction on $(\sigma,\# E)$, 
which implies the strict decrease of the (maximum) value 
of the triplet $(\sigma, \widetilde{\mu},s)$. 
 
\indent{\rm (4)}\quad 
Repeatedly decreasing the value of the triplet $(\sigma,\widetilde{\mu},s)$, we reach the case where $(\sigma,\widetilde{\mu},s) = (\sigma,\infty,0)$ 
or $(\sigma,0,0)$, the monomial case. 
 
\indent{\rm (5)}\quad 
Finally we have only to construct resolution 
of singularities in the monomial case in order to achieve resolution of singularities of the original triplet $(W, {\mathcal R}, E)$.  However, the problem of resolution of singularities in the monomial case in positive characteristic is quite subtle and very difficult.  We only provide a solution in dimension 3 in \S 5. 
 
\medskip 
 
\textbf{Warning}: Even though the invariant $\sigma$ never increases after blow ups chosen in our algorithm, the number of exceptional divisors $\# E$ and hence the pair $(\sigma,\# E)$ may increase.  Therefore, the description of ``by induction on $(\sigma,\# E)$'' in (3) above is slightly imprecise.  The precise mechanism of the induction is manifested as the weaving of the new strand of invariants ``$\inv_{\mathrm{new}}$''.  See  \S 4.3 for details. 
 
\smallskip 
 
\noindent\underline{\rm Description of the triplet of invariants $(\sigma,\widetilde{\mu},s)$} 
 
\medskip 
 
\indent$\circ\ \sigma$:\quad 
It is the invariant $\sigma$ associated to the differential saturation 
${\mathcal D}{\mathcal R}$ of the idealistic filtration of i.f.g.~type 
${\mathcal R}$ (cf. Definitions \ref{DefDR}, \ref{DefLGS}, and \ref{DefST}). 
 
\indent$\circ\ \widetilde{\mu}$:\quad 
It is the (normalized) weak order modulo LGS of the idealistic filtration 
${\mathcal R}$, with respect to $E_{\mathrm{young}}$. 
 
\indent$\circ\ s$:\quad 
It is the number of the components in $E_{\mathrm{aged}} 
= E \setminus E_{\mathrm{young}}$. 
 
\medskip 
 
We explain how to compute the invariant 
$\widetilde{\mu}$ more in detail: 
Let ${\mathbb H}$ be the LGS chosen. 
Given $f \in \widehat{{\mathcal O}_{W,P}}$, let $f = \sum c_{f,B}H^B$ 
be its power series expansion with respect to the LGS 
(and its associated regular system of parameters) (cf. \cite{KM}). 
Then we define 
\begin{align*} 
\mu_P({\mathcal R}) 
&= \inf\left\{ 
\frac1{a}{\mathrm{ord}_P(c_{f,{\mathbb O}})} 
\mid (f,a) \in {\mathcal R}_P, a > 0\right\} \\ 
&= \inf\left\{ 
\frac1{a}{\mathrm{ord}_P(c_{f,{\mathbb O}})} \mid (f,a) \in \widehat{{\mathcal R}_P}, a > 0\right\}, \\ 
\mu_{P,D}({\mathcal R}) 
&= \inf\left\{ 
\frac1{a}{\mathrm{ord}_{\xi_D}(c_{f,{\mathbb O}})} \mid (f,a) \in {\mathcal R}_P, a > 0\right\} \\ 
&= \inf\left\{ 
\frac1{a}{\mathrm{ord}_{\xi_D}(c_{f,{\mathbb O}})} \mid (f,a) \in \widehat{{\mathcal R}_P}, a > 0\right\}, 
\end{align*} 
where $\xi_D$ is the generic point of a component $D$ in $E_{\mathrm{young}}$.  (For the definition of $E_{\mathrm{young}}$, see (iii) of the remark below.) 
Now we define the invariant $\widetilde{\mu}$ by the following formula 
$$\widetilde{\mu} = \mu_P({\mathcal R}) - \sum_{D \subset E_{\mathrm{young}}} \mu_{P,D}({\mathcal R}).$$ 
It is straightforward to see via the formal coefficient lemma that $\widetilde{\mu}$ is independent of the choice of the LGS (and its associated regular system of parameters) and that $\widetilde{\mu}$ is a nonnegative rational number, since our idealistic filtration is of i.f.g.\! type (cf. \cite{K}\cite{KM}). 
 
\medskip 
 
We make the following remarks on the technical but important points about the LGS (and its associated regular system of parameters), the idealistic filtration of i.f.g.\! type ${\mathcal R}$, and $E_{\mathrm{young}} \subset E$ used in the computation above: 
 
\medskip 
 
\indent{\rm (i)}\quad 
The idealistic filtration of i.f.g.type ${\mathcal R}$ used in the computation depends on the history of the behavior of the invariant $\sigma$. 
 
\smallskip 
 
\noindent\emph{Case$\colon 
$ The value of 
$\sigma$ remains the same as the one in the previous year}. 
 
In this case, we keep ${\mathcal R}$ as it is, which is the transformation of the one in the previous year, even though we compute 
the invariant $\sigma$ using the differential saturation ${\mathcal D}{\mathcal R}$.  We take our LGS (a priori only in ${\mathcal D}{\mathcal R}$) to be the transformation of the one in the previous year, which hence sits inside of ${\mathcal R}$. 
 
\noindent\emph{Case$\colon 
$ The value of $\sigma$ is strictly less than the one in the previous year}. 
 
In this case, we replace the original ${\mathcal R}$ with its differential saturation.  We take our LGS from this replaced ${\mathcal R} = {\mathcal D}{\mathcal R}$, which is differentially saturated, and compute $\widetilde{\mu}$ accordingly.  We remark that, in this case, $E_{\mathrm{young}} = \emptyset$ 
and hence that $\widetilde{\mu} = \mu_P({\mathcal R})$. 
 
\smallskip 
 
We note that, in year $0$, we also replace the original ${\mathcal R}$ with its differential saturation (cf. Remark 3 (2)). 
 
\indent{\rm (ii)}\quad 
The LGS ${\mathbb H} = \{(h_{\alpha}, p^{e_{\alpha}})\}_{\alpha = 1}^t \subset \widehat{{\mathcal R}_P}$ and its associated regular system of 
parameters $X = (x_1, \ldots, x_t, x_{t+1}, \ldots, x_d) 
\subset \widehat{{\mathcal O}_{W,P}}$ are taken in such a way that they 
satisfy the condition $(\heartsuit)$ consisting of 
the three requirements described below: 
 
\smallskip 
 
\noindent\fbox{\textbf{Condition $(\heartsuit)$}} 
 
\smallskip 
 
\indent$\bullet$\quad $h_{\alpha} \equiv x_{\alpha}^{p^{e_{\alpha}}}\ 
\mathrm{ mod }\ \widehat{{\mathfrak m}}^{p^{e_{\alpha}} + 1}$ 
for $\alpha = 1, \ldots, t$, 
 
\indent$\bullet$\quad 
the idealistic filtration of i.f.g.\! type $\widehat{{\mathcal R}_P}$ is 
saturated for $\{{\partial^n}/{\partial x_{\alpha}\!}^n |  n \in {\mathbb Z}_{\geq 0}, \alpha = 1, \ldots, t\}$, 
and 
 
\indent$\bullet$\quad 
the defining equations for the components of $E_{\mathrm{young}}$, which are transversal to the LGS, form a part of the regular system of parameters, i.e., $\{x_D \mid D \subset E_{\mathrm{young}}\} \subset \{x_{t+1}, \ldots, x_d\}$. 
 
\smallskip 
 
\noindent (We remark that it is easy to find such $H \subset {\mathcal R}$ 
and $X \subset {\mathcal O}_{W,P}$ that satisfy all the requirements 
but the second one in Condition $(\heartsuit)$.) 
 
\medskip 
 
We recall the following formal coefficient lemma (cf. \cite{KM}), where the assumption is slightly weaker than the original one in the sense that it does not require the idealistic filtration is ${\mathcal D}$-saturated.  However, the same conclusion is valid with the same proof. 
 
\medskip 
 
\noindent\fbox{\textbf{Formal Coefficient Lemma}} 
 
\smallskip 
 
Let ${\mathcal R}$ be an idealistic filtration of i.f.g.\! type. 
Take a subset ${\mathbb H} 
= \{(h_{\alpha}, p^{e_{\alpha}})\}_{\alpha = 1}^t 
\subset \widehat{{\mathcal R}_P}$ 
(not necessarily an LGS) and a regular system of parameters 
$X = (x_1, \ldots, x_t, x_{t+1}, \ldots, x_d) 
\subset \widehat{{\mathcal O}_{W,P}}$, 
satisfying the first two requirements in condition $(\heartsuit)$. 
Then, for $(f,a) \in \widehat{{\mathcal R}_P}$, we have 
$$(c_{f,B},\max\{a - |[B]|, 0\}) 
\in \widehat{{\mathcal R}_P}$$ for any $B$, and in particular, 
$$(c_{f,{\mathbb O}},a) \in \widehat{{\mathcal R}_P},$$ 
where $f = \sum c_{f,B}H^B$ is the power series expansion of $f$ 
with respect to ${\mathbb H}$ and $X$ (cf. \cite{KM}). 
 
\smallskip 
 
\indent{\rm (iii)}\quad 
The symbol $E_{\mathrm{young}}$ refers to the union of the exceptional divisors created after the time when the current value of $\sigma$ first started.  Therefore, by construction, $E_{\mathrm{young}}$ is transversal to the LGS.  We only use $E_{\mathrm{young}}$ in our algorithm, in contrast to the classical algorithm where we have to use both $E_{\mathrm{young}}$ and $E_{\mathrm{new}}$ (cf. \S 3). 
 
\medskip 
 
\noindent\underline{\rm Description of the triplet $(W',{\mathcal R}',E')$} 
$$ 
W' = W,\quad 
{\mathcal R}' = 
\mathrm{Bdry}\left(\mathrm{Comp}({\mathcal R})\right), 
\quad\text{and}\quad 
E' = E \setminus E_{\mathrm{aged}} = E_{\mathrm{young}}, 
$$ 
where 
 
\indent$\bullet$\quad 
$\mathrm{Comp}({\mathcal R})$ is either 
the transformation of the one in the previous year 
if $(\sigma,\widetilde{\mu})$ stays the same, or 
the one constructed below if 
$(\sigma,\widetilde{\mu})$ strictly decreases, 
and 
 
\indent$\bullet$\quad 
$\mathrm{Bdry}\left(\mathrm{Comp}({\mathcal R})\right)$ 
is either the transformation of the one in the previous year 
if $(\sigma,\widetilde{\mu},s)$ stays the same, or 
$${\mathcal G}(\mathrm{Comp}({\mathcal R}) 
\cup \{(x_D,1) \mid D \subset E_{\mathrm{aged}}\})$$ 
if $(\sigma,\widetilde{\mu},s)$ strictly decreases, 
where 
${\mathcal G}(S)$ is the idealistic filtration of 
i.g. type generated by the set $S$, i.e., the smallest 
idealistic filtration of i.g. type containing $S$, and 
$x_D$ is the defining equation of a component 
$D \subset E_{\mathrm{aged}}$. 
 
\medskip 
 
We note that, if the value of the triplet $(\sigma,\widetilde{\mu},s)$ stays the same as in the previous year, then $(W',{\mathcal R}',E')$ is the transformation of the one in the previous year.  We remark that the symbols ``$\mathrm{Comp}$'' and ``$\mathrm{Bdry}$'' represent the ``Companion'' modification and the ``Boundary'' modification, respectively. 
 
\medskip 
 
\noindent\underline{Construction of $\mathrm{Comp}({\mathcal R})$} 
 
\medskip 
 
We describe the construction of the companion modification 
$\mathrm{Comp}({\mathcal R})$, first at the analytic level, following closely  the construction in the classical setting, and then at the algebraic level, showing that the companion modification at the analytic level ``descends'' to the one at the algebraic level, via the argument of ``\'etale descent''. 
 
\medskip 
 
\fbox{Construction at the analytic level} 
 
\medskip 
 
First, we take an LGS ${\mathbb H} \subset \widehat{{\mathcal R}_P}$ and its associated regular system of parameters $X \subset \widehat{{\mathcal O}_{W,P}}$ satisfying the condition $(\heartsuit)$. 
 
We set 
$${\mathbb M}_X = \prod_{D \subset E_{\mathrm{young}}}x_D^{\mu_{P,D}({\mathcal R})}.$$ 
Recall that $\{x_D \mid D \subset E_{\mathrm{young}}\} \subset X$. 
 
Fix a common multiple $L \in {\mathbb Z}_{> 0}$ of the denominators of $\widetilde{\mu}, {\mu}_P({\mathcal R})$, and $\{\mu_{P,D} \mid D \subset E_{\mathrm{young}}\}$. 
Set 
$$ 
\Xi=\widehat{{\mathcal O}_{W,P}} \otimes_k k[x_{t+1}^{\pm \frac{1}{L}}, \ldots, x_d^{\pm \frac{1}{L}}]. 
$$ 
 
We consider the following notion of an idealistic filtration ${\mathcal Q}$ in the generalized sense: 
$${\mathcal Q} = \bigoplus_{n \in {\mathbb Z}_{\geq 0}} 
\left({\mathcal Q}_{\frac{n}{L}},\frac{n}{L}\right) 
\subset \bigoplus_{n \in {\mathbb Z}_{\geq 0}} 
\left( 
\Xi, 
\frac{n}{L}\right) 
$$ 
is a graded $\widehat{{\mathcal O}_{W,P}}$-algebra, where 
 
\medskip 
 
\indent$\bullet$\quad 
the grading is given by $\{{n}/{L} \mid n \in {\mathbb Z}_{\geq 0}\}$, and it is specified as the level in the second factor, 
 
\indent$\bullet$\quad 
${\mathcal Q}_{{n}/{L}} \subset \Xi$ is an $\widehat{{\mathcal O}_{W,P}}$-submodule with the $\widehat{{\mathcal O}_{W,P}}$-module structure induced by the left multiplication on 
$\Xi$ 
(We emphasize that the tensor ``$\otimes$'' is over $k$.) satisfying 
the condition 
${\mathcal Q}_{{0}/{L}} \supset {\mathcal Q}_{{1}/{L}} 
\supset {\mathcal Q}_{{2}/{L}} \supset \cdots$, 
 
\indent$\bullet$\quad 
the algebra structure is given by the addition and multiplication on $\Xi$, while the $\widehat{{\mathcal O}_{W,P}}$-algebra structure is given by the left multiplication on the first factor of $\Xi$, 
 
\indent$\bullet$\quad 
the differential operators act on the first factor, i.e., for a differential operator $\delta$ of $\deg \delta$ and 
$$q = \left(\sum f \otimes g,{n}/{L}\right) \in \left(\Xi, {n}/{L}\right)$$ 
with $f \in \widehat{{\mathcal O}_{W,P}}$ and $g \in k[x_{t+1}^{\pm \frac{1}{L}}, \ldots, x_d^{\pm \frac{1}{L}}]$, we set 
$$\delta q = \left(\sum \delta f \otimes g, \max\left\{{n}/{L} - \deg \delta, 0\right\}\right).$$ 
 
\medskip 
 
We construct $\widehat{\mathrm{Comp}({\mathcal R})}_{{\mathbb H},X}$ in the following manner: 
 
\medskip 
 
\indent{\rm Step 1.}\quad 
We take the idealistic filtration ${\mathcal Q}_1$ in the generalized sense generated by $\{(f \otimes 1,a) \mid (f,a) \in {\mathcal R}_P\}$ and $\{\left(c_{f,{\mathbb O}} \otimes ({\mathbb M}_X^{-1})^a, \widetilde{\mu} \cdot a\right) \mid (f,a) \in {\mathcal R}_P\}$, i.e., 
$${\mathcal Q}_1 = {\mathcal G}\left( 
\left\{(f \otimes 1,a) 
 ,\ 
\left(c_{f,{\mathbb O}} \otimes ({\mathbb M}_X^{-1})^a, \widetilde{\mu} 
\cdot a\right) 
\mid (f,a) \in {\mathcal R}_P\right\}\right),$$ 
where $c_{f,{\mathbb O}}$ is the constant term of the power series expansion $f = \sum c_{f,B}H^B$ with respect to ${\mathbb H}$ and $X$. 
 
\indent{\rm Step 2.}\quad 
We take the idealistic filtration ${\mathcal Q}_2$ in the generalized sense to be the ${\mathcal D}_{E_{\mathrm{young}}}$-saturation of the idealistic filtration ${\mathcal Q}_1$ in the generalized sense, i.e., 
$${\mathcal Q}_2 = {\mathcal D}_{E_{\mathrm{young}}}\left({\mathcal Q}_1\right),$$ 
where ${\mathcal D}_{E_{\mathrm{young}}}$ represents the logarithmic differentials with respect to the simple normal crossing divisor $E_{\mathrm{young}}$ (cf. \cite{K}). 
 
\indent{\rm Step 3.}\quad 
We take the integral level part ${\mathcal P} = \mathrm{ILP}\left({\mathcal Q}_2\right)$ of the idealistic filtration ${\mathcal Q}_2$ in the generalized sense.  That is to say, 
${\mathcal P} = \bigoplus_{a \in {\mathbb Z}_{\geq 0}}({\mathcal P}_a,a)$ 
is a graded $\widehat{{\mathcal O}_{W,P}}$-algebra where, for 
$a \in {\mathbb Z}_{\geq 0}$, we set ${\mathcal P}_a = ({\mathcal Q}_2)_{{n}/{L}}$ with $a = {n}/{L}$. 
 
\indent{\rm Step 4.}\quad 
By taking the ``round up'' and contraction of ${\mathcal P}$, we obtain the usual idealistic filtration of i.f.g.\! type ${\mathcal C} = \bigoplus_{a \in {\mathbb Z}_{\geq 0}}({\mathcal C}_a,a)$ where we set 
$${\mathcal C}_a = \mathrm{RUC}({\mathcal P}_a) \quad\text{for}\quad a \in {\mathbb Z}_{\geq 0}.$$ 
Note that the ``round up'' and contraction map 
$\mathrm{RUC}: {\mathcal P}_a \rightarrow \widehat{{\mathcal O}_{W,P}}$ 
is given by $\mathrm{RUC}\left(\sum f \otimes ({\mathbb M}_X^{-1})^l\right) = \sum f \cdot \lceil({\mathbb M}_X^{-1})^l\rceil$ for $l \in {\mathbb Z}_{\geq 0}$, where 
$$\lceil({\mathbb M}_X^{-1})^l\rceil = \prod_{D \subset E_{\mathrm{young}}}x_D^{\lceil - \mu_{P,D}({\mathcal R}) \cdot l\rceil}.$$ 
Since we have $\mathrm{ord}_{\xi_D}(c_{f,{\mathbb O}}) \geq \mu_{P,D}({\mathcal R}) \cdot a$ for $(f,a) \in {\mathcal R}_P$ and $D \subset E_{\mathrm{young}}$ by definition, 
and since a logarithmic differential operator with respect to $E_{\mathrm{young}}$ does not decrease the power of $x_D$ 
for $D \subset E_{\mathrm{young}}$, 
we conclude that, 
though the image of the RUC map is only in 
$\widehat{{\mathcal O}_{W,P}}[x_{t+1}^{\pm 1}, \ldots, x_d^{\pm 1}]$ 
a priori, it actually lies within 
$\widehat{{\mathcal O}_{W,P}}[x_{t+1}, \ldots, x_d] = \widehat{{\mathcal O}_{W,P}}$, i.e., 
$${\mathcal C}_a = \mathrm{RUC}({\mathcal P}_a) \subset \widehat{{\mathcal O}_{W,P}} \quad\text{for}\quad a \in {\mathbb Z}_{\geq 0}.$$ 
\indent{\rm Step 5.}\quad We set ${\mathcal C} = \widehat{\mathrm{Comp}({\mathcal R})}_{{\mathbb H},X}$. 
 
\pagebreak[4] 
\begin{remark} \ 
 
\indent{\rm (1)}\quad 
We remark that ${\mathcal C}$ is finitely generated.  In fact, if $\{(f_{\lambda},a_{\lambda}) \mid \lambda \in \Lambda\}$ with $\# \Lambda < \infty$ is a set of (local) generators for ${\mathcal R}$, then by setting 
\begin{align*} 
{\mathcal Q}_{1,\Lambda} 
&= {\mathcal G}\left(\left\{(f_{\lambda} \otimes 1,a_{\lambda}) \mid (f_{\lambda},a_{\lambda}) \in {\mathcal R}_P, \lambda \in \Lambda\right\}\right. \\ 
&\phantom{= \left({\mathcal G}\right)} 
\left. \cup \left\{\left(c_{f_{\lambda},{\mathbb O}} \otimes ({\mathbb M}_X^{-1})^{a_{\lambda}}, \widetilde{\mu} \cdot a_{\lambda}\right) \mid (f,a_{\lambda}) \in {\mathcal R}_P, \lambda \in \Lambda\right\}\right), 
\end{align*} 
we see by the formal coefficient lemma applied to the case of an idealistic filtration in the generalized sense (the same proof works as in the usual case (cf. \cite{KM})) that ${\mathcal Q}_1 \subset {\mathcal D}_{E_{\mathrm{young}}}\left({\mathcal Q}_{1,\Lambda}\right)$ while obviously we have ${\mathcal Q}_{1,\Lambda} \subset {\mathcal Q}_1$, and hence that 
$${\mathcal D}_{E_{\mathrm{young}}}\left({\mathcal Q}_1\right) = {\mathcal D}_{E_{\mathrm{young}}}\left({\mathcal Q}_{1,\Lambda}\right).$$ 
It follows easily from this that ${\mathcal C}$ is finitely generated. 
 
\smallskip 
 
\indent{\rm (2)}\quad 
In Step 1, we take the set $\{\left(c_{f,{\mathbb O}} \otimes ({\mathbb M}_X^{-1})^a, \widetilde{\mu} \cdot a\right) \mid (f,a) \in {\mathcal R}_P\}$ as a part of the generators, where the elements $(f,a)$ belong to the idealistic filtration of i.f.g.\! type ${\mathcal R}_P$ at the algebraic level.  Even if we replace this with 
$\{\left(c_{f,{\mathbb O}} \otimes ({\mathbb M}_X^{-1})^a, 
\widetilde{\mu} \cdot a\right) \mid (f,a) \in \widehat{{\mathcal R}_P}\}$, 
where the elements $(f,a)$ belong to the idealistic filtration of i.f.g.\! type $\widehat{{\mathcal R}_P}$ at the analytic (completion) level, the resulting ${\mathcal C}$ will not change. 
 
In fact, this can be seen as follows: Let ${\mathcal C}_{\text{an}}$ be the one obtained by using the analytic one.  (Note that ${\mathcal Q}_{1,\text{an}}$ and ${\mathcal Q}_{2,\text{an}} = {\mathcal D}_{E_{\mathrm{young}}}\left({\mathcal Q}_{1,\text{an}}\right)$ are defined similarly, to be used in the proof of Lemma 3.)  Since 
\begin{align*} 
\lefteqn{ 
\{\left(c_{f,{\mathbb O}} \otimes ({\mathbb M}_X^{-1})^a, \widetilde{\mu} \cdot a\right) \mid (f,a) \in {\mathcal R}_P\} 
} \\ && 
\subset \{\left(c_{f,{\mathbb O}} \otimes ({\mathbb M}_X^{-1})^a, \widetilde{\mu} \cdot a\right) \mid (f,a) \in \widehat{{\mathcal R}_P}\}, 
\end{align*} 
it is obvious that we 
have ${\mathcal C} \subset {\mathcal C}_{\text{an}}$. 
On the other hand, take $(f,a) \in \widehat{{\mathcal R}_P}$ with the power series expansion $f = \sum c_{f,B}H^B$ with respect to ${\mathbb H}$ and $X$, and $c_{f,{\mathbb O}}$ its constant term.  Then by the formal coefficient lemma, we conclude $(c_{f,{\mathbb O}},a) \in \widehat{{\mathcal R}_P} = {\mathcal R}_P \otimes_{{\mathcal O}_{W,P}}\widehat{{\mathcal O}_{W,P}}$.  This means that there exits a finite set $\{(f_{\lambda},a) \in {\mathcal R}_P \mid \lambda \in \Lambda\}$ such that $c_{f,{\mathbb O}} = \sum r_{\lambda}f_{\lambda}$ for some $r_{\lambda} \in \widehat{{\mathcal O}_{W,P}}$. 
 
Set 
$g = \sum r_{\lambda}c_{f_{\lambda},{\mathbb O}}$ 
and 
$h = \sum r_{\lambda}(f_{\lambda} - c_{f_{\lambda},{\mathbb O}})$ 
with 
$c_{f,{\mathbb O}} = g + h$. 
Then since $c_{h,{\mathbb O}} = 0$, we conclude 
$c_{f,{\mathbb O}} = c_{g,{\mathbb O}} + c_{h,{\mathbb O}} = c_{g,{\mathbb O}}$. 
Now since $\{\left(c_{f_{\lambda},{\mathbb O}} \otimes ({\mathbb M}_X^{-1})^a, \widetilde{\mu} \cdot a\right)\} \subset {\mathcal Q}_1$ and hence 
$\left(g \otimes ({\mathbb M}_X^{-1})^a, \widetilde{\mu} \cdot a\right) 
\in {\mathcal Q}_1$, we conclude, by the formal coefficient lemma 
applied to the case of an idealistic filtration in the generalized sense, that 
$$\left(c_{f,{\mathbb O}} \otimes ({\mathbb M}_X^{-1})^a,\widetilde{\mu} \cdot a\right) = \left(c_{g,{\mathbb O}} \otimes ({\mathbb M}_X^{-1})^a,\widetilde{\mu} \cdot a\right) \in {\mathcal D}_{E_{\mathrm{young}}}\left({\mathcal Q}_1\right) = {\mathcal Q}_2,$$ 
which implies ${\mathcal C}_{\text{an}} \subset {\mathcal C}$. 
 
Therefore, we conclude ${\mathcal C} = {\mathcal C}_{\text{an}}$. 
\end{remark} 
\begin{lemma} The companion modification $\widehat{\mathrm{Comp}({\mathcal R})}_{{\mathbb H},X}$ constructed at the analytic level as above is independent of the choice of an LGS ${\mathbb H}$ and its associated regular system of parameters $X$ satisfying the condition $(\heartsuit)$.  That is to say, if ${\mathbb H}'$ and $X'$ are another LGS and its associated regular system of parameters satisfying the condition $(\heartsuit)$, then we have 
$$\widehat{\mathrm{Comp}({\mathcal R})}_{{\mathbb H},X} = \widehat{\mathrm{Comp}({\mathcal R})}_{{\mathbb H}',X'}.$$ 
We write, therefore, 
$$\widehat{\mathrm{Comp}({\mathcal R})} = \widehat{\mathrm{Comp}({\mathcal R})}_{{\mathbb H},X}$$ 
omitting the reference to the LGS and its associated regular system of parameters used in the construction. 
\end{lemma} 
 
\noindent We call $\widehat{\mathrm{Comp}({\mathcal R})}$ the companion modification at the analytic level. 
 
\medskip 
 
\begin{proof} We consider the following two cases. 
 
\smallskip 
 
\noindent\emph{Case$\colon 
$ For any $e \in {\mathbb Z}_{\geq 0}$ and $(h_{\alpha},p^{e_{\alpha}}) \in {\mathbb H}$ with $e_{\alpha} = e$, 
the element $h_{\alpha}$ is a linear combination of 
$$\{(h'_{\beta})^{p^{e-e'_{\beta}}} \mid e'_{\beta} \leq e\}.$$ 
There is no condition on $X$ or $X'$ other than their being associated to ${\mathbb H}$ and ${\mathbb H}'$, respectively}. 
 
In this case, we claim $\widehat{\mathrm{Comp}({\mathcal R})}_{{\mathbb H}',X'} \subset \widehat{\mathrm{Comp}({\mathcal R})}_{{\mathbb H},X}$. 
 
\smallskip 
 
Take $(f,a) \in {\mathcal R}_P$.  Let $f = \sum c_{f,B}H$ be the power series expansion of $f$ with respect to ${\mathbb H}$ and $X$, with its constant term $c_{f,{\mathbb O}}$.  We have $\left(c_{f,{\mathbb O}} \otimes ({\mathbb M}_X^{-1})^a, \widetilde{\mu} \cdot a\right) \in {\mathcal Q}_1$.  Set $g = c_{f,{\mathbb O}}$ and $h = f - c_{f,{\mathbb O}}$.  Let $g = \sum c'_{g,B}{H'}^B$ and $h = \sum c'_{h,B}{H'}^B$ be the power series expansions of $g$ and $h$, respectively, with respect to ${\mathbb H}'$ and $X'$.  Now we conclude, by the formal coefficient lemma applied to the case of an idealistic filtration in the generalized sense, that 
$$\left(c'_{g,{\mathbb O}} \otimes ({\mathbb M}_X^{-1})^a, \widetilde{\mu} \cdot a\right) \in {\mathcal D}_{E_{\mathrm{young}}}\left({\mathcal Q}_{1,\text{an}}\right) = {\mathcal Q}_{2,\text{an}}.$$ 
On the other hand, the assumption of this case implies that $c'_{h,{\mathbb O}} = 0$ and hence that 
$$c'_{g,{\mathbb O}} = c'_{g,{\mathbb O}} + c'_{h,{\mathbb O}} = c'_{g + h,{\mathbb O}} = c'_{f,{\mathbb O}},$$ 
the constant term of the power series expansion $f = \sum c'_{f,B}{H'}^B$ with respect to ${\mathbb H}'$ and $X'$.  Thus, we have 
$$\left(c'_{f,{\mathbb O}} \otimes ({\mathbb M}_X^{-1})^a, \widetilde{\mu} \cdot a\right) \in {\mathcal D}_{E_{\mathrm{young}}}\left({\mathcal Q}_{1,\text{an}}\right) = {\mathcal Q}_{2,\text{an}}.$$ 
This implies by Remark 5 (2) 
$$\widehat{\mathrm{Comp}({\mathcal R})}_{{\mathbb H}',X'} \subset {\mathcal C}_{\text{an}} = \widehat{\mathrm{Comp}({\mathcal R})}_{{\mathbb H},X}.$$ 
(We note that the generators we choose in Step 1 of the construction for $\widehat{\mathrm{Comp}({\mathcal R})}_{{\mathbb H}',X'}$ are of the form 
$(c'_{f,{\mathbb O}} \otimes ({\mathbb M}_{X'}^{-1})^a, 
\widetilde{\mu} \cdot a)$ (not $\otimes ({\mathbb M}_X^{-1})^a)$, 
which are sitting inside of 
$$ 
\widehat{{\mathcal O}_{W,P}} \otimes_k k[{x'_{t+1}}^{\pm \frac{1}{L}}, \ldots, {x'_d}^{\pm \frac{1}{L}}] 
\quad(\text{not of}\ 
\Xi= 
\widehat{{\mathcal O}_{W,P}} \otimes_k k[x_{t+1}^{\pm \frac{1}{L}}, \ldots, x_d^{\pm \frac{1}{L}}]). 
$$ 
However, these differences only contribute to the multiplication of units after Step 3 and Step 4, and hence do not matter for us to conclude the inclusion above.) 
 
\medskip 
 
\noindent\emph{Case$\colon 
X = X'$}. 
 
In this case, we also claim $\widehat{\mathrm{Comp}({\mathcal R})}_{{\mathbb H}',X'} \subset \widehat{\mathrm{Comp}({\mathcal R})}_{{\mathbb H},X}$. 
 
\smallskip 
 
Take $(f,a) \in {\mathcal R}_P$.  Let $f = \sum c'_{f,B}H'$ be the power series expansion of $f$ with respect to ${\mathbb H}'$ and $X'$, with its constant term $c'_{f,{\mathbb O}}$.  By the formal coefficient lemma, we have 
$(c'_{f,{\mathbb O}},a) \in \widehat{{\mathcal R}_P}$. 
Set $g' = c'_{f,{\mathbb O}} \in \widehat{{\mathcal O}_{W,P}}$.  Let $g' = \sum c_{g',B}H^B$ be the power series expansion of $g'$ with respect to $H$ and $X$ with its constant term $c_{g',{\mathbb O}}$.  Then the assumption of $X = X'$ implies $g' = c_{g',{\mathbb O}}$.  Therefore, we conclude that 
$$\left(c'_{f,{\mathbb O}} \otimes ({\mathbb M}_X^{-1})^a,\widetilde{\mu} \cdot a\right) \in \left\{\left(c_{h,{\mathbb O}} \otimes ({\mathbb M}_X^{-1})^a,\widetilde{\mu} \cdot a\right) \mid (h,a) \in \widehat{{\mathcal R}_P}\right\}$$ 
and hence by Remark 5 (2) (cf. the note at the end of the previous case) that 
$$\widehat{\mathrm{Comp}({\mathcal R})}_{{\mathbb H}',X'} \subset {\mathcal C}_{\text{an}} = {\mathcal C} = \widehat{\mathrm{Comp}({\mathcal R})}_{{\mathbb H},X}.$$ 
 
\medskip 
 
Observe that, given an LGS ${\mathbb H}$ and its associated regular system of parameters $X$ satisfying the condition $(\heartsuit)$, we can reach another LGS ${\mathbb H}'$ and its associated regular system of parameters $X'$ satisfying the condition $(\heartsuit)$ by a transformation described in the former case followed by another transformation described in the latter case.  Therefore, by the above analysis, we have $\widehat{\mathrm{Comp}({\mathcal R})}_{{\mathbb H}',X'} \subset \widehat{\mathrm{Comp}({\mathcal R})}_{{\mathbb H},X}$.  Reversing the role of ${\mathbb H}$ and $X$ with that of ${\mathbb H}'$ and $X'$, we then have 
$\widehat{\mathrm{Comp}({\mathcal R})}_{{\mathbb H},X} \subset 
\widehat{\mathrm{Comp}({\mathcal R})}_{{\mathbb H}',X'}$. 
 
Finally we conclude 
$\widehat{\mathrm{Comp}({\mathcal R})}_{{\mathbb H},X} = \widehat{\mathrm{Comp}({\mathcal R})}_{{\mathbb H}',X'}$. 
\end{proof}

\fbox{Construction at the algebraic level } 
 
\begin{proposition} There exists an idealistic filtration of 
i.f.g.\! type $\mathrm{Comp}({\mathcal R})$ at the algebraic 
level (i.e., over ${\mathcal O}_{W,P}$) such that its completion 
coincides with the companion modification 
$\widehat{\mathrm{Comp}({\mathcal R})}$ at the analytic level, i.e., 
$$\{\mathrm{Comp}({\mathcal R})\}^{\widehat{}} 
:= \mathrm{Comp}({\mathcal R}) \otimes_{{\mathcal O}_{W,P}} 
\widehat{{\mathcal O}_{W,P}} = \widehat{\mathrm{Comp}({\mathcal R})}.$$ 
\end{proposition} 
 
\noindent We call $\mathrm{Comp}({\mathcal R})$ the companion modification at the algebraic level. 
\begin{proof} 
\indent{\rm Step 1}.\quad Descent to the Henselization level. 
 
We first note that the ingredients that we used to construct the companion modification at the analytic level; 
 
\medskip 
 
\indent{\rm (i)}\quad 
the LGS ${\mathbb H}$ and its associated regular system of parameters $X$ satisfying the condition $(\heartsuit)$, 
 
\indent{\rm (ii)}\quad 
the constant term $c_{f,{\mathbb O}}$ of the power series 
expansion $f = \sum c_{f,B}H^B$ for $(f,a) \in {\mathcal R}_P$ 
with respect to ${\mathbb H}$ and $X$ (which shows up in the form of $\left(c_{f,{\mathbb O}} \otimes ({\mathbb M}_X^{-1})^a, \widetilde{\mu} \cdot a\right)$ in the construction of the companion modification), 
 
\medskip 
 
actually can be taken at the Henselization level (i.e., they can be 
taken from the Henselization $({\mathcal O}_{W,P})^{h}$ of 
${\mathcal O}_{W,P}$). 
 
In fact, (i) at the Henselization level is a consequence of 
the classical Weierstrass Preparation Theorem and 
Weierstrass Division Theorem (cf. the proof of Proposition 4 (1)). 
(Alternatively, (i) at the Henselization level can be seen using 
the same argument as the one used to see (ii) at the Henselization level.) 
 
We see (ii) at the Henselization level as follows: 
Set $R = {\mathcal O}_{W,P}$, $R^h$ its Henselization, 
and $\widehat{R}$ its completion.  By replacing $R$ with some local ring of an \'etale cover of $\mathrm{Spec}\ R$, we may assume that the LGS ${\mathbb H}$ and the regular system of parameters $X$ satisfying the condition $(\heartsuit)$ are taken from $R$.  Set $A = k[x_{t+1}, \ldots, x_d]_{(x_{t+1}, \ldots, x_d)}$, $A^h$ its Henselization, and $\widehat{A}$ its completion.  By looking at the power series expansion with respect to ${\mathbb H}$ and its associated regular system of parameters $X$ (cf. \cite{KM}), we have 
$$\phi: \widehat{R}/(h_1, \ldots, h_t) \overset{\sim}\longrightarrow \sum_K \widehat{A}X^K$$ 
where on the right hand side the subscript 
$K = (k_1, \ldots, k_t, k_{t+1}, \ldots, k_d) $ $\in {\mathbb Z}_{\geq 0}^d$ 
for the summation varies in the finite range 
$$ 
\begin{cases} 
0 \leq k_{\alpha} \leq p^{e_{\alpha}} - 1 &\text{for}\quad \alpha = 1, \ldots, t \\ 
k_{\alpha} = 0 &\text{for}\quad \alpha = t+1, \ldots, d. 
\end{cases} 
$$ 
Take an element $f \in R \subset \widehat{R}$. 
What we want to show is $c_{f,\mathbb O} \in R^h$. 
 
It suffices to show 
$$(\star) \quad \phi\left(f\ \mathrm{mod}\ (h_1, \ldots, h_t)\right) = c_{f,{\mathbb O}} \in \sum_KA^hX^K.$$ 
We take the coordinate ring $S$ of an affine open 
neighborhood $P \in \mathrm{Spec}\ S \subset W$ such 
that $f \in S$ and $\{h_{\alpha}\}_{\alpha = 1}^t, X \subset S$.  We denote by ${\mathfrak m}_{S,P}$ the maximal ideal of $S$ corresponding to the point $P$.  Note that the ideal $(h_1, \ldots, h_t, x_{t+1}, \ldots, x_d)R$ is ${\mathfrak m}_{S,P}R = {\mathfrak m}_P = (x_1, \ldots, x_t, x_{t+1}, \ldots, x_d)$-primary (in $R$).  By shrinking $\mathrm{Spec}\ S$ if necessary, we may assume that the only prime ideal containing the ideal $(h_1, \ldots, h_t, x_{t+1}, \ldots, x_d)$ is ${\mathfrak m}_{S,P}$. 
 
We regard $R, R^h, S, A, A^h$ as subrings of $\widehat{R}$, 
i.e., $R, R^h, S, A, A^h \subset \widehat{R}$, and we consider 
the subring $SA^h \subset \widehat{R}$ generated by $S, A^h \subset \widehat{R}$.  By abuse of notation, we denote the image of the natural 
projection $SA^h \subset \widehat{R} \rightarrow 
\widehat{R}/(h_1, \ldots, h_t)$ by $SA^h/(h_1, \ldots, h_t)$. 
 
We claim 
$$(\star\star) \quad \phi(SA^h/(h_1, \ldots, h_t)) = \sum_K A^hX^K,$$ 
which clearly implies the assertion $(\star)$. 
 
In the following, we omit the isomorphism $\phi$ from the left hand side in order to ease the notation.  Therefore, the claim $(\star\star)$ is expressed as an equality 
$$SA^h/(h_1, \ldots, h_t) = \sum_K A^hX^K.$$ 
Obviously, we have 
$$SA^h/(h_1, \ldots, h_t) \supset \sum_K A^hX^K.$$ 
Our goal is to show the equality after taking $\otimes_{A^h}\widehat{A}$, i.e., 
$$\mathrm{L.H.S.} = SA^h/(h_1, \ldots, h_t) \otimes_{A^h}\widehat{A} =  \sum_K A^hX^K \otimes_{A^h}\widehat{A} = \mathrm{R.H.S.},$$ 
which, since $\widehat{A}$ is faithfully flat over $A^h$, implies the original equality above before taking $\otimes_{A^h}\widehat{A}$, i.e., $(\star\star)$. 
 
\medskip 
 
\indent$\bullet$\quad Analysis of R.H.S. 
 
\medskip 
 
We have 
$$\mathrm{R.H.S.} = \sum_K A^hX^K \otimes_{A^h}\widehat{A} = \sum_K\widehat{A}X^K.$$ 
 
\indent$\bullet$\quad Analysis of L.H.S.

\medskip 
 
In order to analyze L.H.S., we look at the morphism 
$\theta\colon \mathrm{Spec}(SA^h/ (h_1, \ldots, h_t)) 
\rightarrow \mathrm{Spec}(A^h)$. 
Observe that $\theta$ is quasi-finite, i.e., 
 
\medskip 
 
\indent{\rm (a)}\quad 
it is of finite type, and 
 
\indent{\rm (b)}\quad 
it has finite fibers. 
 
\medskip 
 
The assertion (a) is immediate, since $S$ is 
finitely generated over $k$.  In order to see 
the assertion (b), first look at the morphism 
$\mathrm{Spec}(SA/ (h_1, \ldots, h_t)) \rightarrow \mathrm{Spec}(A)$. 
The fiber of this morphism over the origin downstairs is the origin upstairs, since the ideal $(h_1, \ldots, h_t, x_{t+1}, \ldots, x_d)$ is $(x_1, \ldots, x_t, x_{t+1}, \ldots, x_d)$-primary. 
Now by what is called ``Zariski's Main Theorem by Grothendieck'' 
(cf. \cite{Raynaud}), we conclude that the morphism has finite fibers.  From this it follows easily that the morphism 
$\theta: \mathrm{Spec}(SA^h/(h_1, \ldots, h_t)) 
\rightarrow \mathrm{Spec}(A^h)$ has finite fibers. 
 
Therefore, since $A^h$ is Henselian, we conclude that the morphism 
$\theta$ is finite, i.e., $SA^h/(h_1, \ldots, h_t)$ is a finite 
$A^h$-module.  This implies that $SA^h/(h_1, \ldots, h_t) 
\otimes_{A^h}\widehat{A}$ is the $(x_{t+1}, \ldots, x_d)$-adic 
completion of the $A^h$-module $SA^h/(h_1, \ldots, h_t)$.  
The latter coincides with the 
$(h_1, \ldots, h_t, x_{t+1}, \ldots, x_d)$-adic completion 
of $SA^h/(h_1, \ldots, h_t)$ viewed as an $SA^h$-module.  Since the ideal 
$(h_1, \ldots, \allowbreak h_t, x_{t+1}, \ldots, x_d)$ 
is $(x_1, \ldots, x_t, x_{t+1}, \ldots, x_d)$-pri\-mary, 
the $(h_1, \ldots, h_t, x_{t+1}, \ldots, x_d)$-adic completion 
coincides with the  $(x_1, \ldots, x_t, x_{t+1}, \ldots, x_d)$-adic 
completion. 
Observing that the 
$(x_1, \ldots,  x_t, x_{t+1}, \ldots, x_d)$-adic completion 
of $SA^h$ is $\widehat{R}$, we see that the 
$(x_1, \ldots, x_t,x_{t+1}, \ldots, x_d)$-adic completion 
of $SA^h/(h_1, \ldots, h_t)$ is $\widehat{R}/(h_1, \ldots, h_t)$.  Summarizing and remembering the convention of expressing the isomorphism $\phi$ as an equality, we conclude 
$$\mathrm{L.H.S.} = SA^h/(h_1, \ldots, h_t) \otimes_{A^h}\widehat{A} = \widehat{R}/(h_1, \ldots, h_t) = \sum_K\widehat{A}X^K.$$ 
Therefore, we have 
$$\mathrm{L.H.S.} = \sum_K\widehat{A}X^K = \mathrm{R.H.S.}.$$ 
This completes the argument for Step 1. 
 
\bigskip 
 
Therefore, we conclude that, for each LGS ${\mathbb H}$ and 
its associated regular system of parameters $X$ satisfying 
the condition $(\heartsuit)$, taken at the Henselization level, 
we have an idealistic filtration of i.f.g.\! type 
$\mathrm{Comp}({\mathcal R})_{{\mathbb H},X}$ at 
the Henselization level (i.e., all the ideals are those of 
$({\mathcal O}_{W,P})^{h}$) such that 
$\widehat{\mathrm{Comp}({\mathcal R})} 
= \widehat{\mathrm{Comp}({\mathcal R})}_{{\mathbb H},X} 
= \{\mathrm{Comp}({\mathcal R})_{{\mathbb H},X}\}^{\widehat{}}$. 
 
\bigskip 
 
\indent{\rm Step 2}.\quad Descent to the algebraic level. 
 
Let $U = \mathrm{Spec}\ {\mathcal O}_{W,P}$.  Noting that 
the Henselization is the direct limit of the local rings 
at the closed points over $P$ on the \'etale covers of $U$, 
we can take a collection of \'etale covers 
$\pi_{\lambda}:U_{\lambda} \rightarrow U$ and 
idealistic filtrations of i.f.g.\! type 
$\mathrm{Comp}({\mathcal R})_{\lambda}$ over 
$U_{\lambda}$ such that, for  $Q \in \pi_{\lambda}^{-1}(P)$, 
we have $\{\mathrm{Comp}({\mathcal R})_{\lambda,Q}\}^{h} 
= \mathrm{Comp}({\mathcal R})_{{\mathbb H},X}$ 
for some ${\mathbb H}$ and $X$ 
described as in Step 1.  This implies 
\begin{align*} 
\{\mathrm{Comp}({\mathcal R})_{\lambda,Q}\}^{\widehat{}} 
&= \{\{\mathrm{Comp}({\mathcal R})_{\lambda,Q}\}^{h}\}^{\widehat{}} 
= \{\mathrm{Comp}({\mathcal R})_{{\mathbb H},X}\}^{\widehat{}} \\ 
&= \widehat{\mathrm{Comp}({\mathcal R})}_{{\mathbb H},X} 
= \widehat{\mathrm{Comp}({\mathcal R})}. 
\end{align*} 
That is to say,  the completion of $\mathrm{Comp}({\mathcal R})_{\lambda,Q}$ canonically coincides with the companion modification at the analytic level $\widehat{\mathrm{Comp}({\mathcal R})}$.  This in turn implies that, over $U_{\lambda} \cap U_{\mu} = U_{\lambda} \times_U U_{\mu}$, we have 
$$\mathrm{Comp}({\mathcal R})_{\lambda}|_{U_{\lambda} \cap U_{\mu}} \overset{\phi_{\lambda \mu}}= \mathrm{Comp}({\mathcal R})_{\mu}|_{U_{\lambda} \cap U_{\mu}},$$ 
and this identification $\phi_{\lambda \mu}$ is canonical (and hence the collection of these identifications automatically satisfies the 
cocycle condition $\phi_{\lambda \mu} \circ \phi_{\mu \nu} 
= \phi_{\lambda \nu}$). 
Now it is a consequence of the general \'etale descent argument (cf. \cite{SGA1}\cite{Milne}) that there exists an idealistic filtration of i.f.g.\! type $\mathrm{Comp}({\mathcal R})$ such that $\pi_{\lambda}^*\left(\mathrm{Comp}({\mathcal R})\right) = \mathrm{Comp}({\mathcal R})_{\lambda}$, and hence that $\{\mathrm{Comp}({\mathcal R})\}^{\widehat{}}= \widehat{\mathrm{Comp}({\mathcal R})}$. 
 
This completes the proof of Proposition 2. 
\end{proof} 
 
\medskip 
 
\noindent\underline{\rm Detailed discussion of the mechanism} 
 
\medskip 
 
Note that constructing the resolution sequence by blowing up is referred to as ``proceeding in the vertical direction'', where the process is numbered by ``year'', while constructing the triplet of invariants and its associated modification in a fixed year, is referred to as ``proceeding in the horizontal direction'', where the process is numbered by ``stage''.  See \S 4.3 for more details. 
 
\medskip 
 
\fbox{Proceeding in the horizontal direction} 
 
\begin{proposition} The value of the pair $(\sigma, \# E)$ strictly decreases as we proceed in the horizontal direction from $(W,{\mathcal R},E)$ to $(W',{\mathcal R}',E')$, i.e., we have 
$$(\sigma, \# E) > (\sigma',\# E').$$ 
\end{proposition} 
\begin{proof} We analyze the assertion in the following two cases.  Note that, since ${\mathcal R} \subset {\mathcal R}'$, we have $\sigma \geq \sigma'$ in both cases. 
 
\medskip 
 
\noindent\emph{Case$\colon 
\widetilde{\mu} \neq 0$ or $\infty$.} 
 
In this case, we claim $\sigma > \sigma'$.  We take the LGS ${\mathbb H} = 
 \{(h_{\alpha},p^{e_{\alpha}})\}_{{\alpha} = 1}^t$ and its associated 
 regular system of parameters 
$X = (x_1, \ldots, x_t, x_{t+1},\ldots, x_d)$ 
satisfying the condition $(\heartsuit$) as given in ``Description of the triplet of invariants $(\sigma,\widetilde{\mu},s)$''.  Since ${\mathcal R}$ is an idealistic filtration of i.f.g.\! type, there exists $(f,a) \in {\mathcal R}_P$ 
such that $\mu_P({\mathcal R}) = {\mathrm{ord}_P(c_{f,{\mathbb O}})}/{a}$, 
where $f = \sum c_{f,B}H^B$ is the power series expansion of $f$ with respect to ${\mathbb H}$ and $X$. 
Set 
$${\mathbb M} 
= \prod_{D \subset E_{\mathrm{young}}}x_D^{\mu_{P,D}({\mathcal R})}.$$ 
 
\smallskip 
 
\emph{Subcase$\colon 
\mu_{P,D}({\mathcal R}) \cdot a \in {\mathbb Z}_{\geq 0} 
\quad \forall D \subset E_{\mathrm{young}}$}. 
 
In this subcase, we have ${\mathbb M}^a \in \widehat{{\mathcal O}_{W,P}}$ and $c_{f,{\mathbb O}} \cdot ({\mathbb M}^{-1})^a = c_{f,{\mathbb O}} \cdot ({\mathbb M}^a)^{-1} \in \widehat{{\mathcal O}_{W,P}}$ by definition. 
Moreover, by construction of the companion modification, 
we have $\left(c_{f,{\mathbb O}} \cdot ({\mathbb M}^a)^{-1}, 
\widetilde{\mu} \cdot a\right) \in \widehat{\mathrm{Comp}({\mathcal R})} 
\subset \widehat{{\mathcal R}'}$.  Note that 
$$\widetilde{\mu} \cdot a = \left(\mu_P({\mathcal R}) - \sum_{D \subset E_{\mathrm{young}}} \mu_{P,D}({\mathcal R})\right) \cdot a = \mathrm{ord}_P\left(c_{f,{\mathbb O}} \cdot ({\mathbb M}^{-1})^a\right).$$ 
Note also that, by the characterization of the power series expansion with respect to ${\mathbb H}$ and $X$, we have 
$$c_{f,{\mathbb O}} = \sum_{K \in ({\mathbb Z}_{\geq 0})^d}b_{f,{\mathbb O},K}X^K$$ 
with $b_{f,{\mathbb O},K} \in k[[x_{t+1}, \ldots, x_d]]$ and with $K = (k_1, \ldots, k_t, k_{t+1}, \ldots, k_d)$ varying in the range satisfying the condition 
$$\left\{\begin{array}{ccl} 
0 \leq k_{\alpha} \leq p^{e_{\alpha}} - 1 &\text{ for }& \alpha = 1, \ldots, t \\ 
k_{\alpha} = 0 &\text{ for }& \alpha = t+1, \ldots, d. 
\end{array}\right.$$ 
Therefore, we conclude 
$$ 
\mathrm{In}\left(c_{f,{\mathbb O}} \cdot ({\mathbb M}^{-1})^a\right) 
= c_{f,{\mathbb O}} \cdot ({\mathbb M}^{-1})^a\ \mathrm{mod}\ 
\widehat{{\mathfrak m}}^{\widetilde{\mu} \cdot a + 1} 
= \sum_{K \in ({\mathbb Z}_{\geq 0})^d}a_{f,{\mathbb O},K}X^K 
$$ 
with $a_{f,{\mathbb O},K} \in k[x_{t+1}, \ldots, x_d]$.  Since $L_P({\mathcal R}) = k[x_1^{p^{e_1}}, \ldots, x_t^{p^{e_t}}]$, we conclude 
$$\mathrm{In}\left(c_{f,{\mathbb O}} \cdot ({\mathbb M}^{-1})^a\right) \not\in L_P({\mathcal R}),$$ 
while by definition 
$$\mathrm{In}\left(c_{f,{\mathbb O}} \cdot ({\mathbb M}^{-1})^a\right) \in L_P(\widehat{{\mathcal R}'}) = L_P({\mathcal R}').$$ 
Therefore, we conclude 
$L_P({\mathcal R}) 
\subsetneqq 
L_P(\widehat{{\mathcal R}'})$ and hence $\sigma > \sigma'$. 
 
\smallskip 
 
\emph{Subcase$\colon 
\mu_{P,D}({\mathcal R}) \cdot a \not\in{\mathbb Z}_{\geq 0}$ for some $D \subset E_{\mathrm{young}}$}. 
 
In this case, take $l \in {\mathbb Z}_{> 0}$ such that 
$$l \cdot \mu_{P,D}({\mathcal R}) \cdot a \in {\mathbb Z}_{\geq 0} 
\quad \forall D \subset E_{\mathrm{young}}.$$ 
Then we have $({\mathbb M}^a)^l \in {\mathcal O}_{W,P}$ 
and $\left\{c_{f,{\mathbb O}} \cdot ({\mathbb M}^{-1})^a)\right\}^l 
\in \widehat{{\mathcal O}_{W,P}}$ by definition. 
Moreover, by construction of the companion modification, we have 
$(\left\{c_{f,{\mathbb O}} \cdot ({\mathbb M}^{-1})^a\right\}^l, l \cdot \widetilde{\mu} \cdot a) \in \widehat{\mathrm{Comp}({\mathcal R})} \subset \widehat{{\mathcal R}'}$.  Note that 
$$l \cdot \widetilde{\mu} \cdot a 
= l \cdot (\mu_P({\mathcal R}) - \sum_{D \subset E_{\mathrm{young}}} \mu_{P,D}({\mathcal R})) \cdot a 
= \mathrm{ord}_P\left(\left\{c_{f,{\mathbb O}} \cdot ({\mathbb M}^{-1})^a\right\}^l\right).$$ 
Note also that, since $\mathrm{ord}_{\xi_D}(c_{f,{\mathbb O}}) \geq \mu_{P,D}({\mathcal R}) \cdot a$ by definition, 
since $\mathrm{ord}_{\xi_D}(c_{f,{\mathbb O}}) \in {\mathbb Z}_{\geq 0}$ and since $\mu_{P,D}({\mathcal R}) \cdot a \not\in {\mathbb Z}_{\geq 0}$ by the subcase assumption, we conclude 
$$\mathrm{ord}_{\xi_D}(c_{f,{\mathbb O}}) > \mu_{P,D}({\mathcal R}) \cdot a.$$ 
This implies, $\left\{c_{f,{\mathbb O}} \cdot ({\mathbb M}^{-1})^a\right\}^l$ is divisible by $x_D$, and hence so is 
$$\mathrm{In}\left(\left\{c_{f,{\mathbb O}} \cdot ({\mathbb M}^{-1})^a\right\}^l\right) = \left\{c_{f,{\mathbb O}} \cdot ({\mathbb M}^{-1})^a\right\}^l\ \mathrm{mod}\ \widehat{{\mathfrak m}}^{l \cdot \widetilde{\mu} \cdot a + 1}.$$ 
Since $L_P({\mathcal R}) = k[x_1^{p^{e_1}}, \ldots, x_t^{p^{e_t}}]$ and since $x_D \in \{x_{t+1}, \ldots, x_d\}$, we conclude 
$$\mathrm{In}\left(\left\{c_{f,{\mathbb O}} \cdot ({\mathbb M}^{-1})^a\right\}^l\right) \not\in L_P({\mathcal R}),$$ 
while by definition 
$$\mathrm{In}\left(\left\{c_{f,{\mathbb O}} \cdot ({\mathbb M}^{-1})^a\right\}^l\right) \in L_P(\widehat{{\mathcal R}'}) = L_P({\mathcal R}').$$ 
Therefore, we conclude $L_P({\mathcal R}) 
\subsetneqq 
L_P(\widehat{{\mathcal R}'})$ and hence $\sigma > \sigma'$. 
 
\smallskip 
 
\noindent\emph{Case$\colon 
\widetilde{\mu} = 0$ or $\infty$.} 
 
In this case, we claim $(\sigma, \# E) > (\sigma', \# E')$. 
Observe $s \neq 0$.  (If $s = 0$, then $(\sigma,\widetilde{\mu},s) 
= (\sigma,0,0)$ or $(\sigma,\infty,0)$, and hence we do not construct 
the modification $(W',{\mathcal R}',E')$.)  Therefore, by definition, 
there exists a divisor $D \subset E_{\mathrm{aged}}$ containing $P$. 
Thus, we have $\# E > \# (E \setminus E_{\mathrm{aged}}) = \# E'$. 
(Note that ``$\#$'' represents the number of the components passing through $P$.) 
\end{proof} 
 
\pagebreak[2] 
\fbox{Proceeding in the vertical direction} 
\begin{proposition} Let $(W,{\mathcal R},E) \overset{\pi}\leftarrow 
(\widetilde{W},\widetilde{\mathcal R},\widetilde{E})$ 
be a transformation in the resolution sequence 
(i.e., $(W_i,{\mathcal R}_i,E_i) \overset{\pi_{i+1}}\leftarrow 
(W_{i+1},{\mathcal R}_{i+1},E_{i+1})$ in the sequence described in 
Problem 4 (cf. Remark 2)).  Take a point 
$\widetilde{P} \in \pi^{-1}(P) \cap \mathrm{Sing}(\widetilde{\mathcal R})$. 
Then the value of the triplet $(\sigma, \widetilde{\mu},s)$ 
does not increase as we proceed in the vertical direction, i.e., we have 
$$(\sigma, \widetilde{\mu},s) \geq (\widetilde{\sigma}, 
\widetilde{\ \widetilde{\mu}\ },\widetilde{s}) 
\quad 
\left(\text{i.e.,}\ (\sigma_i, \widetilde{\mu}_i,s_i) \geq 
(\sigma_{i+1}, \widetilde{\mu}_{i+1},s_{i+1})\right).$$ 
More precisely, we have the following: 
 
\medskip 
 
\indent{\rm (1)}\quad 
The invariant $\sigma$ does not increase, i.e., 
$\sigma \geq \widetilde{\sigma}$.  When $\sigma = \widetilde{\sigma}$, 
the transformation of the LGS for $\widehat{{\mathcal R}_P}$ is 
an LGS for $\widehat{\widetilde{\mathcal R}_{\widetilde{P}}}$. 
Moreover, the following property is preserved going from 
$(W,{\mathcal R},E)$ to 
$(\widetilde{W},\widetilde{\mathcal R},\widetilde{E})$: We can choose 
an LGS ${\mathbb H} = \{(h_{\alpha},p^{e_{\alpha}})\}_{\alpha = 1}^t 
\subset \widehat{{\mathcal R}_P}$ and a regular system of parameters 
$X = (x_1, \ldots, x_t, x_{t+1}, \ldots, x_d)$, taken from 
$\widehat{{\mathcal O}_{W,P}}$, satisfying the condition $(\heartsuit)$ below; 
 
\medskip 
 
\indent$\bullet$\quad 
$h_{\alpha} \equiv x_{\alpha}^{p^{e_{\alpha}}}\ 
\mathrm{ mod }\ \widehat{{\mathfrak m}}^{p^{e_{\alpha}} + 1}$ 
for $\alpha = 1, \ldots, t$, 
 
\indent$\bullet$\quad 
the idealistic filtration of i.f.g.\! type $\widehat{{\mathcal R}_P}$ is 
saturated for $\{{\partial^n}/{\partial {x_{\alpha}\!}^n} | 
n \in {\mathbb Z}_{\geq 0}, \alpha = 1, \ldots, t\}$, 
and 
 
\indent$\bullet$\quad 
the defining equations for the components of $E_{\mathrm{young}}$, which are transversal to the LGS, form a part of the regular system of parameters, i.e., 
$\{x_D \mid D \subset E_{\mathrm{young}}\} \subset \{x_{t+1}, \ldots, x_d\}$. 
 
\medskip 
 
\indent{\rm (2)}\quad 
When $\sigma = \widetilde{\sigma}$, the value of the invariant 
$\widetilde{\mu}$ does not increase, i.e., 
$(\sigma,\widetilde{\mu}) \geq 
(\sigma,\widetilde{\ \widetilde{\mu}\ }) 
= (\widetilde{\sigma},\widetilde{\ \widetilde{\mu}\ })$. 
 
\indent{\rm (3)}\quad 
When $\sigma = \widetilde{\sigma}$, the value of the invariant $s$ does not increase, i.e., $s \geq \widetilde{s}$, and hence combined with (2) we have 
$$(\sigma,\widetilde{\mu},s) 
\geq (\sigma,\widetilde{\ \widetilde{\mu}\ },\widetilde{s}) 
= (\widetilde{\sigma},\widetilde{\ \widetilde{\mu}\ },\widetilde{s}).$$ 
\end{proposition} 
 
\begin{proof} 
{\rm (1)}\quad 
Since $C \subset \mathrm{Sing}({\mathcal R})$ and since $C$ is nonsingular, we conclude that there exists a regular system of parameters 
$(y_1, \ldots, y_r, y_{r+1}, \ldots, y_d) 
\subset {\mathcal O}_{W,P}$ such that 
\begin{itemize} 
\item $C = \{y_1 = \cdots = y_r = 0\}$, and 
 
\item we have for $\alpha = 1, \ldots, t$ 
$$\left\{\begin{array}{lcl} 
h_{\alpha} &\equiv& x_{\alpha}^{p^{e_{\alpha}}}\ \mathrm{ mod }\ \widehat{{\mathfrak m}_P}^{p^{e_{\alpha}}+1}, \text{ and }\\ 
x_{\alpha} &\equiv& \sum_{\beta = 1}^r c_{\alpha,\beta}y_{\beta}\ \mathrm{ mod }\ \widehat{{\mathfrak m}_P}^2 \text{ for some }c_{\alpha,\beta} \in k. 
\end{array}\right.$$ 
 
\end{itemize} 
 
By replacing $(y_1, \ldots, y_r, y_{r+1}, \ldots, y_d)$ with some linear transformation and then by replacing $(x_1, \ldots, x_t, x_{t+1}, \ldots, x_d)$ with $(y_1, \ldots, y_r$, 
$y_{r+1},\ldots, y_d)$, we may assume that we have a regular system of parameters 
$(x_1, \ldots, x_t, x_{t+1}, \ldots, x_d) 
\subset \widehat{{\mathcal O}_{W,P}}$ such that 
\begin{itemize} 
\item $C = \{x_1 = \cdots = x_r = 0\}$ ($t \leq r$), and 
\item $h_{\alpha} \equiv x_{\alpha}^{p^{e_{\alpha}}}\ \mathrm{ mod }\ \widehat{{\mathfrak m}_P}^{p^{e_{\alpha}}+1}$ for $\alpha = 1, \ldots, t$. 
\end{itemize} 
 
Observe that, if $t = r$, then after the blow up, for any point $\widetilde{P}$ over each $x_{\alpha}$-chart $(\alpha = 1, \ldots, t)$ we have $\mathrm{ord}_{\widetilde{P}}(\widetilde{h_{\alpha}}) = 0 < p^{e_{\alpha}}$ 
where $\widetilde{h_{\alpha}} 
= h_{\alpha}/x_{\alpha}^{p^{e_{\alpha}}}$, and hence $\pi^{-1}(P) 
\cap \mathrm{Sing}({\mathcal R}) = \emptyset$.  Therefore, 
we may assume $t < r$ and that our point 
$\widetilde{P} \in \pi^{-1}(P) \cap \mathrm{Sing}(\widetilde{\mathcal R})$ 
is in the $x_{\beta}$-chart for some $(t < \beta \leq r)$ 
with the regular system of parameters $(\widetilde{x}_1, 
\ldots, \widetilde{x}_t, \widetilde{x}_{t+1}, \ldots, 
\widetilde{x}_r, \widetilde{x}_{r+1}, \ldots,\widetilde{x}_d)$ where 
$$\widetilde{x}_{\alpha} =\begin{cases} 
x_{\alpha}/x_{\beta} 
&\text{for}\quad 1 \leq \alpha \leq r, \alpha \neq \beta\\ 
x_{\beta} 
&\text{for}\quad \alpha = \beta, \\ 
x_{\alpha} 
&\text{for}\quad r + 1 \leq \alpha \leq d. 
\end{cases} 
$$ 
(For the indices $t + 1 \leq \alpha \leq r, \alpha \neq \beta$, 
we may have to replace $x_{\alpha}$ with $x_{\alpha} - c_{\alpha}x_{\beta}$ for some $c_{\alpha} \in k$ if necessary.) 
 
Now we look at the transformation $\widetilde{h_{\alpha}} = h_{\alpha}/x_{\beta}^{p^{e_{\alpha}}}$ of 
$$h_{\alpha} = \sum_{K = (k_1, \ldots, k_d) \in ({\mathbb Z}_{\geq 0})^d}c_{\alpha,K}X^K 
\quad\text{with}\quad 
c_{\alpha,K} \in k 
$$ 
for $\alpha = 1, \ldots, t$. 
We compute 
$$ 
\widetilde{h_{\alpha}} 
= \frac{h_{\alpha}}{x_{\beta}^{p^{e_{\alpha}}}} 
= 
\sum_{K = (k_1, \ldots, k_d) \in ({\mathbb Z}_{\geq 0})^d} 
c_{\alpha,K}\frac{X^K}{x_{\beta}^{p^{e_{\alpha}}} } 
= 
\sum_{\widetilde{K} 
= (\widetilde{k}_1, \ldots, \widetilde{k}_d) \in ({\mathbb Z}_{\geq 0})^d} 
c_{\alpha,K}\widetilde{X}^{\widetilde{K}} 
$$ 
where 
$\widetilde{k}_u = k_u$ for $u \neq \beta$ and 
$\widetilde{k}_{\beta} = \sum_{u = 1}^rk_u - p^{e_{\alpha}}$. 
Therefore, we have 
$$\sum_{1 \leq u \leq r, u \neq \beta}k_u \geq p^{e_{\alpha}} \Rightarrow \deg_{\widetilde{X}}(\widetilde{X}^{\widetilde{K}}) \geq \deg_X(X^K).$$ 
Therefore, the only terms $\widetilde{X}^{\widetilde{K}}$ with 
$ \deg_{\widetilde{X}}(\widetilde{X}^{\widetilde{K}}) = p^{e_{\alpha}} 
<\deg_X(X^K)$ 
that can possibly appear 
in $\widetilde{h_{\alpha}}$ have to contain one of $\widetilde{x}_{\beta}, \widetilde{x}_{r+1}, \ldots, \widetilde{x}_d$. 
 
Looking at $\widetilde{h_1}, \widetilde{h_2}, \ldots, \widetilde{h_t}$ in the ascending order, we conclude that 
 
\pagebreak[3] 
 
\indent$\bullet$\quad 
either $\sigma > \widetilde{\sigma}$, 
 
\indent$\bullet$\quad 
or $\sigma = \widetilde{\sigma}$, and we have 
$$\widetilde{h_{\alpha}} \equiv \widetilde{x}_{\alpha}^{p^{e_{\alpha}}} + c_{\alpha,\beta}\widetilde{x}_{\beta}^{p^{e_{\alpha}}} + \sum_{u = r+1}^d c_{\alpha,u}\widetilde{x}_u^{p^{e_{\alpha}}}\ \mathrm{ mod }\ \widehat{{\mathfrak m}_P}^{p^{e_{\alpha}} + 1}$$ 
with $\{c_{\alpha,\beta}, c_{\alpha,r+1}, \ldots, c_{\alpha,d}\} \subset k$ for $\alpha = 1, \ldots, t$, and hence 
$\widetilde{\mathbb H} 
= \{(\widetilde{h_{\alpha}}, p^{e_{\alpha}})\}_{\alpha = 1}^t$ 
is an LGS for $\widetilde{\mathcal R}$. 
 
\medskip 
 
This completes the proof of the first part of (1). 
 
\medskip 
 
Next we look at the ``Moreover'' part of (1) 
 
We show the existence of such an LGS ${\mathbb H} = \{(h_{\alpha},p^{e_{\alpha}})\}_{\alpha = 1}^t$ and its associated regular system of parameters $X = (x_1, \ldots, x_t, x_{t+1}, \ldots, x_d)$ that satisfy the condition $(\heartsuit)$ by induction on the year . 
 
When we are in the year when the value of $\sigma$ first started 
(i.e., the value of $\sigma$ is strictly less than the one in the 
previous year), we replace the original ${\mathcal R}$ with its 
differential saturation (cf. the technical but important points 
in the description of the triplet of invariants 
$(\sigma,\widetilde{\mu},s)$ (i)).  Thus we may assume 
${\mathcal R}$ is ${\mathcal D}$-saturated.  Moreover, 
we have $E_{\mathrm{young}} = \emptyset$.  Therefore, 
we have only to take an LGS ${\mathbb H}$ and $X$ such that 
$h_{\alpha} \equiv x_{\alpha}^{p^{e_{\alpha}}}\ 
\mathrm{ mod }\ \widehat{{\mathfrak m}_P}^{p^{e_{\alpha}} + 1}$ 
for $\alpha = 1, \ldots, t$.  The remaining requirements in 
the condition $(\heartsuit$) are then automatically satisfied. 
 
Now we assume that in the current year we have such an 
LGS ${\mathbb H}$ and its associated regular system of parameters $X$ 
satisfying the condition $(\heartsuit)$. 
We show that, assuming $\sigma = \widetilde{\sigma}$, even after the 
transformation we have such an LGS $\widetilde{\mathbb H}$ and 
$\widetilde{X}$ that satisfy the condition $(\heartsuit)$. 
 
\smallskip 
 
\indent{\rm Step 1}.\quad 
We modify our $X$ so that the LGS ${\mathbb H}$ and $X$ are still 
associated, satisfy the condition $(\heartsuit)$ as before, and 
now satisfy the extra condition that the center is defined by 
$C = \{x_1 = \cdots = x_r = 0\}\ (t \leq r)$. 
 
\medskip 
 
\indent{\rm (a)}\quad 
We take another regular system of parameters 
$Y = (y_1, \ldots, y_r$, 
$y_{r+1},\ldots, y_d)$ such that the center is defined by 
$C = \{y_1 = \cdots = y_r = 0\}$.  Note that, if $C \subset D$ 
for any component $D \subset E_{\mathrm{young}}$, then we include 
$x_D$ in $Y$.  Note that such $x_D$ is included in $X$. 
 
Then since $C \subset \mathrm{Sing}({\mathcal R})$, we conclude that 
$$x_{\alpha} \equiv \sum_{\beta = 1}^r c_{\alpha,\beta}y_{\beta}\ 
\mathrm{ mod }\ \widehat{{\mathfrak m}_P}^2 
\quad \text{ for some }c_{\alpha,\beta} \in k $$ 
for $\alpha = 1, \ldots, t$. 
Therefore, by taking a suitable linear transformation among 
$\{y_1, \ldots, y_r\}$, we may assume 
$$x_{\alpha} \equiv y_{\alpha}\ \mathrm{ mod }\ 
\widehat{{\mathfrak m}_P}^2 \quad\text{for}\quad \alpha = 1, \ldots, t.$$ 
It is straightforward to see that, since ${\mathbb H}$ and 
$X = (x_1, \ldots, x_t, x_{t+1}, \ldots, x_d)$ 
satisfy the condition $(\heartsuit)$, so does ${\mathbb H}$ and $(y_1, \ldots, y_t, x_{t+1}, \ldots, x_d)$.  By replacing $X$ with $(y_1, \ldots, y_t, x_{t+1}, \ldots, x_d)$, we may assume that ${\mathbb H}$ and $X$ satisfy the condition $(\heartsuit)$, and $(x_1, \ldots, x_t, y_{t+1}, \ldots, y_r, y_{r+1}, \ldots, y_d)$ is a regular system of parameters 
with $C = \{x_1 = \cdots = x_t = y_{t+1} = \cdots = y_r = 0\}$. 
 
\indent{\rm (b)}\quad 
We look at $\{y_{\beta} \mid t+1 \leq \beta \leq r\}$.  By applying the Weierstrass Division Theorem consecutively (and replacing the original $y_{\beta}$ after multiplying some unit if necessary) or by simply looking at the power series expansion with respect to $X$, we write 
$$y_{\beta} = \sum_{\alpha = 1}^t q_{\beta,\alpha}x_{\alpha} 
+ g(x_{t+1}, \ldots, x_d) \quad\text{for}\quad \beta = t+1, \ldots, r$$ 
with $q_{\beta,\alpha} \in \widehat{{\mathcal O}_{W,P}}$ and 
$g(x_{t+1}, \ldots, x_d) \in k[[x_{t + 1}, \ldots, x_d]]$.  Set 
$$y_{\beta}' = y_{\beta} - \sum_{\alpha = 1}^t q_{\beta,\alpha}x_{\alpha} 
\quad\text{for}\quad \beta = t+1, \ldots, r.$$ 
Choose $\{x_{r+1}, \ldots, x_d\} \subset \{x_{t+1}, \ldots, x_d\}$ 
(after renumbering of 
$x_{t+1},  \ldots, x_d$ if necessary and keeping the condition 
$\{x_D \mid D \subset E_{\mathrm{young}}\} \subset 
\{x_{t+1}, \ldots, x_d\}$) such that  $(x_1, \ldots, x_t, y'_{t+1}, \ldots, y'_r, x_{r+1}, \ldots, x_d)$ is a regular system of parameters.  Now replace $X$ with 
$(x_1, \ldots, x_t, y'_{t+1}, \ldots, y'_r$, 
$x_{r+1}, \ldots, x_d)$.  Then it is straightforward to see that, since ${\mathbb H}$ and the previous $X$ satisfy the condition $(\heartsuit)$, so do ${\mathbb H}$ and the new $X$.  Now by construction, the center $C$ is defined 
by $C = \{x_1 = \cdots = x_r = 0\}\ (t \leq r)$. 
 
\medskip 
 
Step 2. Analysis after blow up. 
 
\smallskip 
 
As in the proof of the first part, we may assume $t < r$ and that our point $\widetilde{P} \in \pi^{-1}(P) \cap \mathrm{Sing}(\widetilde{\mathcal R})$ is in the $x_{\beta}$-chart for some $(t < \beta \leq r)$ with the regular system of coordinates $\widetilde{X} = (\widetilde{x}_1, \ldots, \widetilde{x}_t, \widetilde{x}_{t+1}, \ldots, 
\widetilde{x}_r, \widetilde{x}_{r+1}, \ldots, \widetilde{x}_d)$ 
where 
$$\widetilde{x}_{\alpha} =\begin{cases} 
x_{\alpha}/x_{\beta} 
&\text{for}\quad 1 \leq \alpha \leq r,\ \alpha \neq \beta, 
\\ 
x_{\beta} 
&\text{for}\quad \alpha = {\beta}, 
\\ 
x_{\alpha} 
&\text{for}\quad r + 1 \leq \alpha \leq d. 
 
\end{cases}$$ 
(For the indices 
$t + 1 \leq \alpha \leq r, \alpha \neq \beta$, 
we may have to replace $x_{\alpha}$ with $x_{\alpha} - c_{\alpha}x_{\beta}$ for some $c_{\alpha} \in k$ if necessary.) 
 
It is straightforward to see that the above $\widetilde{X}$ satisfies the following two conditions; 
 
\smallskip 
 
\indent$\bullet$\quad 
the idealistic filtration of i.f.g.\! type $\widehat{\widetilde{\mathcal R}_{\widetilde{P}}}$, which is (the completion of) the transformation of $\widehat{{\mathcal R}_P}$, is 
saturated for $\{{\partial^n}/{\partial {\widetilde{x}_{\alpha}}^n} \mid n \in {\mathbb Z}_{\geq 0}, {\alpha} = 1, \ldots, t\}$, 
and 
 
\indent$\bullet$\quad 
the defining equations for the $\widetilde{E}_{\mathrm{young}}$ form a part of the regular system of parameters, i.e., 
$\{\widetilde{x}_{\widetilde{D}} \mid \widetilde{D} \subset \widetilde{E}_{\mathrm{young}}\} \subset \{\widetilde{x}_{t+1}, \ldots, \widetilde{x}_d\}$. 
 
\smallskip 
 
We have only to replace $\widetilde{X}$ with $(\widetilde{x}_1', \ldots, \widetilde{x}_t', \widetilde{x}_{t+1}, \ldots, \widetilde{x}_d)$ in order to satisfy the remaining condition 
 
\smallskip 
 
\indent$\bullet$\quad 
$\widetilde{h_{\alpha}} \equiv 
\widetilde{x}_{\alpha}^{p^{e_i}}\ \mathrm{ mod }\ 
\widehat{{\mathfrak m}_{\widetilde{P}}}^{p^{e_{\alpha}} + 1}$ 
for $\alpha = 1, \ldots, t$, 
 
\noindent 
while keeping the other two requirements as above, where 
$$\widetilde{x}_{\alpha}' 
= \widetilde{x}_{\alpha} + c_{\alpha,\beta}^{1/p^{e_{\alpha}}} 
\widetilde{x}_{\beta} + \sum_{u = r+1}^d c_{\alpha,u}^{1/p^{e_{\alpha}}} 
\widetilde{x}_u 
\quad\text{for}\quad \alpha = 1, \ldots, t$$ 
(using the same notation as in the proof of the first part). 
 
Then the LGS $\widetilde{\mathbb H}$ and $\widetilde{X}$ satisfy the condition $(\heartsuit)$. 
 
This finishes the proof of ``Moreover'' part of (1). 
 
\medskip 
 
{\rm (2)}\quad 
We use the same notation used in the proof of (1).  We take an LGS 
${\mathbb H}$ and its associated regular system of parameters $X$, 
satisfying the condition $(\heartsuit)$ and the extra condition that 
the center is defined by $C = \{x_1 = \cdots = x_r = 0\}$ 
(see Step 1 in the proof of ``Moreover'' part of (1)). 
When $\sigma = \widetilde{\sigma}$, we make the following two observations. 
 
\smallskip 
 
\indent$\bullet$\quad 
The transformation $\widetilde{\mathbb H}$ is an LGS of $\widetilde{\mathcal R}$. 
 
\indent$\bullet$\quad 
For $(f,a) \in \widehat{{\mathcal R}_P}$, let $\sum c_{f,B}H^B$ be 
the power series expansion of $f$ with respect to ${\mathbb H}$ and 
its associated regular system of parameters 
$(x_1, \ldots, x_t, x_{t+1}, \ldots, x_d)$, with the 
``constant term'' being $c_{f,{\mathbb O}}$.  Look at 
its transformation $(\widetilde{f},a)$ with 
$\widetilde{f} = f/x_{\beta}^a$.  The constant term 
$c_{\widetilde{f},{\mathbb O}}$ of the transformation 
$\widetilde{f}$ with respect to $\widetilde{H}$ and its 
associated regular system of parameters 
$(\widetilde{x}_1', \ldots, \widetilde{x}_t', \widetilde{x}_{t+1}, 
\ldots, \widetilde{x}_d)$ is the transformation 
$\widetilde{c_{f,{\mathbb O}}} = c_{f,{\mathbb O}}/x_{\beta}^a$, where 
$$\widetilde{x}_{\alpha}' = \widetilde{x}_{\alpha} 
+ c_{\alpha,\beta}^{1/p^{e_{\alpha}}}\widetilde{x}_{\beta} 
+ \sum_{u = r+1}^d c_{\alpha,u}^{1/p^{e_{\alpha}}}\widetilde{x}_u 
\quad\text{for}\quad 
\alpha = 1, \ldots, t.$$ 
 
\smallskip 
 
Now the inequality $\widetilde{\mu} \geq \widetilde{\ \widetilde{\mu}\ }$ follows from these two observations and the condition $C \subset \mathrm{Sing}\left(\mathrm{Comp}({\mathcal R})\right)$ in our new setting by the same argument as the one used to show that the invariant $\mathrm{w\text{-}ord}$ does not increase under transformation in the classical setting. 
 
\medskip 
 
{\rm (3)}\quad 
This follows immediately from the fact that the center $C$ of blow up for the transformation $\pi$ is nonsingular and transversal to the boundary $E$ (hence to $E_{\mathrm{aged}}$), and from the fact that, since $\sigma = \widetilde{\sigma}$, the aged part $\widetilde{E}_{\mathrm{aged}}$ of $\widetilde{E}$ is the strict transform of $E_{\mathrm{aged}}$ by definition. 
\end{proof} 
\end{subsection} 
\begin{subsection}{Weaving of the new strand of invariants ``$\inv_{\mathrm{new}}$''} 
In \S 4.3, we interpret the inductive scheme explained in \S 4.2 in terms of weaving the new strand of invariants ``$\inv_{\mathrm{new}}$''. 
 
We weave the strand of invariants ``$\inv_{\mathrm{new}}$'' consisting of the units of the form $(\sigma^j, \widetilde{\mu}^j,s^j)$, computed from the modifications $(W^j,{\mathcal R}^j,E^j)$ constructed simultaneously along the weaving process. 
 
\medskip 
 
\textbf{Weaving Process} 
 
\medskip 
 
We describe the weaving process inductively. 
 
Note that constructing a sequence for resolution of singularities by blowing up is referred to as ``proceeding in the vertical direction'' passing from one year to the next, indicated by the subscript ``$i$'', while weaving the strand and constructing the modifications passing from one stage to the next, indicated by the superscript ``$j$'', is referred to ``proceeding in the horizontal direction'' staying in a fixed year. 
 
\smallskip 
 
Suppose we have already woven the strands and constructed the modifications up to year $(i-1)$. 
 
Now we are in year $i$ (looking at the neighborhood of a point  $P_i \in \text{Sing}({\mathcal R}_i) \subset W_i$ (cf. Remark 2)). 
 
We start with $(W_i,{\mathcal R}_i,E_i) = (W_i^0,{\mathcal R}_i^0,E_i^0)$, just renaming the transformation $(W_i,{\mathcal R}_i,E_i)$ in year $i$ of the resolution sequence as the $0$-th stage modification $(W_i^0,{\mathcal R}_i^0,E_i^0)$ in year $i$. 
 
Suppose that we have already woven the strand 
up to the $(j-1)$-th unit 
$$\left(\inv_{\mathrm{new}}\right)_i^{\leq j-1} = (\sigma_i^0,\widetilde{\mu}_i^0,s_i^0)(\sigma_i^1,\widetilde{\mu}_i^1,s_i^1) \cdots 
(\sigma_i^{j-1}, \widetilde{\mu}_i^{j-1},s_i^{j-1})$$ 
and that we have also constructed the modifications up to the $j$-th one 
$$(W_i^0,{\mathcal R}_i^0,E_i^0), (W_i^1,{\mathcal R}_i^1,E_i^1), 
\ldots, (W_i^{j-1},{\mathcal R}_i^{j-1},E_i^{j-1}), 
(W_i^j,{\mathcal R}_i^j,E_i^j).$$ 
 
Our task is to compute the $j$-th unit $(\sigma_i^j,\widetilde{\mu}_i^j,s_i^j)$ and construct the $(j+1)$-th modification $(W_i^{j+1},{\mathcal R}_i^{j+1},E_i^{j+1})$ (unless the weaving process is over at the $j$-th stage). 
 
\medskip 
 
\noindent\underline{\rm Computation of the j-th unit $(\sigma_i^j,\widetilde{\mu}_i^j,s_i^j)$} 
 
\medskip 
 
\indent$\circ\ \sigma_i^j$:\quad 
It is the invariant $\sigma$ associated to the 
differential saturation ${\mathcal D}{\mathcal R}_i^j$ 
of the idealistic filtration of i.f.g.\! type ${\mathcal R}_i^j$ 
(cf. Proposition 1). 
 
\smallskip 
 
\indent$\circ\ \widetilde{\mu}_i^j$:\quad 
It is the (normalized) weak order modulo LGS of the idealistic filtration of i.f.g.\! type ${\mathcal R}_i^j$ with respect to $(E_i^j)_{\mathrm{young}}$. 
(For the definition of $(E_i^j)_{\mathrm{young}} (\subset E_i^j)$, see 
the description of the invariant $s_i^j$ below. 
 
We note that, in order to compute $\widetilde{\mu}_i^j$, 
 
\indent$\bullet$\quad 
we keep ${\mathcal R}_i^j$ as it is, which is the 
 transformation of ${\mathcal R}_{i-1}^j$, 
 
\indent$\bullet$\quad 
we take the LGS to be the transformation of the one in 
the previous year if 
$\left(\inv_{\mathrm{new}}\right)_i^{\leq j-1}(\sigma_i^j) 
=  \left(\inv_{\mathrm{new}}\right)_{i-1}^{\leq j-1}(\sigma_{i-1}^j)$, 
 
\indent$\bullet$\quad 
we replace the original ${\mathcal R}_i^j$ 
with its differential saturation ${\mathcal D}{\mathcal R}_i^j$, and 
 
\indent$\bullet$\quad 
we take the 
LGS from the replaced 
${\mathcal R}_i^j = {\mathcal D}{\mathcal R}_i^j$ if 
$\left(\inv_{\mathrm{new}}\right)_i^{\leq j-1} (\sigma_i^j) 
< \left(\inv_{\mathrm{new}}\right)_{i-1}^{\leq j-1}(\sigma_{i-1}^j)$. 
 
\smallskip 
 
\indent$\circ\ s_i^j$:\quad 
It is the number of the components in 
$(E_i^j)_{\mathrm{aged}} = E_i^j \setminus (E_i^j)_{\mathrm{young}}$, 
where $(E_i^j)_{\mathrm{young}} (\subset E_i^j)$ is the union of the exceptional divisors created after the year when the value $\left(\inv_{\mathrm{new}}\right)_i^{\leq j-1}(\sigma_i^j)$ first started.  We note that this third factor 
is included even if $\widetilde{\mu}_i^j = \infty$ or $0$. 
 
\medskip 
 
We note that, if $(\sigma_i^j,\widetilde{\mu}_i^j,s_i^j) 
= (\sigma_i^j,\infty,0)$ or $(\sigma_i^j,0,0)$, 
then we declare that the $(j = m)$-th unit is the last one, and we stop the weaving process at the $m$-th stage in year $i$. 
 
Thus the weaving process of the strand comes to 
an end in a fixed year $i$, with the strand 
$\left(\inv_{\mathrm{new}}\right)_i$ taking the following form 
$$ 
\left({\inv_{\mathrm{new}}}\right)_i 
= (\sigma_i^0,\widetilde{\mu}_i^0,s_i^0) 
(\sigma_i^1,\widetilde{\mu}_i^1,s_i^1) \cdots 
(\sigma_i^j, \widetilde{\mu}_i^j,s_i^j) \cdots 
(\sigma_i^m, \widetilde{\mu}_i^m,s_i^m), 
$$ 
where 
$(\sigma_i^m, \widetilde{\mu}_i^m,s_i^m) = 
(\sigma_i^m,\infty,0)$ or $(\sigma_i^m,0,0)$. 
 
\medskip 
 
\textbf{Termination in the horizontal direction} 
 
\medskip 
 
We note that termination of the weaving process in the horizontal direction is a consequence of the fact that going from the $j$-th unit to the $(j+1)$-th unit we have $(\sigma_i^j, \# E_i^j) > (\sigma_i^{j+1}, \# E_i^{j+1})$ (cf. Proposition 3) and that the value set of $(\sigma, \# E)$ satisfies the descending chain condition. 
 
\medskip 
 
\noindent\underline{\rm Construction of the $(j+1)$-th modification $(W_i^{j+1},{\mathcal R}_i^{j+1},E_i^{j+1})$} 
 
\medskip 
 
We note that we construct the $(j+1)$-th modification 
only when $(\sigma_i^j,\widetilde{\mu}_i^j,s_i^j) 
\neq (\sigma_i^j,\infty,0)$ or $(\sigma_i^j,0,0)$. 
 
We follow the construction described in the mechanism discussed in \S 4.2. 
Starting from $(W_i^j,{\mathcal R}_i^j,E_i^j)$, we construct 
$(W_i^{j+1},{\mathcal R}_i^{j+1},E_i^{j+1})$ as below: 
\begin{align*} 
W_i^{j+1} &= W_i^j, 
\quad 
{\mathcal R}_i^{j+1} = \mathrm{Bdry}\left(\mathrm{Comp} 
({\mathcal R}_i^j)\right), 
\quad 
\text{and}\\ 
E_i^{j+1} &= E_i^j \setminus (E_i^j)_{\mathrm{aged}} 
= (E_i^j)_{\mathrm{young}}, 
\end{align*} 
where 
 
\indent$\bullet$\quad 
denoting 
$\left(\inv_{\mathrm{new}}\right)_i^{\leq j-1} 
(\sigma_i^j,\widetilde{\mu}_i^j)$ 
by $\beta_i^j$, 
we set $\mathrm{Comp}({\mathcal R}_i^j)$ to be either 
the transformation of $\mathrm{Comp}({\mathcal R}_{i-1}^j)$ if 
$\beta_i=\beta_{i-1}$, or 
the one obtained by applying the construction in \S 4.2 if 
$\beta_i^j<\beta_{i-1}^j$, and 
 
\indent$\bullet$\quad 
denoting 
$\left(\inv_{\mathrm{new}}\right)_i^{\leq j}$ by $\gamma_i^j$, 
we set $\mathrm{Bdry}(\mathrm{Comp}({\mathcal R}_i^j))$ to be either 
the transformation of 
$\mathrm{Bdry}(\mathrm{Comp} ({\mathcal R}_{i-1}^j))$ if 
$\gamma_i^j=\gamma_{i-1}^j$, or 
$${\mathcal G}(\mathrm{Comp}({\mathcal R}_i^j) \cup 
\{(x_D,1) \mid D \subset (E_i^j)_{\mathrm{aged}}\})$$ 
if $\gamma_i^j<\gamma_{i-1}^j$. 
 
\medskip 
 
We note that, if 
$\gamma_i^j=\gamma_{i-1}^j$, 
then 
$(W_i^{j+1},{\mathcal R}_i^{j+1},E_i^{j+1})$ is 
the transformation of 
$(W_{i-1}^{j+1},{\mathcal R}_{i-1}^{j+1},E_{i-1}^{j+1})$. 
 
\medskip 
 
We also observe the following. 
 
\smallskip 
 
\noindent$\ast$\quad 
The ambient space remains 
the same throughout a fixed year $i$, i.e., 
$$W_i = W_i^0 = W_i^1 = \cdots = W_i^j = W_i^{j+1} = \cdots 
= W_i^{m-1} = W_i^m.$$ 
This is in clear contrast to the classical setting, where we 
take a consecutive sequence of the hypersurfaces of maximal 
contact (cf. \S 3.3) 
$$W_i = H_i^0 \supset H_i^1 \supset \cdots 
\supset H_i^j \supset H_i^{j+1} \supset \cdots \supset 
H_i^{m-1} \supset H_i^m.$$ 
 
\noindent$\ast$\quad 
The idealistic filtration of i.f.g.\! type gets enlarged 
(not necessarily strictly) under modification, i.e., 
$${\mathcal R}_i = {\mathcal R}_i^0 \subset {\mathcal R}_i^1 
\subset \cdots \subset {\mathcal R}_i^j \subset {\mathcal R}_i^{j+1} \cdots 
\subset {\mathcal R}_i^{m-1} \subset {\mathcal R}_i^m.$$ 
 
\noindent$\ast$\quad 
The boundary divisor decreases (not necessarily strictly) 
under modification, i.e., 
$$E_i = E_i^0 \supset E_i^1 \supset \cdots \supset E_i^j 
\supset E_i^{j+1} \supset \cdots \supset E_i^{m-1} \supset E_i^m.$$ 
 
\medskip 
 
\textbf{Summary of our algorithm in char$(k) = p > 0$ 
in terms of ``$\inv_{\mathrm{new}}$''} 
 
\medskip 
 
We start with $(W,{\mathcal R},E) = (W_0,{\mathcal R}_0,E_0)$. 
 
Suppose we have already constructed the resolution sequence 
up to year $i$ 
$$(W,{\mathcal R},E) = (W_0,{\mathcal R}_0,E_0) 
\leftarrow \cdots \leftarrow (W_i,{\mathcal R}_i,E_i).$$ 
 
We weave the strand of invariants in year $i$ described as above 
$$ 
\left({\inv_{\mathrm{new}}}\right)_i 
= 
(\sigma_i^0,\widetilde{\mu}_i^0,s_i^0) 
(\sigma_i^1,\widetilde{\mu}_i^1,s_i^1) 
\cdots 
(\sigma_i^j, \widetilde{\mu}_i^j,s_i^j) 
\cdots 
(\sigma_i^m, \widetilde{\mu}_i^m, s_i^m), 
$$ 
where 
$(\sigma_i^m, \widetilde{\mu}_i^m, s_i^m)= 
(\sigma_i^m,\infty,0)$ or $(\sigma_i^m,0,0)$. 
 
\medskip 
 
\noindent\emph{Case$\colon 
(\sigma_i^m, \widetilde{\mu}_i^m, s_i^m) = (\sigma_i^m,\infty,0)$}. 
 
In this case, we take the center of blow up in year $i$ for 
the transformation to be the singular locus of the last 
modification $(W_i^m,{\mathcal R}_i^m,E_i^m)$, i.e., 
$\mathrm{Sing}({\mathcal R}_i^m)$, which is easily seen 
to be nonsingular as follows.  We go back to the year 
$\iota:=i_{\text{aged}}^m$ when the value of 
$\left(\inv_{\mathrm{new}}\right)_i^{\leq m-1}(\sigma_i^m)$ 
first started. 
Observe that $\widetilde{\mu}_i^m = \infty$ implies 
$\widetilde{\mu}_{\iota}^m = \infty$.  The Nonsingularity 
Principle (cf. \cite{K} and \cite{KM}) applied to 
${\mathcal R}_{\iota}^m 
= {\mathcal D}{\mathcal R}_{\iota}^m$ 
tells us that there exists a regular system of parameters 
$(x_1, \ldots, x_t, x_{t+1}, \ldots, x_d)$ at $P_{\iota}$ 
such that 
${\mathcal R}_{\iota}^m = {\mathcal G}((x_1^{p^{e_1}}, p^{e_1}), 
\ldots, (x_t^{p^{e_t}}, p^{e_t}))$.  Note that the center 
$C_{\iota}$ of blow up in year $\iota$ is contained in 
$\mathrm{Sing}({\mathcal R}_{\iota}^m)$.  From this it 
follows inductively that there exists a regular system 
of parameters $(x_{1,i'}, \ldots, x_{t,i'}, 
x_{t+1,i'}, \ldots, x_{d,i'})$ at $P_{i'} \in 
\text{Sing}({\mathcal R}_{i'}^m) \subset W_{i'}$ 
such that $\text{Sing}({\mathcal R}_{i'}^j) 
= {\mathcal G}((x_{1,i'}^{p^{e_1}}, p^{e_1}), 
\ldots, (x_{t,i'}^{p^{e_t}}, p^{e_t}))$ for 
$\iota \leq i' \leq i$.  In particular, 
$\mathrm{Sing}({\mathcal R}_i^m)$ is nonsingular.  We note that 
$\mathrm{Sing}({\mathcal R}_i^m) \subset 
\bigcap_{D \subset E_i \setminus (E_i^m)_{\mathrm{young}}}D$ 
by construction of the boundary modifications, and that 
$\mathrm{Sing}({\mathcal R}_i^m)$ is transversal to 
$(E_i^m)_{\mathrm{young}}$.  This implies that the center 
$\mathrm{Sing}({\mathcal R}_i^m)$ is transversal to $E_i$. 
After blow up, the singular locus of $(W_i^m,{\mathcal R}_i^m,E_i^m)$ 
disappears.  Since resolution of singularities for 
$(W_i^m,{\mathcal R}_i^m,E_i^m)$ implies the strict 
decrease of the value of 
$\left({\inv_{\mathrm{new}}}\right)^{\leq m-1}$, we have 
$\left({\inv_{\mathrm{new}}}\right)_i^{\leq m-1} > 
\left({\inv_{\mathrm{new}}}\right)_{i+1}^{\leq m-1}$. 
 
\noindent\emph{Case$\colon 
(\sigma_i^m, \widetilde{\mu}_i^m, s_i^m) = (\sigma_i^m,0,0)$}. 
 
In this case, by following the procedure specified for 
resolution of singularities in the monomial case, we achieve 
resolution of singularities of the $m$-th modification, which 
implies the strict decrease of the value of 
$\left({\inv_{\mathrm{new}}}\right)^{\leq m-1}$.  
(It is possible that in the middle of the procedure 
the value of the $\left({\inv_{\mathrm{new}}}\right)^{\leq m-1}$ 
(or $\left({\inv_{\mathrm{new}}}\right)^{\leq m-1}(\sigma^m)$) 
strictly decreases). 
 
\medskip 
 
\noindent Note: In both cases above, we are using the fact that 
the value of ``${\inv_{\mathrm{new}}}$'' never increases after 
each transformation in the resolution sequence, in order to derive 
the strict decrease of 
$\left({\inv_{\mathrm{new}}}\right)^{\leq m-1}$ (or 
$\left({\inv_{\mathrm{new}}}\right)^{\leq m-1}(\sigma^m)$). 
 
\medskip 
 
\textbf{Termination in the vertical direction} 
 
\medskip 
 
By looking at the conclusions of the two cases above at the end 
of the weaving process in the horizontal direction, we conclude 
that in some year $i'$, the value of the strand strictly decreases, i.e., 
$\left({\inv_{\mathrm{new}}}\right)_i > 
\left({\inv_{\mathrm{new}}}\right)_{i'}$. 
 
Now we claim that the value of the strand ``$\inv_{\mathrm{new}}$'' can not decrease infinitely many times.  In fact, suppose by induction we have shown that the value of $\left({\inv_{\mathrm{new}}}\right)^{\leq t-1}$ can not decrease infinitely many times.  Then after some year the value of $\left({\inv_{\mathrm{new}}}\right)^{\leq t-1}$ stabilizes.  Since the value of the invariant $\sigma$ satisfies the descending chain condition, after some year (say after year $i_{t-1}'$), the value of $\left({\inv_{\mathrm{new}}}\right)^{\leq t-1}(\sigma^t)$ stabilizes.  Therefore, after year $i_{t-1}'$, we use the transformation of ${\mathcal R}_{i_{t-1}'}^t$ in order to compute the invariant $\widetilde{\mu}$.  (See (i) in the technical but important points in the computation of  the triplet $(\sigma,\widetilde{\mu},s)$.)  This implies that the denominator of the invariant $\widetilde{\mu}$ is bounded, and hence that the invariant $\widetilde{\mu}$ can not decrease infinitely many times.  Since the invariant $s$, being a nonnegative integer, can not decrease infinitely many times, we conclude that $\left({\inv_{\mathrm{new}}}\right)^{\leq t-1}(\sigma^t,\widetilde{\mu}^t,s^t) = \left({\inv_{\mathrm{new}}}\right)^{\leq t}$ can not decrease infinitely many times. 
 
For each $t$, let year $i_t$ be the time when the stabilization of $\left({\inv_{\mathrm{new}}}\right)^{\leq t}$ first starts, i.e., 
$$ 
\left({\inv_{\mathrm{new}}}\right)_{i_t - 1}^{\leq t} 
> \left({\inv_{\mathrm{new}}}\right)_{i_t}^{\leq t} 
= \left({\inv_{\mathrm{new}}}\right)_i^{\leq t} 
\quad\text{for}\quad i \geq i_t. 
$$ 
Note that $\{i_t\}$ is a (not necessarily strictly) increasing sequence, 
i.e., $i_t \leq i_{t'}$ if $t \leq t'$. 
Let $\sigma_t = \sigma_{i_t}^t$ be the first factor of the $t$-th unit of $\left({\inv_{\mathrm{new}}}\right)_{i_t}^{\leq t}$.  Note that $\{\sigma_t\}$ is a (not necessarily strictly) decreasing sequence.  That is to say, we have $\sigma_t \geq \sigma_{t+1}$, which follows easily if we look at year $i_{t+1}$ 
and see $\sigma_t = \sigma_{i_t}^t 
= \sigma_{i_{t+1}}^t \geq \sigma_{i_{t+1}}^{t+1} = \sigma_{t+1}$. 
 
We claim that we have either $\sigma_t > \sigma_{t+1}$ or $\sigma_t = \sigma_{t+1} > \sigma_{t+2}$.  This can be seen by the following reasoning. 
 
\smallskip 
 
\indent$\bullet$\quad 
First look at the $t$-th unit in year $i_t$, and observe $(\sigma_{i_t}^t,\widetilde{\mu}_{i_t}^t,s_{i_t}^t) \neq (\sigma_{i_t}^t,\infty,0)$ or $(\sigma_{i_t},0,0)$.  In fact, if $(\sigma_{i_t}^t,\widetilde{\mu}_{i_t}^t,s_{i_t}^t) = (\sigma_{i_t}^t,\infty,0)$ or $(\sigma_{i_t},0,0)$, then the weaving process is over at the $t$-th stage in year $i_t$. 
When $(\sigma_{i_t}^t,\widetilde{\mu}_{i_t}^t,s_{i_t}^t) = (\sigma_{i_t}^t,\infty,0)$, by the single blow up with center $\mathrm{Sing}({\mathcal R}_i^t)$, we accomplish resolution of singularities for $(W_{i_t}^t,{\mathcal R}_{i_t}^t,E_{i_t}^t)$. 
This, however, implies the strict decrease of $\left({\inv_{\mathrm{new}}}\right)^{\leq t-1}$, contradicting its stability after year $i_{t-1} (\leq i_t)$.  If $(\sigma_{i_t}^t,\widetilde{\mu}_{i_t}^t,s_{i_t}^t) = (\sigma_{i_t},0,0)$, then by the procedure of resolution of singularities in the monomial case, we accomplish resolution of singularities for $(W_{i_t}^t,{\mathcal R}_{i_t}^t,E_{i_t}^t)$, which implies the strict decrease of the value of $\left({\inv_{\mathrm{new}}}\right)^{\leq t-1}$. 
This also contradicts the stability of $\left({\inv_{\mathrm{new}}}\right)^{\leq t-1}$ after year $i_{t-1} (\leq i_t)$.  (It is possible that in the middle of the procedure the value of the $\left({\inv_{\mathrm{new}}}\right)^{\leq t-1}$ (or $\left({\inv_{\mathrm{new}}}\right)^{\leq t-1}(\sigma^t)$) strictly decreases. 
In the former case, it would contradict the stability of $\left({\inv_{\mathrm{new}}}\right)^{\leq t-1}$ after year $i_{t-1} (\leq i_t)$.  In the latter case, it would contradict the stability of $\left({\inv_{\mathrm{new}}}\right)^{\leq t}$ after year $i_t$.) 
 
\smallskip 
 
\indent$\bullet$\quad 
If $\widetilde{\mu}_{i_t}^t \neq 0$ or $\infty$, 
then $\sigma_{i_t}^t > \sigma_{i_t}^{t+1}$ 
(cf. the proof of Proposition 3). 
Since $\sigma_{i_t}^{t+1} \geq \sigma_{i_{t+1}}^{t+1}$, we 
conclude $\sigma_t = \sigma_{i_t}^t > \sigma_{i_t}^{t+1} 
\geq \sigma_{i_{t+1}}^{t+1} = \sigma_{t+1}$. 
 
\smallskip 
 
\indent$\bullet$\quad 
We consider the case where $\widetilde{\mu}_{i_t}^t = 0$.  By the first observation, we have $s_{i_t}^t \neq 0$ and the weaving process continues onto the $(t+1)$-th unit in year $i_t$.  We have $\sigma_{i_t}^t \geq \sigma_{i_t}^{t+1}$. 
 
\noindent\emph{Case$\colon 
\sigma_{i_t}^t > \sigma_{i_t}^{t+1}$}. 
 
We have $\sigma_t = \sigma_{i_t}^t > \sigma_{i_t}^{t+1} \geq \sigma_{i_{t+1}}^{t+1} = \sigma_{t+1}$. 
 
\noindent\emph{Case$\colon 
\sigma_{i_t}^t = \sigma_{i_t}^{t+1}$}. 
 
Since year $i_t$ is the time when the value of $\left({\inv_{\mathrm{new}}}\right)_{i_t}^{\leq t}$ first started, we have $\left(E_{i_t}^{t+1}\right)_{\mathrm{young}} = \emptyset$.  The idealistic filtration of i.f.g.\! type ${\mathcal R}_{i_t}^{t+1} \supset {\mathcal R}_{i_t}^t$ contains a monomial of the defining equations of $(E_{i_t}^t)_{\mathrm{young}}$. 
This implies $\widetilde{\mu}_{i_t}^{t+1} \neq 0$ or $\infty$. 
We have $\sigma_{i_t}^{t+1} \geq \sigma_{i_{t+1}}^{t+1}$. 
 
\noindent\emph{Subcase$\colon 
\sigma_{i_t}^{t+1} > \sigma_{i_{t+1}}^{t+1}$}. 
 
We have 
$\sigma_t = \sigma_{i_t}^t = \sigma_{i_t}^{t+1} 
> \sigma_{i_{t+1}}^{t+1} = \sigma_{t+1}$. 
 
\noindent\emph{Subcase$\colon 
\sigma_{i_t}^{t+1} = \sigma_{i_{t+1}}^{t+1}$}. 
 
Since $\widetilde{\mu}_{i_t}^{t+1} \neq \infty$, 
we have $\widetilde{\mu}_{i_{t+1}}^{t+1} \neq \infty$. 
 
\noindent\emph{Subsubcase$\colon 
\widetilde{\mu}_{i_{t+1}}^{t+1} \neq 0$}. 
 
We have $\sigma_{i_{t+1}}^{t+1} > \sigma_{i_{t+1}}^{t+2}$.  This implies $\sigma_t = \sigma_{i_t}^t = \sigma_{i_t}^{t+1} =  \sigma_{i_{t+1}}^{t+1} > \sigma_{i_{t+1}}^{t+2} \geq \sigma_{i_{t+2}}^{t+2} = \sigma_{t+2}$. 
 
\noindent\emph{Subsubcase$\colon 
\widetilde{\mu}_{i_{t+1}}^{t+1} = 0$}. 
 
By the first observation, we have $s_{i_{t+1}}^{t + 1} \neq 0$.  That is to say, there is a component $D$ of $(E_{i_{t+1}}^{t+1})_{\mathrm{aged}} \subset E_{i_{t+1}}^{t+1} = E_{i_{t+1}}^t \setminus (E_{i_{t+1}}^t)_{\mathrm{aged}} = (E_{i_{t+1}}^t)_{\mathrm{young}}$ passing through that point. 
$(E_{i_{t+1}}^t)_{\mathrm{young}}$ is the union of the exceptional divisors created after the year when the value of $\left({\inv_{\mathrm{new}}}\right)_{i_{t+1}}^{\leq t-1}(\sigma_{i_{t+1}}^t) = \left({\inv_{\mathrm{new}}}\right)_{i_t}^{\leq t-1}(\sigma_{i_t}^t)$ first started. 
Therefore, $D$ is transversal to the LGS of ${\mathcal R}_{i_{t+1}}^{t+1}$, which is the transformation of the LGS of ${\mathcal R}_{i_t}^t$ since $\sigma_{i_t}^t = \sigma_{i_t}^{t+1} = \sigma_{i_{t+1}}^{t+1}$. 
Since ${\mathcal R}_{i_{t+1}}^{t+2}$ contains $(x_D,1)$, where $x_D$ is the defining equation of $D$, we conclude $\sigma_{i_{t+1}}^{t+1} > \sigma_{i_{t+1}}^{t+2}$. 
Therefore, we have $\sigma_t = \sigma_{i_t}^t = \sigma_{i_t}^{t+1} 
= \sigma_{i_{t+1}}^{t+1} > \sigma_{i_{t+1}}^{t+2} 
\geq \sigma_{i_{t+2}}^{t+2} = \sigma_{t+2}$. 
 
\indent$\bullet$\quad 
We consider the case where $\widetilde{\mu}_{i_t}^t = \infty$.  By the first observation, we have $s_{i_t}^t \neq 0$ and the weaving process continues onto the $(t+1)$-th unit in year $i_t$.  We have $\sigma_{i_t}^t \geq \sigma_{i_t}^{t+1}$. 
 
\noindent\emph{Case$\colon 
\sigma_{i_t}^t > \sigma_{i_t}^{t+1}$}. 
 
We have $\sigma_t = \sigma_{i_t}^t > \sigma_{i_t}^{t+1} \geq \sigma_{i_{t+1}}^{t+1} = \sigma_{t+1}$. 
 
\noindent\emph{Case$\colon 
\sigma_{i_t}^t = \sigma_{i_t}^{t+1}$}. 
 
Since year $i_t$ is the time when the value of $\left({\inv_{\mathrm{new}}}\right)_{i_t}^{\leq t}$ first started, we have $\left(E_{i_t}^{t+1}\right)_{\mathrm{young}} = \emptyset$.  This implies $\widetilde{\mu}_{i_t}^{t+1} \neq 0$.  If $\widetilde{\mu}_{i_t}^{t+1} \neq \infty$, then we can carry the same argument as above ($\widetilde{\mu}_{i_t}^t = 0$ 
and \emph{Case}$\colon 
\sigma_{i_t}^t = \sigma_{i_t}^{t+1}$) to conclude that either $\sigma_t > \sigma_{t+1}$ or $\sigma_t = \sigma_{t+1} > \sigma_{t+2}$.  Therefore, we have only to consider the case where $\widetilde{\mu}_{i_t}^{t+1} = \infty$.  We have $\sigma_{i_t}^{t+1} \geq \sigma_{i_{t+1}}^{t+1}$. 
 
\noindent\emph{Subcase$\colon 
\sigma_{i_t}^{t+1} > \sigma_{i_{t+1}}^{t+1}$}. 
 
We have $\sigma_t = \sigma_{i_t}^t = \sigma_{i_t}^{t+1} > 
\sigma_{i_{t+1}}^{t+1} = \sigma_{t+1}$. 
 
\noindent\emph{Subcase$\colon 
\sigma_{i_t}^{t+1} = \sigma_{i_{t+1}}^{t+1}$}. 
 
Since $\widetilde{\mu}_{i_t}^{t+1} = \infty$, we have $\widetilde{\mu}_{i_{t+1}}^{t+1} = \infty$. 
By the first observation, we have $s_{i_{t+1}}^{t + 1} \neq 0$.  That is to say, there is a component $D$ of $(E_{i_{t+1}}^{t+1})_{\mathrm{aged}} \subset E_{i_{t+1}}^{t+1} = E_{i_{t+1}}^t \setminus (E_{i_{t+1}}^t)_{\mathrm{aged}} = (E_{i_{t+1}}^t)_{\mathrm{young}}$ passing through that point. 
$(E_{i_{t+1}}^t)_{\mathrm{young}}$ is the union of the exceptional divisors created after the year when the value of $\left({\inv_{\mathrm{new}}}\right)_{i_{t+1}}^{\leq t-1}(\sigma_{i_{t+1}}^t) = \left({\inv_{\mathrm{new}}}\right)_{i_t}^{\leq t-1}(\sigma_{i_t}^t)$ first started. 
Therefore, $D$ is transversal to the LGS of ${\mathcal R}_{i_{t+1}}^{t+1}$, which is the transformation of the LGS of ${\mathcal R}_{i_t}^t$ since $\sigma_{i_t}^t = \sigma_{i_t}^{t+1} = \sigma_{i_{t+1}}^{t+1}$.  Since ${\mathcal R}_{i_{t+1}}^{t+2}$ contains $(x_D,1)$, where $x_D$ is the defining equation of $D$, 
 we conclude $\sigma_{i_{t+1}}^{t+1} > \sigma_{i_{t+1}}^{t+2}$. 
Therefore, 
we have $\sigma_t = \sigma_{i_t}^t = \sigma_{i_t}^{t+1} 
= \sigma_{i_{t+1}}^{t+1} > 
\sigma_{i_{t+1}}^{t+2} \geq \sigma_{i_{t+2}}^{t+2} = \sigma_{t+2}$. 
 
\smallskip 
 
This completes the reasoning for the claim that we have 
either $\sigma_t > \sigma_{t+1}$ or $\sigma_t = \sigma_{t+1} > \sigma_{t+2}$. 
 
\medskip 
 
Now we finish the argument for termination in the vertical direction as follows. 
 
\smallskip 
 
Since the value of the invariant $\sigma$ satisfies the descending chain condition, the increase of the value of $t$ stops after finitely many times. Finally, therefore, we conclude that the value of the strand ``$\inv_{\mathrm{new}}$'' can not decrease infinitely many times. 
 
Therefore, the algorithm terminates after finitely many years, 
achieving resolution of singularities for $(W,{\mathcal R},E)$. 
\end{subsection} 
\begin{subsection}{Brief discussion on the monomial case in positive characteristic} 
Here in \S 4.4, we briefly discuss why the problem of resolution of singularities in the monomial case in positive characteristic is much more subtle and difficult than the one in characteristic zero. 
 
Recall we say that the triplet $(W,{\mathcal R},E)$ is in the monomial case  (at $P \in \mathrm{Sing}({\mathcal R}) \subset W$) in our setting if (and only if) the triplet of invariants takes the value $(\sigma,\widetilde{\mu},s) = (\sigma,0,0)$.  (Precisely speaking, the triplet $(W,{\mathcal R},E)$ sits in the middle of constructing the sequence, say in year ``$i$'' and at stage ``$j$'', for resolution of singularities.  However, we omit the subscript and superscript ``$(\ )_i^j$'', indicating the year and the stage, for simplicity of the notation.) 
 
\smallskip 
 
The description of the monomial case at the analytic level is given below. 
 
\medskip 
 
\noindent\fbox{SITUATION} 
 
\medskip 
 
\indent$\circ$\quad 
The condition $\widetilde{\mu} = 0$ is interpreted as follows: 
 
\smallskip 
 
We can choose a regular system of parameters 
$(x_1, \ldots, x_t, x_{t+1}, \ldots,  x_d)$, 
taken from $\widehat{{\mathcal O}_{W,P}}$, such that 
 
\medskip 
 
\indent{\rm (1)}\quad 
the elements in the LGS ${\mathbb H} = \{(h_{\alpha},p^{e_{\alpha}})\}_{\alpha = 1}^t$ are of the form 
$$h_{\alpha} = x_{\alpha}^{p^{e_{\alpha}}} + \text{higher terms} \quad\text{for}\quad  \alpha = 1, \ldots, t,$$ 
(We sometimes call the higher terms in the above expression ``the tail part''.) 
 
\indent{\rm (2)}\quad 
there is a monomial ${\mathbb M} = \prod_{D \subset E_{\mathrm{young}}}x_D^{r_D}$ of the defining equations $x_D$ of the components $D$ in $E_{\mathrm{young}}$ with 
$$ 
({\mathbb M},a) \in \widehat{{\mathcal R}_P} 
\text{ for some }a \in {\mathbb Z}_{> 0}, 
$$ 
where $\{x_D \mid D \subset E_{\mathrm{young}}\} \subset \{x_{\alpha} \mid \alpha = t + 1, \ldots, d\}$ and $\sum r_D > a$, 
 
\indent{\rm (3)}\quad 
the idealistic filtration of i.f.g.\! type $\widehat{{\mathcal R}_P}$ 
is 
saturated for 
$\{{\partial^n}/{\partial {x_{\alpha}}^n} \mid n \in {\mathbb Z}_{\geq 0}, \alpha = 1, \ldots, t\}$, 
 
\medskip 
 
\noindent 
satisfying the following condition: for an arbitrary $(f,\lambda) \in \widehat{{\mathcal R}_P}$ with $f = \sum c_{f,B}H^B$ 
being the power series expansion with respect to the LGS ${\mathbb H}$ and the associated to the regular system of parameters 
$(x_1, \ldots, x_t, x_{t+1}, \ldots, x_d)$, 
we have ${\mathbb M}^{{\lambda}/{a}}$ dividing the constant term $c_{f,\mathbb O}$, i.e., ${\mathbb M}^{{\lambda}/{a}} | c_{f,\mathbb O}$.  Note that, using the formal coefficient lemma, we see $(c_{f,\mathbb O},\lambda) \in \widehat{{\mathcal R}_P}$. 
 
\medskip 
 
\indent$\circ$\quad 
The condition $s = 0$ is of course equivalent to saying that 
there is no component of $E_{\mathrm{aged}}$ passing through $P$. 
 
\medskip 
 
\noindent \fbox{$\mathrm{char}(k) = 0$} 
 
\smallskip 
 
In fact, the above description \fbox{SITUATION} of the monomial 
case is also valid when $\mathrm{char}(k) = 0$, with all the 
elements of the LGS concentrated at level 1, i.e., 
$p^{e_{\alpha}} = 1$ for $\alpha = 1, \ldots, t$.  
Moreover, we can replace $x_{\alpha}$ with $h_{\alpha}$ 
so that we have $h_{\alpha} = x_{\alpha}$ 
for $\alpha = 1, \ldots, t$. 
Then there is \emph{no} ``tail part''. 
In order to construct resolution of singularities for 
$(W,{\mathcal R},E)$, we have only to carry out the 
resolution process for $(V, (({\mathbb M}),a)|_V, E|_V)$, 
which is the triplet in the monomial case in the classical 
setting, where $V = \{x_1 = \cdots = x_t = 0\}$ is a nonsingular 
subvariety inside of $W$.  Note that, since $E_{\mathrm{aged}} 
= \emptyset$ (in a neighborhood of $P$, because $s = 0$), 
the third factor $E|_V$ is a simple normal crossing divisor on $V$. 
 
\medskip 
 
\noindent \fbox{$\mathrm{char}(k) = p > 0$} 
 
\smallskip 
 
In contrast to the case in $\mathrm{char}(k) = 0$, the elements in 
the LGS may not be concentrated at level 1 in general.  Therefore, 
we usually have the tail parts for those $h_{\alpha}$'s at higher levels in 
$\mathrm{char}(k) = p > 0$. 
The hypersurfaces defined by $\{h_{\alpha} = 0\}$ by those elements are 
singular hypersurfaces.  Therefore, one is forced to analyze the 
monomial restricted to a singular subvariety (defined as the 
intersection of the singular hypersurfaces), or alternatively 
to analyze the combination of the monomial and the elements of 
the LGS with the tail parts included, while sticking to the original 
nonsingular ambient space.  The latter is what we do in \S 5 of this 
paper in dimension 3. 
 
\medskip 
 
In the simplest terms, NO or YES ``tail part'' is what makes the 
difference between the monomial case in $\mathrm{char}(k) = 0$ 
and the one in $\mathrm{char}(k) = p > 0$. 
\end{subsection} 
\end{section} 
\begin{section}{Detailed discussion on the monomial case in dimension 3} 
The purpose of \S 5 is to discuss how to construct resolution of singularities (at the local level (cf. Remark 2)) in the monomial case.  We refer the 
reader to \S 4.4 \fbox{SITUATION} for the precise description of the monomial case. 
\begin{subsection}{Case analysis according to the invariant ``$\tau$''} 
In \S 5.1, we analyze the situation according to 
the value of the invariant $\tau$. 
 
Recall that the invariant $\tau$ is just the number of the elements in an LGS, and hence that, in dimension 3, it takes the value $\tau = 0,1,2,3$. 
 
It turns out that the analysis of the case $\tau = 0,2$ or $3$ is rather easy.  We devote \S 5.2, \S 5.3, \S 5.4 to the analysis of the most subtle and difficult case $\tau = 1$. 
 
\smallskip 
 
\noindent \fbox{\textbf{Case}$\colon 
\tau = 0$} 
 
\smallskip 
 
In this case, there is no element in an LGS.  We conclude that for an arbitrary $(f,\lambda) \in \widehat{{\mathcal R}_P}$ the monomial 
${\mathbb M}^{{\lambda}/{a}}$ divides $c_{f,{\mathbb O}} = f$.  Therefore, we can carry out the same algorithm for resolution of singularities of the triplet $(W,({\mathcal I},a),E) = (W,(({\mathbb M}),a),E)$ in the monomial case in the classical setting in characteristic zero, using the invariant $\Gamma$ discussed in \S 3.4.  (Note that in the middle of the procedure the invariant $\sigma$ may drop.  If that happens, then we are no longer in the monomial case.  In that case, we go through the mechanism described in \S 4.2 with the reduced new value of $\sigma$ to reach the new monomial case.) 
 
\smallskip 
 
\noindent \fbox{\textbf{Case}$\colon 
\tau = 1$} 
 
\smallskip 
 
This case will be thoroughly discussed in \S 5.2, \S 5.3, \S 5.4. 
 
\smallskip 
 
\noindent \fbox{\textbf{Case}$\colon 
\tau = 2$} 
 
\smallskip 
 
In this case, we can choose a regular system of parameters $(x,y,z)$, taken from $\widehat{{\mathcal O}_{W,P}}$, such that 
 
\medskip 
 
\indent{\rm (1)}\quad 
the two elements in the LGS ${\mathbb H} = \{(h_1,p^{e_1}), (h_2,p^{e_2})\}$ are of the form 
$$\left\{\begin{array}{lcl} 
h_1 &=& z^{p^{e_1}} + \text{higher terms} \\ 
h_2 &=& y^{p^{e_2}} + \text{higher terms}, 
\end{array}\right.$$ 
 
\indent{\rm (2)}\quad there is a monomial ${\mathbb M} = x^r$ of the defining equation $x$ of the component $\{x = 0\} \subset E_{\mathrm{young}}$ with 
$$({\mathbb M},a) = (x^r,a) \in \widehat{{\mathcal R}_P} \text{ for some }a \in {\mathbb Z}_{> 0},$$ 
 
\indent{\rm (3)}\quad the idealistic filtration $\widehat{{\mathcal R}_P}$ is 
saturated for $\{{\partial^n}/{\partial z^n},  {\partial^m}/{\partial y^m}\mid n, m \in {\mathbb Z}_{\geq 0}\}$. 
 
\medskip 
 
Then it is easy to see that $\mathrm{Sing}({\mathcal R}) = P$ (in a neighborhood of $P$) and hence that 
the only possible transformation is the blow up with center $P$. 
After blow up, we see that the (possibly) non-empty singular locus lies only over the $x$-chart.  We also see that the singular locus, if non-empty, consists of a single point $\widetilde{P} \in \widetilde{W}$ (in a neighborhood of the inverse image of $P$) with a regular system of parameters $(\widetilde{x}, \widetilde{y}, \widetilde{z}) = (x, y/x, z/x)$.  The new LGS 
$$\widetilde{\mathbb H} = \{(\widetilde{h_1},p^{e_1}), (\widetilde{h_2},p^{e_2})\} = \{(h_1/x^{p^{e_1}},p^{e_1}), (h_2/x^{p^{e_2}},p^{e_2})\}$$ 
are of the form 
$$\left\{\begin{array}{lcl} 
\widetilde{h_1} &=& (\widetilde{z}')^{p^{e_1}} + \text{higher terms} \\ 
\widetilde{h_2} &=& (\widetilde{y}')^{p^{e_2}} + \text{higher terms} 
\end{array}\right.$$ 
where (cf. the proof of Proposition 4) 
$$\left\{\begin{array}{lcl} 
\widetilde{z}' &=& \widetilde{z} + (c_z)^{1/p^{e_1}}\widetilde{x} \\ 
\widetilde{y}' &=& \widetilde{y} + (c_y)^{1/p^{e_2}}\widetilde{x} 
\end{array}\right. \quad\text{for some}\quad c_z, c_y \in k$$ 
(in case the invariant $\sigma$ does not decrease).  We calculate the new monomial to be 
$$(\widetilde{x}^{r - a},a).$$ 
Since the power of $x$ in the above monomial can not decrease infinitely many times, we achieve resolution of singularities after finitely many repetitions of this procedure.  (Note that in the middle of the procedure the invariant $\sigma$ may drop.  If that happens, then we are no longer in the monomial case.  In that case, we go through the mechanism described in \S 4.2 with the reduced new value of $\sigma$ to reach the new monomial case.) 
 
\smallskip 
 
\noindent \fbox{\textbf{Case}$\colon 
\tau = 3$} 
 
\smallskip 
 
This case does not happen.  In fact, suppose this case did happen. 
 Then go back to the year $\iota:=i_{\mathrm{aged}}$ 
when the current value of $\sigma$ first started.  Since the current 
value of $\tau$ is equal to $3$, so is the value of $\tau$ back in 
year $\iota$.  Take the blow up with center 
$P_{\iota}$, which is the image of $P$ in year 
$\iota$.  Then it is immediate to see that there 
is no singular locus (in a neighborhood of the inverse image of 
$P_{\iota}$) after blow up.  This implies in turn that 
$\mathrm{Sing}({\mathcal R}_{\iota}) = P_{\iota}$ 
(in a neighborhood of $P_{\iota}$).  The only possible 
transformation in the resolution sequence, therefore, is the blow up 
with center $P_{\iota}$.  After blow up, however, 
we already saw $\mathrm{Sing}({\mathcal R}_{\iota+1}) 
= \emptyset$ over $P_{\iota}$.  This is a contradiction, 
since $P_{\iota+1}$, which is the image of $P$ in year 
$\iota+1$, should be included in 
$\mathrm{Sing}({\mathcal R}_{\iota+1})$.  (Note that 
we have $\iota < i$, since the value of 
$\widetilde{\mu}$ is never zero when the new value of $\sigma$ 
starts in year $\iota$ and since the current value of 
$\widetilde{\mu}$ is zero being in the monomial case in year $i$.) 
 
\medskip 
 
\underline{\textbf{Focus on the case $\tau = 1$}} 
 
\medskip 
 
In the following \S 5.2, \S 5.3, \S 5.4, we focus on, and restrict ourselves to, the case $\tau = 1$.  We carry out the computation of the invariants at the analytic (completion) level, even though the centers of blow ups are chosen at the algebraic level, and hence all the procedures in the algorithm are carried out at the algebraic level. 
 
\medskip 
 
We restate the \fbox{SITUATION} described in \S 4.4 in a slightly refined form in our particular case $\tau = 1$, for the sole purpose of fixing the notation for \S 5.2, \S 5.3, \S 5.4.

\medskip 
 
\noindent \fbox{SITUATION} 
 
\medskip 
 
We can choose a regular system of parameters $(x,y,z)$, taken from $\widehat{{\mathcal O}_{W,P}}$, such that 
 
\medskip 
 
\indent{\rm (1)}\quad 
via the Weierstrass Preparation Theorem the unique element in the LGS ${\mathbb H} = \{(h,p^e)\}$ is of the form 
$$h = z^{p^e} + a_1z^{p^e-1} + a_2z^{p^e-2} + \cdots + a_{p^e-1}z + a_{p^e}$$ 
with 
$$a_i \in k[[x,y]] \quad \text{and}\quad \mathrm{ord}_P(a_i) > i \quad\text{for}\quad i = 1, \ldots,p^e,$$ 
 
\indent{\rm (2)}\quad 
there is a monomial ${\mathbb M} = x^{\alpha}y^{\beta}$ of the defining equation(s) of the component(s) $H_x = \{x = 0\}$ (and possibly $H_y = \{y = 0\}$) in $E_{\mathrm{young}}$ with 
$$({\mathbb M},a) = (x^{\alpha}y^{\beta},a) \in {\mathcal R}_P 
\quad\text{and}\quad \alpha + \beta > a,$$ 
(We write ${\mathbb M}_u := x^{\alpha/a}y^{\beta/a}$ and call ${\mathbb M}_u$ 
the \emph{usual} monomial.) 
 
\indent{\rm (3)}\quad 
the idealistic filtration $\widehat{{\mathcal R}_P}$ is 
saturated for $\{{\partial^n}/{\partial z^n}\mid n\in {\mathbb Z}_{\geq 0}\}$, 
 
\medskip 
 
\noindent 
satisfying the following condition: for an arbitrary $(f,\lambda) \in \widehat{{\mathcal R}_P}$ 
with $f = \sum c_{f,B}H^B = \sum c_{f,b}h^b$ being the 
power series expansion with respect to the LGS ${\mathbb H} = \{(h,p^e)\}$ and its associated regular system of parameters $(x,y,z)$, we have 
${\mathbb M}^{{\lambda}/{a}}$ dividing the constant term $c_{f,{\mathbb O}} = c_{f,0}$, i.e., ${\mathbb M}^{{\lambda}/{a}} \mid c_{f,{\mathbb O}} = c_{f,0}$. 
 
\medskip 
 
In particular, by looking at ${\partial^n h}/{\partial z^n}$ for 
$n = 1, \ldots, p^e - 1$, we see that the coefficient $a_i$ is divisible by $({\mathbb M}_u)^i$ for 
$i = 1, \ldots, p^e - 1$ (but maybe not for $i = p^e$).  That is to say, 
$$({\mathbb M}_u)^i \mid a_i,\quad \text{i.e.,}\ \ 
x^{\lceil {i \alpha}/{a}\rceil}y^{\lceil {i \beta}/{a}\rceil} \mid a_i \quad\text{for}\quad i = 1, \ldots, p^e - 1.$$ 
 
Throughout \S 5.2, \S 5.3, \S 5.4, we are under \fbox{SITUATION} described as above. 
\end{subsection} 
\begin{subsection}{``Cleaning'' and the invariant ``$\mathrm{H}$'' (in the case $\tau = 1$)} 
The purpose of this section is to introduce the invariant ``$\mathrm{H}$'' through the process of ``cleaning''.  We follow closely the argument developed by Benito-Villamayor \cite{BV3}, making some modifications to fit it into our own setting in the framework of the Idealistic Filtration Program. 
\begin{defn} We define the slope of $h$ at $P$ with respect to $(x,y,z)$ by the formula 
$$\mathrm{Slope}_{h,(x,y,z)}(P) 
= \min\left\{\frac1{p^e} {\mathrm{ord}_P(a_{p^e})},\ \mu(P)\right\}.$$ 
\end{defn} 
\begin{remark} \ 
 
\indent{\rm (i)}\quad 
Since we are in the monomial case and hence $\widetilde{\mu} = 0$ and since the monomial is $({\mathbb M},a) = (x^{\alpha}y^{\beta},a)$, we compute 
$$\mu(P) 
= \frac1{a}{\mathrm{ord}_P(x^{\alpha}y^{\beta})} 
= \frac{\alpha + \beta}{a} 
= \mathrm{ord}_P({\mathbb M}_u),$$ 
while 
$$\mu(\xi_{H_x}) = \frac{\alpha}{a} 
\quad\text{and}\quad 
\mu(\xi_{H_y}) = \frac{\beta}{a}.$$ 
 
\indent{\rm (ii)}\quad 
We have a ``very good'' control over the coefficients $a_i$ except for the constant term $a_{p^e}$, in the sense that 
$$({\mathbb M}_u)^i = \left(x^{\alpha/a}y^{\beta/a}\right)^{i} \mid 
a_i \quad\text{for}\quad i = 1, \ldots, p^e - 1,$$ 
which implies 
$$\frac1{i}{\mathrm{ord}_P(a_i)} \geq \mu(P) \quad\text{for}\quad 
i = 1, \ldots, p^e - 1.$$ 
Therefore, this control leads to the following observation 
\begin{align*} 
\mathrm{Slope}_{h,(x,y,z)}(P) 
&= \min\left\{\mu(P),\ \frac1{p^e}{\mathrm{ord}_P(a_{p^e})}\right\} \\ 
&= \min\left\{\mu(P),\ \frac1{i}{\mathrm{ord}_P(a_i)}\ ;\ 
i = 1, \ldots, p^e\right\}. 
\end{align*} 
\end{remark} 
\begin{defn}[Well-adaptedness (cf. \cite{BV3})]\label{DefWA} 
We say $h$ is well-adapted at $P$ with respect to $(x,y,z)$ if one of the following two conditions holds: 
\begin{enumerate} 
\item[A.] $\mathrm{Slope}_{h,(x,y,z)}(P) = \mu(P)$. 
\item[B.] 
$\mathrm{Slope}_{h,(x,y,z)}(P) 
= {\mathrm{ord}_P(a_{p^e})}/{p^e} <\mu(P)$ 
and the initial form $\mathrm{In}_P(a_{p^e})$ 
is \emph{not} a $p^e$-th power. 
\end{enumerate} 
Similarly, we say $h$ is well-adapted at $\xi_{H_x}$, where $\xi_{H_x}$ is the generic point of the hypersurface $H_x = \{x = 0\}$ in $E_{\text{young}}$, if one of the following two conditions holds: 
 
\begin{enumerate} 
\item[A.] $\mathrm{Slope}_{h,(x,y,z)}(\xi_{H_x}) = \mu(\xi_{H_x}) = \alpha/a$. 
\item[B.] $\mathrm{Slope}_{h,(x,y,z)}(\xi_{H_x}) 
= {\mathrm{ord}_{\xi_{H_x}}(a_{p^e})}/{p^e} <\mu(\xi_{H_x})$ 
and the initial form $\mathrm{In}_{\xi_{H_x}}(a_{p^e})$ 
is \emph{not} a $p^e$-th power. 
\end{enumerate} 
 
The notion of $h$ being well-adapted at $\xi_{H_y}$, where $\xi_{H_y}$ is the generic point of the hypersurface $H_y = \{y = 0\}$ in $E_{\text{young}}$, is defined in an identical manner. 
 
Note that if 
$$a_{p^e} = \sum_{k + l \geq d} c_{kl}x^ky^l = x^r\left\{g(y) + x \cdot \omega(x,y)\right\}$$ 
where $c_{kl} \in k$, $0 \neq g(y) \in k[[y]]$, 
$\omega(x,y) \in k[[x,y]]$, $d = \mathrm{ord}_P(a_{p^e})$ and 
$r = \mathrm{ord}_{\xi_{H_x}}(a_{p^e})$, then 
$$ 
\mathrm{In}_P(a_{p^e}) = \sum_{k + l = d}c_{kl}x^ky^l 
\quad\text{and}\quad 
\mathrm{In}_{\xi_{H_x}}(a_{p^e}) = x^rg(y). 
$$ 
\end{defn} 
\begin{proposition} \ 
 
\indent{\rm (1)}\quad 
There exist an LGS ${\mathbb H} = \{(h,p^e)\}$ and a regular system of parameters $(x,y,z)$, as described in \fbox{SITUATION} in \S 5.1, such that $h$ is well-adapted at $P$, $\xi_{H_x}$ and $\xi_{H_y}$ simultaneously with respect to $(x,y,z)$.  Note that we require the property that the idealistic filtration ${\mathcal R}_P$ is $\left\{\frac{\partial^n}{\partial z^n} \mid n \in {\mathbb Z}_{\geq 0}\right\}$-saturated with respect to the regular system of parameters $(x,y,z)$ (cf. \fbox{SITUATION} (3)). 
 
\indent{\rm (2)}\quad 
If $h$ is well-adapted at $* = P, \xi_{H_x}$ or $\xi_{H_y}$ with respect to $(x,y,z)$, then $\mathrm{Slope}_{h,(x,y,z)}(*)$ is independent of the choice of $h$ and $(x,y,z)$. 
\end{proposition} 
\begin{proof} \ 
 
\indent{\rm (1)}\quad 
We start with $h$ and $(x,y,z)$ 
as given in \fbox{SITUATION} in \S 5.1. 
 
\medskip 
 
\indent{\rm Step 1}.\quad 
Modify $h$ and $(x,y,z)$ to be well-adapted at $P$. 
 
Suppose we are in Case A or Case B as described in Definition \ref{DefWA}. 
Then $h$ is already well-adapted at $P$ with respect to $(x,y,z)$ and there is no modification needed. 
 
Therefore, we may assume that we are not in 
either Case A or Case B.  That is to say, we have 
$$ 
\mathrm{Slope}_{h,(x,y,z)}(P) 
= \frac1{p^e}{\mathrm{ord}_P(a_{p^e})} < \mu(P), 
$$ 
and 
$$ 
\mathrm{In}_P(a_{p^e}) 
= \sum_{k + l = d}c_{kl}x^ky^l, 
\quad \text{with}\ 
d = \mathrm{ord}_P(a_{p^e}),\ \text{is a $p^e$-th power}. 
$$ 
Take 
$$\left\{\mathrm{In}_P(a_{p^e})\right\}^{1/p^e} \in k[x,y],$$ 
and set 
$$z' = z + \left\{\mathrm{In}_P(a_{p^e})\right\}^{1/p^e}, 
\quad\text{i.e.,}\ \  
z = z' - \left\{\mathrm{In}_P(a_{p^e})\right\}^{1/p^e}.$$ 
Plug this into 
$$h = z^{p^e} + a_1z^{p^e-1} + a_2z^{p^e-2} + \cdots + a_{p^e-1}z + a_{p^e}$$ 
to obtain 
$$h = z'^{p^e} + a'_1z'^{p^e-1} + a'_2z'^{p^e-2} + \cdots + a'_{p^e-1}z' + a'_{p^e}$$ 
with 
$$a'_i \in k[[x,y]] \quad\text{for}\quad i = 1, \ldots, p^e - 1, p^e.$$ 
Since for $i = 1, \ldots, p^e - 1$, we have (cf. Remark 6 (ii)) 
$$\frac1{i} 
{\mathrm{ord}_P(a_i)} 
\geq \mu(P) > 
\frac1{p^e} 
{\mathrm{ord}_P(a_{p^e})} 
= \mathrm{ord}_P\left(\left\{\mathrm{In}_P(a_{p^e})\right\}^{1/p^e}\right),$$ 
we conclude that $a'_{p^e}$ is of the form 
$$a'_{p^e} = a_{p^e} - \mathrm{In}_P(a_{p^e}) + \text{higher terms}$$ 
and hence that 
$$\frac1{p^e}{\mathrm{ord}_P(a'_{p^e})} > 
\frac1{p^e}{\mathrm{ord}_P(a_{p^e})}.$$ 
 
We go back to the starting point, replacing the original $h$ and $(x,y,z)$ by $h'$ and $(x',y',z') = (x,y,z')$, with strictly increased 
${\mathrm{ord}_P(a'_{p^e})}/{p^e}$. 
Since $\mu(P) < \infty$, we conclude that, after finitely 
many repetitions of this process, we have to come to the situation 
where we are in Case A or Case B, i.e., where $h$ 
is well-adapted at $P$ with respect to $(x,y,z)$. 
 
\medskip 
 
\indent{\rm Step 2}.\quad 
Modify $h$ and $(x,y,z)$ further to be well-adapted at $\xi_{H_x}$ without destroying the well-adaptedness at $P$. 
 
Take $h$ which is well-adapted at $P$ with respect to $(x,y,z)$, as obtained through Step 1. 
 
Suppose we are in Case A or Case B as described in the second half of Definition \ref{DefWA}.  Then $h$ is already well-adapted at $\xi_{H_x}$ with respect to $(x,y,z)$ and there is no modification needed. 
 
Therefore, we may assume that we are not in 
either Case A or Case B.  That is to say, we have 
$$ 
\mathrm{Slope}_{h,(x,y,z)}(\xi_{H_x}) 
= \frac1{p^e}{\mathrm{ord}_{\xi_{H_x}}(a_{p^e})} < 
\mu(\xi_{H_x}),$$ 
and 
$$ 
\mathrm{In}_{\xi_{H_x}}(a_{p^e}) 
= x^r g(y) \text{ is a }p^e\text{-th power}, 
$$ 
where $a_{p^e} =  x^r \left\{g(y) + x \cdot \omega(x,y)\right\}$ 
and $r = \mathrm{ord}_{\xi_{H_x}}(a_{p^e})$. 
Take 
$$\left\{\mathrm{In}_{\xi_{H_x}}(a_{p^e})\right\}^{1/p^e} \in k[[y]][x],$$ 
and set 
$$z' = z + \left\{\mathrm{In}_{\xi_{H_x}}(a_{p^e})\right\}^{1/p^e}, 
\quad\text{i.e.,}\ \ 
z = z' - \left\{\mathrm{In}_{\xi_{H_x}}(a_{p^e})\right\}^{1/p^e}.$$ 
Then as in Step 1, we see 
$$\frac1{p^e}{\mathrm{ord}_{\xi_{H_x}}(a'_{p^e})} > 
\frac1{p^e}{\mathrm{ord}_{\xi_{H_x}}(a_{p^e})}.$$ 
We go back to the starting point, replacing the original $h$ and $(x,y,z)$ by $h'$ and $(x',y',z') = (x,y,z')$, with strictly increased 
${\mathrm{ord}_{\xi_{H_x}}(a'_{p^e})}/{p^e}$. 
Since $\mu(\xi_{H_x}) < \infty$, we conclude that, 
after finitely many repetitions of this process, we have to come to 
the situation where we are in Case A or Case B, i.e., where $h$ 
is well-adapted at $\xi_{H_x}$ with respect to $(x,y,z)$. 
 
The only issue here is to check, in the process, the property that $h$ is well-adapted at $P$ is preserved. 
 
\medskip 
 
\noindent\emph{Case$\colon 
{\mathrm{ord}_P(a_{p^e})}/{p^e} \geq \mu(P)$}. 
 
In this case, we have 
$$\mathrm{ord}_P\left(\left\{\mathrm{In}_{\xi_{H_x}}(a_{p^e})\right\}^{1/p^e}\right) 
\geq \frac1{p^e}{\mathrm{ord}_P(a_{p^e})} \geq \mu(P)$$ 
and, for $i = 1, \ldots, p^e - 1$, we have (cf. Remark 6 (ii)) 
$$\frac1{i}{\mathrm{ord}_P(a_i)} \geq \mu(P).$$ 
Therefore, we conclude that 
$$\frac1{p^e}{\mathrm{ord}_P(a'_{p^e})} \geq \mu(P),$$ 
and hence that $h'$ stays well-adapted at $P$ with respect to $(x',y',z')$. 
 
\smallskip 
 
\noindent\emph{Case$\colon 
{\mathrm{ord}_P(a_{p^e})}/{p^e} < \mu(P)$}. 
 
In this case, we have 
$$\mathrm{ord}_P 
\left(\left\{\mathrm{In}_{\xi_{H_x}}(a_{p^e})\right\}^{1/p^e}\right) 
\geq \frac1{p^e}{\mathrm{ord}_P(a_{p^e})}$$ 
and, for $i = 1, \ldots, p^e - 1$, we have (cf. Remark 6 (ii)) 
$$\frac1{i}{\mathrm{ord}_P(a_i)} \geq \mu(P) > 
\frac1{p^e}{\mathrm{ord}_P(a_{p^e})}.$$ 
Hence, since $\mathrm{In}_P(a_{p^e})$ is not a $p^e$-th power and since the degree $d = \mathrm{ord}_p(a_{p^e})$-part of $\mathrm{In}_{\xi_{H_x}}(a_{p^e})$ is a $p^e$-th power, we have 
$$\mathrm{In}_P(a'_{p^e}) = \mathrm{In}_P(a_{p^e}) - \left\{\text{the degree }d = \mathrm{ord}_p(a_{p^e})\text{-part of }\mathrm{In}_{\xi_{H_x}}(a_{p^e})\right\}.$$ 
Therefore, we conclude that 
$$\mathrm{ord}_P(a'_{p^e}) = \mathrm{ord}_P\left(\mathrm{In}_P(a'_{p^e})\right) = \mathrm{ord}_P\left(\mathrm{In}_P(a_{p^e})\right) = \mathrm{ord}_P(a_{p^e}) < \mu(P)$$ 
and $\mathrm{In}_P(a'_{p^e})$ is not a $p^e$-th power, and hence that $h'$ stays well-adapted at $P$ with respect to $(x',y',z')$. 
 
\medskip 
 
\indent{\rm Step 3}.\quad 
Modify $h$ and $(x,y,z)$ still further to be well-adapted at $\xi_{H_y}$ without destroying the well-adaptedness at $P$ and $\xi_{H_x}$. 
 
The process of this step is almost identical to that of Step 2, and hence is left to the reader as an exercise. 
 
\smallskip 
 
We note that the requirement as described in \fbox{SITUATION} (3) is met, since the original $(x,y,z)$ satisfies the property and since we only modify $z$ by adding the elements in $k[[x,y]]$ throughout the process. 
 
\smallskip 
 
This finishes the proof of (1). 
 
\medskip 
 
\indent{\rm (2)}\quad 
We only give a proof for the case where $* = P$, since the proof for the case where $* = \xi_{H_x}$ is identical. 
 
Take $h$ which is well-adapted at $P$ with respect to $(x,y,z)$, with the property that the idealistic filtration ${\mathcal R}_P$ is 
saturated for $\{{\partial^n}/{\partial z^n} \mid n \in {\mathbb Z}_{\geq 0}\}$. 
 
We set 
$$ 
\mathrm{H}(P)= 
\min\left\{\mu(P),\ 
\max\left\{\frac1{p^e}{\mathrm{ord}_P(h'|_{Z'})} 
\, \left|\, 
\begin{array}{l} 
h', (x',y',z'), \\ 
Z' = \{z' = 0\} 
\end{array} 
\right\}\right. 
\right\}$$ 
where, computing the above ``$\max$'', we let $h'$ and $(x',y',z')$ vary among all such pairs consisting of the unique element in an LGS and a regular system of parameters that satisfy the condition 
$$h' \equiv {z'}^{p^e} \mathrm{mod}\ \widehat{{\mathfrak m}_P}^{p^e + 1}.$$ 
 
It suffices to show 
$$\mathrm{Slope}_{h,(x,y,z)}(P) = \mathrm{H}(P),$$ 
since the number $\mathrm{H}(P)$ 
is obviously independent of the choice of $h$ and $(x,y,z)$. 
 
\smallskip 
 
Observe 
$$\mathrm{Slope}_{h,(x,y,z)}(P) = 
\min\left\{ 
\frac1{p^e}{\mathrm{ord}_P(a_{p^e})}\ ,\mu(P)\right\} \leq \mathrm{H}(P),$$ 
since 
$$\frac1{p^e}{\mathrm{ord}_P(h|_Z)} 
= \frac1{p^e}{\mathrm{ord}_P(a_{p^e})}$$ 
and since $h$ and $(x,y,z)$ form such a pair consisting of the unique element in an LGS and a regular system of parameters that satisfy the condition 
$$h \equiv z^{p^e} \mathrm{mod}\ \widehat{{\mathfrak m}_P}^{p^e + 1}.$$ 
 
Now we prove the inequality in the opposite direction 
$$\mathrm{Slope}_{h,(x,y,z)}(P) = \min\left\{ 
\frac1{p^e}{\mathrm{ord}_P(a_{p^e})},\ \mu(P)\right\} \geq \mathrm{H}(P).$$ 
If $\mathrm{Slope}_{h,(x,y,z)}(P) = \mu(P)$, 
then the above inequality obviously holds.  Therefore, we may assume that 
$\mathrm{Slope}_{h,(x,y,z)}(P) = {\mathrm{ord}_P(a_{p^e})}/{p^e} < \mu(P)$ and that $\mathrm{In}_P(a_{p^e})$ is not a $p^e$-th power. 
 
Take an arbitrary pair $h'$ and $(x',y',z')$ as described in the definition of $H(P)$ above. 
 
We claim that 
$${\mathrm{ord}_P(h'|_{Z'})} 
= {\mathrm{ord}_P(h|_{Z'})} 
\leq {\mathrm{ord}_P(h|_Z)} 
= {\mathrm{ord}_P(a_{p^e})} < p^e\mu(P),$$ 
which implies the required inequality. 
 
Let 
$$h' = \sum c_BH^B = \sum_{b \in {\mathbb Z}_{\geq 0}} c_bh^b$$ 
be the power series expansion of $h'$ with respect to the LGS ${\mathbb H} = \{(h,p^e)\}$ and its associated regular system of parameters $(x,y,z)$. 
 
Since 
$$h' \equiv z'^{p^e} \equiv c \cdot z^{p^e}\ \mathrm{mod}\ 
\widehat{{\mathfrak m}_P}^{p^e + 1} \quad\text{for some} 
\quad c \in k^{\times},$$ 
we conclude 
$$h' = \sum_{b > 0}c_bh^b + c_0 = u \cdot h + c_0 $$ 
for some unit $u$ in $\widehat{{\mathcal O}_{W,P}}$. 
Moreover, by the formal coefficient lemma, we have 
$(c_0,p^e) \in \widehat{{\mathcal R}_P}$.  This implies that 
$({\mathbb M}_u)^{p^e} \mid c_0$ and hence that 
${\mathrm{ord}_P(c_0|_{Z'})} 
\geq {\mathrm{ord}_P(c_0)} 
\geq p^e\mu(P)$. 
 
Therefore, it suffices to prove 
$${\mathrm{ord}_P(h|_{Z'})} 
\leq {\mathrm{ord}_P(h|_Z)} 
\left(= {\mathrm{ord}_P(a_{p^e})} < {p^e}\mu(P)\right).$$ 
(Then 
${\mathrm{ord}_P(h|_{Z'}\!)} 
\!=\! {\mathrm{ord}_P((u \cdot h)|_{Z'}\!)} 
\!=\! {\mathrm{ord}_P((u \cdot h + c_0)|_{Z'}\!)} 
\!=\! {\mathrm{ord}_P(h'|_{Z'}\!)}$.) 
 
\smallskip 
 
Now by the Weierstrass Preparation Theorem, we have 
$$z' = v \cdot (z + w)$$ 
for some unit $v$ in $\widehat{{\mathcal O}_{W,P}}$ 
and $w \in k[[x,y]]$. 
Since $Z' = \{z' = 0\} = \{z + w = 0\}$, by replacing 
$z'$ with $z + w$, we may assume that $z'$ is of the form 
$$z' = z + w, \quad\text{i.e.,}\ \ z = z' - w 
\quad\text{with}\quad w \in k[[x,y]].$$ 
Plug this into 
$$h = z^{p^e} + a_1z^{p^e-1} + a_2z^{p^e-2} + \cdots + a_{p^e-1}z + a_{p^e}$$ 
to obtain 
$$h = z'^{p^e} + a'_1z'^{p^e-1} + a'_2z'^{p^e-2} + \cdots + a'_{p^e-1}z' + a'_{p^e}$$ 
with 
$$a'_i \in k[[x,y]] \quad\text{for}\quad i = 1, \ldots, p^e - 1, p^e,$$ 
where 
$$a'_{p^e} = (- w)^{p^e} + a_1(- w)^{p^e - 1} + a_2(- w)^{p^e - 2} 
+ \cdots + a_{p^e - 1}(- w) + a_{p^e}.$$ 
Observe that, since (cf. \fbox{SITUATION} in \S 5.1) 
$$({\mathbb M}_u)^i \mid a_i \quad\text{for}\quad i = 1, \ldots, p^e - 1,$$ 
we have 
$$\frac1{i}{\mathrm{ord}_P(a_i)} 
\geq \mathrm{ord}_P({\mathbb M}_u) = \mu(P) > 
\frac1{p^e}{\mathrm{ord}_P(a_{p^e})} 
\quad\text{for}\quad 1\leq i<p^e.$$ 
 
\smallskip 
 
\noindent\emph{Case$\colon 
\mathrm{ord}_P(w) > {\mathrm{ord}_P(a_{p^e})}/{p^e}$}. 
 
In this case, we have 
$${\mathrm{ord}_P(h|_{Z'})} 
= {\mathrm{ord}_P(a'_{p^e})} 
= {\mathrm{ord}_P(a_{p^e})} 
= {\mathrm{ord}_P(h|_Z)}.$$ 
 
\noindent\emph{Case$\colon 
\mathrm{ord}_P(w) = {\mathrm{ord}_P(a_{p^e})}/{p^e}$}. 
 
In this case, since $\mathrm{In}_P(a_{p^e})$ is not a $p^e$-th power, we see via the observation above that 
$$\mathrm{In}_P(a'_{p^e}) = \mathrm{In}_P((- w)^{p^e}) + \mathrm{In}_P(a_{p^e}) \neq 0$$ 
and 
$${\mathrm{ord}_P(h|_{Z'})} 
= {\mathrm{ord}_P(a'_{p^e})} 
= {\mathrm{ord}_P(a_{p^e})} 
= {\mathrm{ord}_P(h|_Z)}.$$ 
 
\noindent\emph{Case$\colon 
\mathrm{ord}_P(w) < {\mathrm{ord}_P(a_{p^e})}/{p^e}$}. 
 
In this case, we have 
\begin{align*} 
{\mathrm{ord}_P(h|_{Z'})} 
&= {\mathrm{ord}_P(a'_{p^e})} = {\mathrm{ord}_P((- w)^{p^e})} \\ 
&= p^e\mathrm{ord}_p(w) < {\mathrm{ord}_P(a_{p^e})}= {\mathrm{ord}_P(h|_Z)}. 
\end{align*} 
 
Therefore, in all the cases above, we have 
$${\mathrm{ord}_P(h|_{Z'})}\leq {\mathrm{ord}_P(h|_Z)}.$$ 
This completes the proof of Proposition 5. 
\end{proof} 
\begin{defn}[Invariant ``$\mathrm{H}$''] We define the invariant $\mathrm{H}$ by the following formula 
$$\mathrm{H}(*) := \mathrm{Slope}_{h,(x,y,z)}(*)$$ 
where $h$ is well-adapted at $* = P, \xi_{H_x}$, or $\xi_{H_y}$ with respect to $(x,y,z)$.  By Proposition 5, the invariant $\mathrm{H}$ is 
independent of the choice of $h$ and $(x,y,z)$. 
\end{defn} 
\begin{defn}[the \emph{tight} monomial] We define the \emph{tight} 
monomial ${\mathbb M}_t$ by the formula 
$${\mathbb M}_t = x^{h_x}y^{h_y} 
\quad\text{where}\quad 
h_x = \mathrm{H}(\xi_{H_x}), h_y = \mathrm{H}(\xi_{H_y}).$$ 
Recall that the \emph{usual} monomial ${\mathbb M}_u$ is defined by the formula 
$${\mathbb M}_u = x^{\alpha/a}y^{\beta/a} 
\quad\text{where}\quad 
\frac{\alpha}{a} = \mu(\xi_{H_x}), 
\ 
\frac{\beta}{a} = \mu(\xi_{H_y}).$$ 
Note that we have 
${\mathbb M}_t \mid {\mathbb M}_u$, 
that is to say, we have $0 \leq h_x \leq \mu(\xi_{H_x})$ and $0 \leq h_y \leq \mu(\xi_{H_y})$, which follow easily from the definition. 
\end{defn} 
\end{subsection} 
\begin{subsection}{Description of the procedure (in the case $\tau = 1$)} 
 
\medskip 
 
\textbf{Analysis of the singular locus 
$\mathrm{Sing}({\mathcal R})$ at $P$} 
 
\medskip 
 
First we analyze the singular locus $\mathrm{Sing}({\mathcal R})$ of the idealistic filtration ${\mathcal R}$ of i.f.g.\! type at $P$. 
\begin{proposition} We have the following description of the singular locus $\mathrm{Sing}({\mathcal R})$ at $P$, denoted by $\mathrm{Sing}({\mathcal R})_P$, according to the values of 
$h_x = \mathrm{H}(\xi_{H_x})$ and $h_y = \mathrm{H}(\xi_{H_y})$: 
 
$$\mathrm{Sing}({\mathcal R})_P = \left\{\begin{array}{lclcl} 
V(z,x) \cup V(z,y)&\text{ if } &h_x \geq 1 &\text{ and } &h_y \geq 1\\ 
V(z,x) &\text{ if } &h_x \geq 1 &\text{ and } &h_y < 1\\ 
V(z, y) &\text{ if } &h_x < 1 &\text{ and } &h_y \geq 1\\ 
V(z,x,y) = P &\text{ if } &h_x < 1 &\text{ and } &h_y < 1, 
\end{array}\right.$$ 
where ``V'' denotes the vanishing locus and where $(x,y,z)$ is a regular system of parameters at $P$ with respect to which $h$ is well-adapted simultaneously at $P$, $\xi_{H_x}$, and $\xi_{H_y}$. 
\end{proposition} 
\begin{proof} Note first that, since $({\mathbb M},a) = (x^{\alpha}y^{\beta},a) \in \widehat{{\mathcal R}_P}$ 
with $a \in {\mathbb Z}_{> 0}$, we have 
$$\mathrm{Sing}({\mathcal R})_P \subset \{x = 0\} \cup \{y = 0\} = H_x \cup H_y.$$ 
Then the asserted description is a consequence of the following analysis of $\mathrm{Sing}({\mathcal R})_P \cap H_x$ (and that of $\mathrm{Sing}({\mathcal R})_P \cap H_y$, which is identical and hence omitted). 
 
\medskip 
 
\noindent\emph{Case$\colon 
h_x \geq 1$}. 
 
In this case, we have 
 
\indent$\bullet$\quad 
$(h,p^e) \in \widehat{{\mathcal R}_P}$ with $h = z^{p^e} + a_1z^{p^e-1} 
+ \cdots + a_{p^e-1}z + a_{p^e}$ 
being well-adapted both at $P$ and $\xi_{H_x}$, 
 
\indent$\bullet$\quad 
$x \mid a_i$ for $i = 0, \ldots, p^e - 1$, since $\alpha/a \geq h_x \geq 1$ and since ${\mathbb M}_u \mid a_i$ for $i = 0, \ldots, p^e - 1$, 
 
\indent$\bullet$\quad 
$x \mid a_{p^e}$, since $h_x \geq 1$, 
 
\indent$\bullet$\quad 
$h = 0$ on $\mathrm{Sing}({\mathcal R})_P$, 
 
\noindent which imply 
 
\indent$\bullet$\quad 
$z = 0$ on $\mathrm{Sing}({\mathcal R})_P \cap H_x$. 
 
Therefore, we conclude 
$$\mathrm{Sing}({\mathcal R})_P \cap H_x \subset V(z,x).$$ 
 
On the other hand, for $Q \in V(z,x)$, we have 
 
\indent$\bullet$\quad 
$\mathrm{ord}_Q(a_i) \geq \mathrm{ord}_Q(({\mathbb M}_u)^i) \geq \alpha/a \cdot i \geq i$ for $i = 0, \ldots, p^e - 1$, since ${\mathbb M}_u \mid a_i$ for $i = 0, \ldots, p^e - 1$ and since $\alpha/a \geq h_x \geq 1$, 
 
\indent$\bullet$\quad 
$\mathrm{ord}_Q(a_{p^e}) \geq h_x \cdot p^e \geq p^e$, 
since $h_x \geq 1$, 
 
\noindent which imply 
 
\indent$\bullet$\quad 
$\mathrm{ord}_Q(h) \geq p^e$. 
 
Therefore, for any $(f,\lambda) \in \widehat{{\mathcal R}_P}$ with $f = \sum c_bh^b$ being the power series expansion with respect to the LGS ${\mathbb H} = \{(h,p^e)\}$ and its associated regular system of parameters $(x,y,z)$, we have 
$$\mathrm{ord}_Q(f) \geq \lambda$$ 
since 
 
\indent$\bullet$\quad 
$\mathrm{ord}_Q(c_b) \geq \mathrm{ord}_Q(({\mathbb M}_u)^{\lambda - p^e \cdot b}) \geq \alpha/a \cdot (\lambda - p^e \cdot b) \geq \lambda - p^e \cdot b$ for $b$ with $\lambda - p^e \cdot b \geq 0$, since $(c_b,\lambda - p^e \cdot b) \in \widehat{{\mathcal R}_P}$, which follows from the formal coefficient lemma. 
 
Therefore, we conclude 
$$\mathrm{Sing}({\mathcal R})_P \cap H_x \supset V(z,x),$$ 
and hence 
$$\mathrm{Sing}({\mathcal R})_P \cap H_x = V(z,x).$$ 
 
\medskip 
 
\noindent\emph{Case$\colon 
1 > h_x >  0$.} 
 
In this case, we have 
 
\indent$\bullet$\quad 
$(h,p^e) \in \widehat{{\mathcal R}_P}$ with $h = z^{p^e} + a_1z^{p^e-1} 
+ \cdots + a_{p^e-1}z + a_{p^e}$ 
being well-adapted both at $P$ and $\xi_{H_x}$, 
 
\indent$\bullet$\quad 
$x \mid a_i$ for $i = 0, \ldots, p^e - 1$, since $\alpha/a \geq h_x > 0$ and since ${\mathbb M}_u \mid a_i$ for $i = 0, \ldots, p^e - 1$, 
 
\indent$\bullet$\quad 
$x \mid a_{p^e}$, since $h_x > 0$, 
 
\indent$\bullet$\quad 
$h = 0$ on $\mathrm{Sing}({\mathcal R})_P$, 
 
\noindent 
which imply 
 
\indent$\bullet$\quad 
$z = 0$ on $\mathrm{Sing}({\mathcal R})_P \cap H_x$. 
 
\medskip 
 
\noindent\emph{Subcase$\colon 
h_x = \mu(\xi_x)$}. 
\quad 
In this subcase, the inequality $1 > h_x$ and the inclusion $(M,a) \in \widehat{{\mathcal R}_P}$ imply 
 
$$\mathrm{Sing}({\mathcal R})_P \cap H_x \subset V(z,x,y).$$ 
 
\noindent\emph{Subcase$\colon 
h_x = {\mathrm{ord}_{\xi_x}(a_{p^e})}/{p^e}$}. 
\quad 
In this subcase, we have $1 > h_x = {r}/{p^e}$ where 
$r = \mathrm{ord}_{\xi_{H_x}}(a_{p^e})$, and hence 
$a_{p^e} = x^r \cdot \gamma(x,y)$ where $\gamma(x,y)$ is not divisible by $x$.  Therefore, we conclude 
$$\mathrm{Sing}({\mathcal R})_P \cap H_x \subset V(z,x) \cap \{\gamma(x,y) = 0\} = V(z,x,y).$$ 
 
Therefore, in both subcases, we have 
$$\mathrm{Sing}({\mathcal R})_P \cap H_x = V(z,x,y) = P.$$ 
 
\medskip 
 
\noindent\emph{Case$\colon 
h_x = 0$}. 
 
\smallskip 
 
\noindent\emph{Subcase$\colon 
\alpha/a = \mu(\xi_{H_x}) = 0$}.\quad 
In this subcase, we have $\beta/a =  \mu(\xi_{H_y}) \geq 1$. 
Therefore, we conclude 
$$\mathrm{Sing}({\mathcal R})_P \cap H_x \subset V(z,x,y) 
\ \text{and hence}\ 
\mathrm{Sing}({\mathcal R})_P \cap H_x = V(z,x,y) = P.$$ 
 
\smallskip 
 
\noindent\emph{Subcase$\colon 
\alpha/a = \mu(\xi_{H_x}) > 0$}.\quad 
In this subcase, we have 
 
\indent$\bullet$\quad 
$x \mid a_i$ for $i = 0, \ldots, p^e - 1$, since $\alpha/a > 0$ and since ${\mathbb M}_u \mid a_i$ for $i = 0, \ldots, p^e - 1$. 
 
Set $a_{p^e} = g(y) + x \cdot \omega(x,y)$.  Then $g(y) \neq 0$ and $g(y) = \mathrm{In}_{\xi_{H_x}}(a_{p^e})$ is not a $p^e$-th power, since $h$ is well-adapted at $\xi_{H_x}$.  We also observe that, for $Q \in \mathrm{Sing}({\mathcal R})_P \cap H_x$, we have 
$$p^e \leq \mathrm{ord}_Q(h) \leq \mathrm{ord}_Q(h|_{H_x}) = \mathrm{ord}_Q(z^{p^e} + g(y)).$$ 
Therefore, there exists $e' < e$ such that 
$$\frac{\partial^{p^{e'}}}{\partial y^{p^{e'}}}(z^{p^e} + g(y)) = \frac{\partial^{p^{e'}}}{\partial y^{p^{e'}}}g(y) \neq 0$$ 
and that 
$$\mathrm{ord}_Q\left(\frac{\partial^{p^{e'}}}{\partial y^{p^{e'}}}g(y)\right) =  \mathrm{ord}_Q\left(\frac{\partial^{p^{e'}}}{\partial y^{p^{e'}}}(z^{p^e} + g(y))\right)\geq p^e - p^{e'} > 0.$$ 
This implies $y = 0$ at $Q$.  Therefore, we conclude 
$$\mathrm{Sing}({\mathcal R})_P \cap H_x \subset V(z,x,y), 
\ \text{and hence}\ 
\mathrm{Sing}({\mathcal R})_P \cap H_x = V(z,x,y) = P.$$ 
This completes the proof of Proposition 6. 
\end{proof} 
 
\medskip 
 
\textbf{Description of the procedure for 
resolution of singularities} 
 
\medskip 
 
Now based upon the analysis of the support $\mathrm{Sing}({\mathcal R})_P$, we give the following (local) description of the procedure (around the point $P$) for resolution of singularities in the monomial case with $\tau = 1$: 
 
\medskip 
 
\indent{\rm Step 1.}\quad 
Check if $\dim \mathrm{Sing}({\mathcal R})_P = 1$. 
 
If the answer is yes, then blow up the 1-dimensional components 
one by one.  If there are two 1-dimensional components meeting at $P$, then we blow up the one associated to the boundary divisor with bigger $\mu$ first.  If the boundary divisors associated to the two 1-dimensional components 
have the same $\mu$, then we blow up first the one 
associated to the boundary divisor created later in the history. 
Since the invariant $\mu$ strictly decreases under this procedure, 
this step comes to an end after finitely many times with the dimension 
of the singular locus dropping to $0$. 
 
If the answer is no, i.e., $\dim \mathrm{Sing}({\mathcal R})_P = 0$, 
then go to Step 2. 
 
\medskip 
 
\indent{\rm Step 2.}\quad 
Once $\dim \mathrm{Sing}({\mathcal R})_P = 0$, blow up the isolated point in the singular locus.  Then go back to Step 1. 
 
\smallskip 
 
\indent Repeat these steps. 
 
\medskip 
 
(We note that, as long as the value of the invariant $\sigma$ remains the same, we stay in the monomial case during the procedure above.  We also note that in the middle of the procedure the invariant $\sigma$ may drop.  If the latter happens, we are no longer in the monomial case.  In this case, we go through the mechanism described in \S 4.2 with the reduced new value of $\sigma$ to reach the new monomial case.) 
 
\medskip 
 
What remains to be shown is that the above procedure terminates after finitely many repetitions.  This termination of the procedure is the main subject of \S 5.4. 
 
\end{subsection} 
\begin{subsection}{Termination of the procedure (in the case $\tau = 1$)} 
 
\medskip 
 
\textbf{Notion of a good/bad point (resp. hypersurface)} 
 
\medskip 
 
In order to analyze the termination of the procedure, we introduce the following notion of the point $P$ being ``good or bad'' and the boundary divisor $H_x$ (or $H_y$) being ``good or bad''. 
 
\begin{defn}[``good/bad'' point (cf. \cite{BV3}]  We say 
$P$ is a \emph{good} (resp. a \emph{bad}) point if 
$\mu(P)-\mathrm{H}(P) = 0$ (resp. $>0$). 
Similarly, we say 
$H_x$ is a \emph{good} (resp. \emph{bad}) hypersurface if 
$\mu(\xi_{H_x})-\mathrm{H}(\xi_{H_x}) = 0$ 
(resp. $>0$), 
where $\xi_{H_x}$ is the generic point of the hypersurface $H_x$. 
 
The notion of $H_y$ being a good or bad hypersurface is defined in an identical manner. 
\end{defn} 
\begin{lemma} Let $W \overset{\pi}{\leftarrow} \widetilde{W}$ be the blow up with center $P$, $E_P$ the exceptional divisor, $\widetilde{{\mathcal R}}$ the transformation of the idealistic filtration of i.f.g.\! type ${\mathcal R}$.  Then 
$P$ is a good (resp. bad) point if and only if 
$E_P$ is a good (resp. bad) hypersurface. 
\end{lemma} 
\begin{proof} Take $h$ and a regular system of parameters as described 
in \fbox{SITUATION} in \S 5.1.  We may further assume that $h$ is 
well-adapted at $P$ with respect to the regular system of 
parameters $(x,y,z)$ (cf. Proposition 5).  Take a point 
$\widetilde{P} \in \pi^{-1}(P) \cap \mathrm{Supp}(\widetilde{{\mathcal R}}) \subset \widetilde{W}$.  Since the singular locus is empty over the $z$-chart, $\widetilde{P}$ should be 
either in the $x$-chart or in the $y$-chart. 
Say, $\widetilde{P}$ is in the $x$-chart with a regular system of parameters $(\widetilde{x},\widetilde{y},\widetilde{z}) = (x, y/x - c,z/x)$ for some $c \in k$.  We also assume that the invariant $\sigma$ stays the same, and hence (cf. Proposition 4) that $(\widetilde{h},p^e)$ is the unique element in the LGS at $\widetilde{P}$ where 
$$\widetilde{h} 
= \frac{\pi^*(h)}{x^{p^e}} 
= \widetilde{z}^{p^e} + \widetilde{a_1}\widetilde{z}^{p^e - 1} 
+ \cdots + \widetilde{a_{p^e}} 
\quad \text{with}\quad 
\widetilde{a_i} = \frac{\pi^*(a_i)}{x^i} 
\in k[[\widetilde{x},\widetilde{y}]].$$ 
We compute 
\begin{align*} 
\lefteqn{ 
\mathrm{Slope}_{\widetilde{h}, 
(\widetilde{x},\widetilde{y},\widetilde{z})}(\xi_{E_P}) 
= \min\left\{\frac1{p^e}{\mathrm{ord}_{\xi_{E_P}}(\widetilde{a_{p^e}})}, 
\ \mu({\xi_{E_P})}\right\} 
}\\ 
&= \min\left\{\frac1{p^e}{\mathrm{ord}_P(a_{p^e})}- 1, 
\ \mu(P) - 1\right\} 
= \min\left\{\frac1{p^e}{\mathrm{ord}_P(a_{p^e})}, 
\ \mu(P)\right\} - 1\\ 
&= \mathrm{H}(P) - 1. 
\end{align*} 
 
\noindent\emph{Case$\colon 
P$ is a good point, i.e., $\mathrm{H}(P) = \mu(P)$}. 
 
In this case, we have 
$$\mathrm{Slope}_{\widetilde{h}, (\widetilde{x},\widetilde{y},\widetilde{z})}(\xi_{E_P}) = \mathrm{H}(P) - 1 = \mu(P) - 1 = \mu(\xi_{E_P}).$$ 
Therefore, we conclude that $E_P$ is a good hypersurface (with $\widetilde{h}$ well-adapted at $\xi_{E_P}$ with respect to $(\widetilde{x},\widetilde{y},\widetilde{z})$ and $\mathrm{H}(\xi_{E_P}) = \mu(\xi_{E_P})$). 
 
\medskip 
 
\noindent\emph{Case$\colon 
P$ is a bad point, i.e., $\mathrm{H}(P) < \mu(P)$}. 
 
In this case, we have 
$$\mathrm{Slope}_{\widetilde{h}, (\widetilde{x},\widetilde{y},\widetilde{z})}(\xi_{E_P}) = \mathrm{H}(P) - 1 < \mu(P) - 1 = \mu(\xi_{E_P}).$$ 
Therefore, in order to see that $E_P$ is a bad hypersurface, 
i.e., $\mathrm{H}(\xi_{E_P}) < \mu(\xi_{E_P})$, we have only to 
show that $\widetilde{h}$ is well-adapted at $\xi_{E_P}$ 
with respect to $(\widetilde{x},\widetilde{y},\widetilde{z})$, 
i.e., $\mathrm{In}_{\xi_{E_P}}(\widetilde{a_{p^e}})$ is not 
a $p^e$-th power, and hence that 
$\mathrm{Slope}_{\widetilde{h}, (\widetilde{x},\widetilde{y}, 
\widetilde{z})}(\xi_{E_P}) = H(\xi_{E_P})$. 
 
Set $a_{p^e} = \sum_{k + l \geq d} c_{kl}x^ky^l$ where $d = \mathrm{ord}_P(a_{p^e})$.  Then 
\begin{align*} 
\pi^*(a_{p^e}) &= \sum_{k + l \geq d}c_{kl}x^{k + l} 
\left(\frac{y}{x}\right)^l 
= x^d\left\{\sum_{k + l = d}c_{kl} 
\left(\frac{y}{x}\right)^l 
+ x \cdot 
\Omega\left(x,\frac{y}{x}\right) 
\right\} \\ 
&= x^d\left\{\phi\left(\frac{y}{x}\right) + x \cdot 
\Omega\left(x,\frac{y}{x}\right)\right\}, 
\end{align*} 
where $\phi(T) = \sum_{k + l = d}c_{kl}T^l$. 
Hence, we have 
\begin{align*} 
\widetilde{a_{p^e}} &= \frac{\pi^*(a_{p^e})}{x^{p^e}} 
= x^{d - p^e}\left\{ 
\phi\left(\frac{y}{x} - c + c\right) 
+ x \cdot 
\Omega\left(x,\frac{y}{x} - c + c\right) 
\right\} \\ 
&= \widetilde{x}^{d - p^e}\left\{\phi(\widetilde{y} + c) + \widetilde{x} \cdot \Omega(\widetilde{x},\widetilde{y} + c)\right\}.\\ 
\end{align*} 
Therefore, we conclude 
\begin{align*} 
\lefteqn{ 
\mathrm{In}_P(a_{p^e}) = \sum_{k + l = d}c_{kl}x^ky^l = 
x^d\phi\left(\frac{y}{x}\right) 
\text{ is a }p^e\text{-th power} 
} \\ 
&\Longleftrightarrow 
p^e\mid d \text{ and } \phi(T) \text{ is a }p^e\text{-th power} \\ 
&\Longleftrightarrow 
\mathrm{In}_{\xi_{E_P}}(\widetilde{a_{p^e}}) = \widetilde{x}^{d - p^e}\phi(\widetilde{y} + c) \text{ is a }p^e\text{-th power}. 
\end{align*} 
Now since $P$ is a bad point in this case, $\mathrm{In}_P(a_{p^e})$ is not a $p^e$-th power.  Therefore, by the above equivalence, 
$\mathrm{In}_{\xi_{E_P}}(\widetilde{a_{p^e}})$ is not a $p^e$-th power, either. 
 
\smallskip 
 
This completes the proof of Lemma 4. 
\end{proof} 
\begin{remark} \ 
 
\indent{\rm (1)}\quad 
The above proof is slightly sloppy in the sense that, after blow up, we may end up having $\mathrm{ord}_{\widetilde{P}}(\widetilde{a_{p^e}}) = p^e$ 
and $\mathrm{In}_{\widetilde{P}}(\widetilde{a_{p^e}}) 
= c\widetilde{x}^{p^e}$ for some 
$c \in k \setminus \{0\}$, even under the condition  $\widetilde{P} \in \pi^{-1}(P) \cap \mathrm{Supp}(\widetilde{{\mathcal R}}) \subset \widetilde{W}$ and the assumption that the invariant $\sigma$ stays the same (cf. the proof of Proposition 4 (1)).  Then we would have to replace $(\widetilde{x},\widetilde{y},\widetilde{z})$ 
with $(\widetilde{x},\widetilde{y},\widetilde{z}' 
= \widetilde{z} + c^{1/p^e}\widetilde{x})$ 
to guarantee condition (1) in \fbox{SITUATION} in \S 5.1.  Accordingly, we have to analyze $\widetilde{a_{p^e}}'$.  It is straightforward, however, to see that the same statement holds for $\widetilde{a_{p^e}}'$.  The details are left to the reader as an exercise. 
 
\indent{\rm (2)}\quad 
As will be clear in the presentation that follows, especially in the way we classify the configurations and we define the new invariant 
``$\mathrm{inv}_{\mathrm{MON}}$'', the focus of our proof centers around the analysis looking at whether the hypersurface of our concern is good/bad.  The notion of a point being good/bad, though related to our analysis 
via Lemma 4, is somewhat auxiliary. 
\end{remark} 
\textbf{Configurations} 
 
\medskip 
 
Looking at the boundary divisor(s) in $E_{\text{young}}$ 
at the point $P \in \mathrm{Sing}({\mathcal R})$ and seeing 
whether they are good or bad, we come up with the following classification of the ``configurations''.  Note that the pictures depict the configurations in a 2-dimensional manner, taking the intersection with the hypersurface $Z = \{z = 0\}$. 
 
\medskip 
 
\indent{\rm \textcircled{\footnotesize 1}}\quad 
The point $P$ is only on one boundary divisor (in $E_{\text{young}}$), say $H_x$, which is good. 
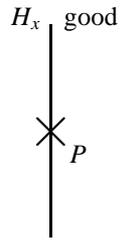
\ZU{ 
 \VCL{$H_x$}{good} 
 \CPO{$P$} 
} 
\indent{\rm \textcircled{\footnotesize 2}}\quad 
The point $P$ is at the intersection of two boundary divisors (in $E_{\text{young}}$), both of which are good. 
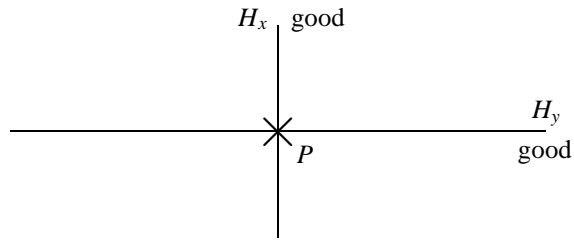
\ZU{ 
 \VCL{$H_x$}{good} 
 \HOL{$H_y$}{good} 
 \CPO{$P$} 
} 
\indent{\rm \textcircled{\footnotesize 3}}\quad 
The point $P$ is only on one boundary divisor (in $E_{\text{young}}$), say $H_x$, which is bad. 
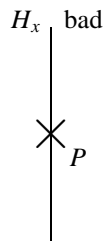
\ZU{ 
 \VCL{$H_x$}{bad} 
 \CPO{$P$} 
} 
\indent{\rm \textcircled{\footnotesize 4}}\quad 
The point $P$ is at the intersection of two boundary divisors (in $E_{\text{young}}$), one of which, say, $H_x$, is bad, while the other, say $H_y$, is good. 
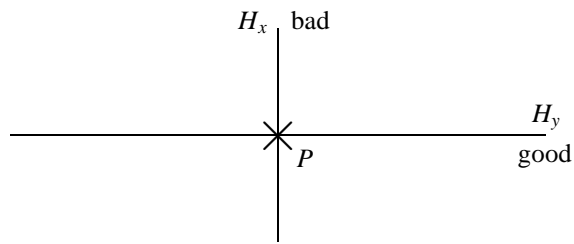
\ZU{ 
 \VCL{$H_x$}{bad} 
 \HOL{$H_y$}{good} 
 \CPO{$P$} 
} 
\indent{\rm \textcircled{\footnotesize 5}}\quad 
The point $P$ is at the intersection of two boundary divisors (in $E_{\text{young}}$), say $H_x$ and $H_y$, both of which are bad. 
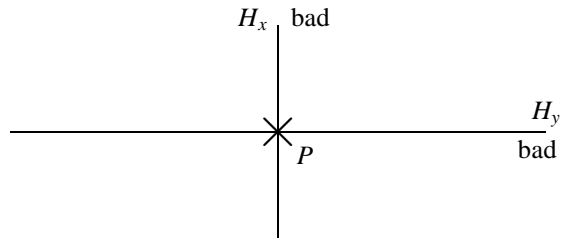
\ZU{ 
 \VCL{$H_x$}{bad} 
 \HOL{$H_y$}{bad} 
 \CPO{$P$} 
} 
 
\noindent\fbox{\textbf{Basic strategy to show termination of the procedure}} 
 
\smallskip 
 
After the blow up $W \overset{\pi}\leftarrow \widetilde{W}$ specified by the procedure described in \S 5.3, we show that, at $\widetilde{P} \in \pi^{-1}(P) \in \widetilde{W}$, one of the following holds: 
 
\medskip 
 
\indent$\bullet$\quad 
$\widetilde{P} \not\in \mathrm{Sing}(\widetilde{R})$, i.e., 
$\mathrm{Sing}(\widetilde{R}) = \emptyset$ in a neighborhood 
of $\widetilde{P}$, 
 
\indent$\bullet$\quad 
$\widetilde{P} \in \mathrm{Sing}(\widetilde{R})$ and the invariant $\sigma$ drops, or 
 
\indent$\bullet$\quad 
$\widetilde{P} \in \mathrm{Sing}(\widetilde{R})$ and the invariant ``$\mathrm{inv}_{\mathrm{MON}}(P)$'' strictly decreases, i.e., 
$\mathrm{inv}_{\mathrm{MON}}(P) > \mathrm{inv}_{\mathrm{MON}}(\widetilde{P})$.  We note that ``$\mathrm{inv}_{\mathrm{MON}}(P)$'' is a new invariant, which we introduce below, attached to the point $P \in \mathrm{Sing}({\mathcal R})$ in any one of the configurations \textcircled{\footnotesize 1} through \textcircled{\footnotesize 5}. 
 
\medskip 
 
Since the invariant ``$\mathrm{inv}_{\mathrm{MON}}$'' can not decrease infinitely many times, the procedure described in \S 5.3 terminates either with $\widetilde{P} \not\in \mathrm{Supp}(\widetilde{R})$ or with the drop of the invariant $\sigma$ after finitely many repetitions. 
 
This completes the description of the basic strategy. 
 
\begin{defn}[Invariant ``$\mathrm{inv}_{\mathrm{MON}}$''] 
We define the invariant 
``$\mathrm{inv}_{\mathrm{MON}}(P)$'' associated to a point $P \in \mathrm{Sing}(\widetilde{R})$ in each of the configurations \textcircled{\footnotesize 1} through \textcircled{\footnotesize 5} (cf. \textbf{Configurations}) as follows (Note that ``$\mathrm{MON}$'' is short for ``$\mathrm{MONOMIAL}$''.): 
$$\mathrm{inv}_{\mathrm{MON}}(P) = 
\begin{cases} 
(0, 0, \mu_x) & \text{in configuration \textcircled{\footnotesize 1}}, \\ 
(0, 0, \min\{\mu_x,\mu_y\}, \max\{\mu_x,\mu_y\}) & \text{in configuration \textcircled{\footnotesize 2}}, \\ 
(\rho_x, 0, \mu_x) & \text{in configuration \textcircled{\footnotesize 3}},\\ 
(\min\{\rho_x,\mu_x\}, \max\{\rho_x,\mu_x\}) & \text{in configuration \textcircled{\footnotesize 4}},\\ 
(\min\{\rho_x,\rho_y\}, \max\{\rho_x,\rho_y\}) & \text{in configuration \textcircled{\footnotesize 5}}, 
\end{cases} 
$$ 
where $\mu_x = \mu(\xi_{H_x}), \mu_y = \mu(\xi_{H_y})$ and the invariant $\rho$ is defined as below to determine $\rho_x, \rho_y$. 
\end{defn} 
\begin{proposition}Let $P$ be the point in configuration 
\textcircled{\footnotesize 3}, \textcircled{\footnotesize 4}, 
or \textcircled{\footnotesize 5}, and let $H_x$ be 
a bad boundary divisor (in $E_{\text{young}}$). 
Suppose $h$ is well-adapted at $\xi_{H_x}$ 
with respect to $(x,y,z)$ (satisfying the conditions 
as described in \fbox{SITUATION} in \S 5.1 at the same time).  Write 
$$a_{p^e} = x^r \left\{g(y) + x \cdot \omega(x,y)\right\} 
\quad\text{where}\quad 
r = \mathrm{ord}_{\xi_{H_x}}(a_{p^e}).$$ 
Set 
\begin{align*} 
\rho_{h,(x,y,z),H_x}(P) &= 
\frac1{p^e} 
\{\mathrm{res\text{-}ord}^{(p^e)}_P(\mathrm{In}_{\xi_{H_x}}(a_{p^e})) 
- \mathrm{ord}_{\xi_{H_x}}(a_{p^e})\} \\ 
&=\frac1{p^e} 
\{\mathrm{res\text{-}ord}^{(p^e)}_P\left(x^rg(y)\right) - r\}, 
\end{align*} 
i.e., 
$$\rho_{h,(x,y,z),H_x}(P) =\begin{cases} 
\mathrm{ord}_P\left(g(y)\right)/p^e 
& \text{in case }r \not\equiv 0 \text{ mod }\ p^e \\ 
\mathrm{res\text{-}ord}^{(p^e)}_P\left(g(y)\right) /p^e 
& \text{in case }r \equiv 0 \text{ mod }\ p^e, 
\end{cases} 
$$ 
where $\mathrm{res\text{-}ord}^{(p^e)}_P$ is the smallest degree of the term which appears with a nonzero coefficient and which is not a $p^e$-th power. 
 
Then $\rho_{h,(x,y,z),H_x}(P)$ is independent of the choice of 
$h$ and $(x,y,z)$. 
\end{proposition} 
\begin{proof} Note first that $r = \mathrm{H}(\xi_{H_x}) \cdot p^e$ is independent of the choice of $h$ and $(x,y,z)$ by Proposition 5 (2). 
 
We set 
$$\rho_x = \max\left\{ \frac1{p^e} 
{\mathrm{ord}_P\left(\left\{(h'|_{Z'}) \cdot {x'}^{- r}\right\}|_{Z' \cap H_{x'}}\right)} 
\ \left|\ 
\begin{array}{l} 
h', (x',y',z'), \\ 
Z' = \{z' = 0\}, \\ 
H_{x'} = \{x' = 0\} \\ 
\phantom{H_{x'}} = H_x 
\end{array}\right\}\right.,$$ 
where, computing the above ``$\max$'', we let $h'$ and $(x',y',z')$ vary among all such pairs consisting of the unique element in an LGS and a regular system of parameters that satisfy the condition 
$$\left\{ 
\begin{array}{lcl} 
h' &\equiv& {z'}^{p^e}\ \mathrm{ mod }\ \widehat{{\mathfrak m}_P}^{p^e + 1}, \text{ and}\\ 
\mathrm{ord}_{\xi_{H_{x'}}}(h'|_{Z'}) &=& \mathrm{ord}_{\xi_{H_x}}(h'|_{Z'}) = r. 
\end{array}\right.$$ 
It suffices to show that 
$$\rho_{h,(x,y,z),H_x}(P) = \rho_x,$$ 
as the number $\rho_x$ is obviously independent of the choice of $h$ and $(x,y,z)$. 
 
\smallskip 
 
Firstly we claim that the inequality 
$$\rho_{h,(x,y,z),H_x}(P) \leq \rho_x$$ 
holds.  Note that $h$ and $(x,y,z)$ form such a pair satisfying the above 
condition described in the definition of $\rho_x$ computing the ``$\max$'', 
since $h$ is well-adapted at $\xi_{H_x}$ with respect to $(x,y,z)$. 
 
\smallskip 
 
\noindent\emph{Case$\colon 
r \not\equiv 0\ \mathrm{ mod }\ p^e$}. 
 
In this case, we have 
$$\rho_{h,(x,y,z),H_x}(P) 
= \frac1{p^e}{\mathrm{ord}_P(g(y))} 
= \frac1{p^e}{\mathrm{ord}_P\left(\left\{(h|_{Z}) 
\cdot x^{- r}\right\}|_{Z \cap H_x}\right)} \leq \rho_x$$ 
by the definition of $\rho_x$ above taking the ``max'' among all such pairs. 
 
\smallskip 
 
\noindent\emph{Case$\colon 
r \equiv 0\ \mathrm{ mod }\ p^e$}. 
 
\smallskip 
 
In this case, we modify $z$ in the following way. 
 
Set $g(y) = \sum_{n \in {\mathbb Z}_{\geq 0}}b_ny^n$ with $b_n \in k$. 
We have 
$$ 
x^r \left\{ 
\sum_{n < 
\mathrm{res\text{-}ord}_P^{(p^e)}(g(y))}b_ny^n\right\} = w^{p^e} 
\quad\text{for some}\quad w \in k[x,y], 
$$ 
since L.H.S. is a $p^e$-th power by the case assumption $r \equiv 0\ \mathrm{ mod }\ p^e$ and by the definition of $ \mathrm{res\text{-}ord}_P^{(p^e)}$. 
Then set $z' = z + w$, i.e., $z = z' - w$. 
 
Plug this into 
$$h = z^{p^e} + a_1z^{p^e-1} + a_2z^{p^e-2} + \cdots + a_{p^e-1}z + a_{p^e}$$ 
to obtain 
$$h = z'^{p^e} + a'_1z'^{p^e-1} + a'_2z'^{p^e-2} + \cdots 
+ a'_{p^e-1}z' + a'_{p^e}$$ 
with 
$a'_i \in k[[x,y]]$ for $i = 1, \ldots, p^e - 1, p^e,$ 
where 
$$a'_{p^e} = (- w)^{p^e} + a_1(- w)^{p^e - 1} + a_2(- w)^{p^e - 2} 
+ \cdots + a_{p^e}.$$ 
Observe that, since (cf. \fbox{SITUATION} in \S 5.1) 
$$({\mathbb M}_u)^i \mid a_i \quad\text{for}\quad i = 1, \ldots, p^e - 1,$$ 
we have 
$$\frac1{i}{\mathrm{ord}_{\xi_{H_x}}(a_i)} 
\geq \mathrm{ord}_{\xi_{H_x}}({\mathbb M}_u) = \mu(\xi_{H_x}) > 
\frac1{p^e}{\mathrm{ord}_{\xi_{H_x}}(a_{p^e})} 
= \frac{r}{p^e}$$ 
for $i = 1, \ldots, p^e - 1$, 
where the second strict inequality follows from the assumption that $H_x$ is a bad boundary divisor. 
 
Hence, we observe that $a'_{p^e}$ is of the following form 
$$a'_{p^e} = x^r \left\{g'(y) + x \cdot \omega'(x,y)\right\}$$ 
where 
$$g'(y) = g(y) - w^{p^e} \cdot x^{-r} 
= \sum_{n \geq \mathrm{res\text{-}ord}_P^{(p^e)}(g(y))}b_ny^n.$$ 
Therefore, we conclude that 
\begin{align*} 
\rho_{h,(x,y,z),H_x}(P) 
&= \frac1{p^e} 
{\mathrm{res\text{-}ord}_P^{(p^e)}(g(y))} 
= \frac1{p^e}{\mathrm{ord}_P(g'(y))} 
\\ 
&=\frac1{p^e}{\mathrm{ord}_P\left(\left\{(h|_{Z'}) \cdot 
x^{- r}\right\}|_{Z' \cap H_x}\right)} 
\leq \rho_x. 
\end{align*} 
Secondly we prove the inequality in the opposite direction holds, i.e., 
$$\rho_{h,(x,y,z),H_x}(P) \geq \rho_x.$$ 
 
Take an arbitrary pair of $h'$ and $(x',y',z')$ 
satisfying the conditions described in the definition of $\rho_x$ 
computing the ``$\max$''. 
 
We claim that 
$$\rho_{h,(x,y,z),H_x}(P) \geq 
\frac1{p^e}{\mathrm{ord}_P\left(\left\{(h'|_{Z'}) 
\cdot {x'}^{-r}\right\}|_{Z' \cap H_{x'}}\right)},$$ 
which implies the required inequality. 
 
Let 
$$h' = \sum c_BH^B = \sum_{b \in {\mathbb Z}_{\geq 0}} c_bh^b$$ 
be the power series expansion of $h'$ with respect to the LGS ${\mathbb H} = \{(h,p^e)\}$ and its associated regular system of parameters $(x,y,z)$. 
 
Since 
$$h' \equiv z'^{p^e} \equiv c \cdot z^{p^e}\ \mathrm{mod}\ \widehat{{\mathfrak m}_P}^{p^e + 1} 
\quad\text{for some}\quad c \in k^{\times},$$ 
we conclude 
$$h' = \sum_{b > 0}c_bh^b + c_0 = u \cdot h + c_0 $$ 
for some unit $u$ in $\widehat{{\mathcal O}_{W,P}}$. 
 
Moreover, by the formal coefficient lemma, we have 
$(c_0,p^e) \in \widehat{{\mathcal R}_P}$.  This implies 
that $({\mathbb M}_u)^{p^e} \mid c_0$, and hence that 
$${\mathrm{ord}_{\xi_{H_x}}(c_0|_{Z'})} 
\geq {\mathrm{ord}_{\xi_{H_x}}(c_0)} 
\geq {p^e} \mu(\xi_{H_x}) > {r}.$$ 
Hence we have 
$$\left\{(c_0|_{Z'}) \cdot {x'}^{-r}\right\}|_{Z' \cap H_{x'}} = \left\{(c_0|_{Z'}) \cdot x^{-r}\right\}|_{Z' \cap H_x}= 0,$$ 
which implies 
\begin{align*} 
&{\mathrm{ord}_P\!\left(\left\{\!(h'|_{Z'}) 
\cdot {x'}^{-r}\!\right\}|_{Z' \cap H_{x'}}\!\right)} 
= {\mathrm{ord}_P\!\left(\left\{\!\left((u \cdot h + c_0)|_{Z'}\right) 
\cdot {x'}^{-r}\!\right\}|_{Z' \cap H_{x'}}\!\right)} 
\\ 
&\quad= {\mathrm{ord}_P\!\left(\left\{(u \cdot h|_{Z'}) \cdot {x'}^{-r}\right\}|_{Z' \cap H_{x'}}\right)} 
= {\mathrm{ord}_P\!\left(\left\{(h|_{Z'}) \cdot {x'}^{-r}\right\}|_{Z' \cap H_{x'}}\right)} \\ 
&\quad= {\mathrm{ord}_P\left(\left\{(h|_{Z'}) \cdot x^{-r}\right\}|_{Z' \cap H_x}\right)}. 
\end{align*} 
Therefore, it suffices to prove 
$$\rho_{h,(x,y,z),H_x}(P) \geq 
\frac1{p^e}{\mathrm{ord}_P\left(\left\{(h|_{Z'}) 
\cdot x^{-r}\right\}|_{Z' \cap H_x}\right)}.$$ 
Now by the Weierstrass Preparation Theorem, we have 
$$z' = v \cdot (z + w)$$ 
for some unit $v$ in $\widehat{{\mathcal O}_{W,P}}$ and $w \in k[[x,y]]$. 
Since $Z' = \{z' = 0\} = \{z + w = 0\}$, by replacing $z'$ with $z + w$, 
we may assume that $z'$ is of the form 
$$z' = z + w, \quad\text{i.e.,}\ \  z = z' - w 
\quad\text{with}\quad w \in k[[x,y]].$$ 
Plug this into 
$$h = z^{p^e} + a_1z^{p^e-1} + a_2z^{p^e-2} 
+ \cdots + a_{p^e-1}z + a_{p^e}$$ 
to obtain 
$$h = z'^{p^e} + a'_1z'^{p^e-1} + a'_2z'^{p^e-2} + \cdots 
+ a'_{p^e-1}z' + a'_{p^e}$$ 
with 
$a'_i \in k[[x,y]]$ for $i = 1, \ldots, p^e - 1, p^e,$ 
where 
$$a'_{p^e} = (- w)^{p^e} + a_1(- w)^{p^e - 1} + a_2(- w)^{p^e - 2} 
+ \cdots + a_{p^e}.$$ 
Observe that, since (cf. \fbox{SITUATION} in \S 5.1) 
$$({\mathbb M}_u)^i \mid a_i \quad\text{for}\quad i = 1, \ldots, p^e - 1,$$ 
we have 
$$\frac1{i}{\mathrm{ord}_{\xi_{H_x}}(a_i)} 
\geq \mathrm{ord}_{\xi_{H_x}}({\mathbb M}_u) = \mu(\xi_{H_x}) > 
\frac1{p^e}{\mathrm{ord}_{\xi_{H_x}}(a_{p^e})} = \frac{r}{p^e}$$ 
for $i = 1, \ldots, p^e - 1$. We claim that 
$$\mathrm{ord}_{\xi_{H_x}}(w) \geq \frac{r}{p^e}.$$ 
In fact, suppose $\mathrm{ord}_{\xi_{H_x}}(w) < {r}/{p^e}$. 
Then using the above observation we would have 
$$\mathrm{ord}_{\xi_{H_x}}(h|_{Z'}) = \mathrm{ord}_{\xi_{H_x}}(a'_{p^e}) 
=  \mathrm{ord}_{\xi_{H_x}}((- w)^{p^e}) < r.$$ 
By the equation $h' = u \cdot h + c_0$ and by the inequality $\mathrm{ord}_{\xi_{H_x}}(c_0|_{Z'}) > r$, this would also imply 
$$\mathrm{ord}_{\xi_{H_x}}(h'|_{Z'}) = \mathrm{ord}_{\xi_{H_x}}(h|_{Z'}) < r,$$ 
which is against the choice of $h'$ and $(x',y',z')$ that we started with, satisfying the conditions described in 
the definition of $\rho_x$ computing the ``$\max$''. 
 
\smallskip 
 
Now we are at the stage to finish the argument to prove the inequality 
$$\rho_{h,(x,y,z),H_x}(P) \geq 
\frac1{p^e}{\mathrm{ord}_P\left(\left\{(h|_{Z'}) \cdot x^{-r}\right\}|_{Z' \cap H_x}\right)}.$$ 
 
\pagebreak[2] 
 
\noindent\emph{Case$\colon 
r \not\equiv 0\ \mathrm{mod}\ p^e$}. 
 
In this case, the claimed inequality $\mathrm{ord}_{\xi_{H_x}}(w) \geq 
{r}/{p^e}$ implies the strict inequality $\mathrm{ord}_{\xi_{H_x}}(w) > 
{r}/{p^e}$, since $\mathrm{ord}_{\xi_{H_x}}(w)$ is an integer.  Together with the observation, we conclude that 
$$\mathrm{In}_{\xi_{H_x}}(a'_{p^e}) = \mathrm{In}_{\xi_{H_x}}(a_{p^e}) = x^r g(y),$$ 
and hence that 
$$\rho_{h,(x,y,z),H_x}(P) 
= \frac1{p^e}{\mathrm{ord}_P(g(y))} 
= \frac1{p^e}{\mathrm{ord}_P\left(\left\{(h|_{Z'}) \cdot x^{-r}\right\}|_{Z' \cap H_x}\right)}.$$ 
 
\noindent\emph{Case$\colon 
r \equiv 0\ \mathrm{mod}\ p^e$}. 
 
In this case, set 
$$w = x^{{r}/{p^e}} \left\{h(y) + x \cdot \theta(x,y)\right\} 
\quad\text{with}\quad h(y) \in k[[y]], \theta(x,y) \in k[[x,y]].$$ 
Together with the observation, we conclude that 
$$\mathrm{In}_{\xi_{H_x}}(a'_{p^e}) = \mathrm{In}_{\xi_{H_x}}(a_{p^e}) - x^r h(y)^{p^e}= x^r \left\{g(y) - h(y)^{p^e}\right\},$$ 
and hence that 
\begin{align*} 
\rho_{h,(x,y,z),H_x}(P) 
&= \frac1{p^e} 
{\mathrm{res\text{-}ord}^{(p^e)}_P\left(g(y)\right)} 
\geq \frac1{p^e} 
{\mathrm{ord}_P\left(g(y) - h(y)^{p^e}\right)} 
\\ 
&= \frac1{p^e} 
{\mathrm{ord}_P\left(\left\{(h|_{Z'}) \cdot x^{-r}\right\}|_{Z' \cap H_x}\right)}. 
\end{align*} 
This completes the proof of Proposition 7. 
\end{proof} 
\begin{defn}[Invariant ``$\rho$''] Let the situation be 
as described in Proposition 7.  We define 
the invariant $\rho$ of $H_x$ at $P$, denoted by $\rho_x$, by the formula 
$$\rho_x =  \rho_{h,(x,y,z),H_x}(P).$$ 
Invariant $\rho_y$ is defined in an identical manner, 
in case $H_y$ is a bad boundary divisor (in $E_{\text{young}}$) 
passing through $P$ in configuration \textcircled{\footnotesize 5}. 
\end{defn} 
 
\medskip 
 
\textbf{Study of the behavior of the invariants under blow up} 
 
\medskip 
 
\noindent \fbox{\textbf{Case}$\colon 
$ blow up with 1-dimensional center} 
\begin{claim} Let $W \overset{\pi}\leftarrow \widetilde{W}$ be 
the blow up with center being a 1-dimensional component $C = V(z,x)$ 
of $\mathrm{Sing}({\mathcal R})$ (cf. Proposition 6), and 
$\widetilde{{\mathcal R}}$ the transformation of the idealistic 
filtration ${\mathcal R}$ of i.f.g.\! type.  Then there is possibly 
only one point $\widetilde{P} \in \mathrm{Sing}(\widetilde{{\mathcal R}}) 
\cap \pi^{-1}(P) \subset \widetilde{W}$, lying in the $x$-chart, 
with the regular system of parameters 
$(\widetilde{x},\widetilde{y},\widetilde{z}) = (x,y,z/x)$. 
The behavior of the invariants under blow up is described 
as follows (in case the invariant $\sigma$ does not drop): 
\begin{align*} 
\mu_{\widetilde{x}} &= \mu_x - 1, 
&h_{\widetilde{x}} &= h_x - 1, 
&\text{and, in case }H_x \text{ is bad,}\quad 
&\rho_{\widetilde{x}} = \rho_x. 
\\ 
\mu_{\widetilde{y}} &= \mu_y,\phantom{-1} 
&h_{\widetilde{y}} &= h_y,\phantom{-1} 
& \text{and, in case }H_y \text{ is bad,}\quad 
&\rho_{\widetilde{y}} = \rho_y - 1. 
\end{align*} 
\end{claim} 
\begin{proof}  Write down 
$$h = z^{p^e} + a_1z^{p^e-1} + a_2z^{p^e-2} + \cdots 
+ a_{p^e-1}z + a_{p^e}$$ 
with 
$$a_i \in k[[x,y]] \quad\text{and}\quad 
\mathrm{ord}_P(a_i) > i \quad\text{for}\quad i = 1, \ldots, p^e,$$ 
as described in \fbox{SITUATION} in \S 5.1.  We may assume further that $h$ is well-adapted at $P$, $\xi_{H_x}$ and $\xi_{H_y}$ simultaneously with respect to $(x,y,z)$ (cf. Proposition 5). 
 
It is straightforward to see that, if 
a point $\widetilde{P} \in \pi^{-1}(P) \subset \widetilde{W}$ 
lies in the $z$-chart, then $\mathrm{ord}_P(\widetilde{h}) < p^e$, where $\widetilde{h} = {\pi^*(h)}/{z^{p^e}}$, and hence 
$\widetilde{P} \not\in \mathrm{Sing}(\widetilde{{\mathcal R}})$. 
Therefore, there is possibly only one point 
$\widetilde{P} \in \mathrm{Sing}(\widetilde{{\mathcal R}}) \cap 
\pi^{-1}(P) \subset \widetilde{W}$, lying in the $x$-chart, 
with the regular system of parameters $(\widetilde{x},\widetilde{y},\widetilde{z}) = (x,y,z/x)$. 
 
With respect to this regular system of parameters, we compute the transform of the monomial $(x^{\alpha}y^{\beta},a)$ to be $(\widetilde{x}^{\alpha - a}\widetilde{y}^{\beta},a)$.  Therefore, we conclude 
$$ 
\mu_{\widetilde{x}} = \mu_x - 1, \quad 
\mu_{\widetilde{y}} = \mu_y. 
$$ 
Set 
$$a_{p^e} = x^r \left\{g(y) + x \cdot \omega(x,y)\right\}$$ 
with 
$$r = \mathrm{ord}_{\xi_{H_x}}(a_{p^e}),\quad 
0\neq g(y) \in k[[y]], 
\quad\text{and}\quad 
\omega(x,y) \in k[[x,y]].$$ 
After blow up, we compute 
$$\widetilde{h} 
= \frac{\pi^*(h)}{x^{p^e}} 
= \widetilde{z}^{p^e} + \widetilde{a_1}\widetilde{z}^{p^e-1} 
+ \widetilde{a_2}\widetilde{z}^{p^e-2} + \cdots 
+ \widetilde{a_{p^e-1}}\widetilde{z} + \widetilde{a_{p^e}}$$ 
with 
$$\widetilde{a_i} = \frac{\pi^*(a_i)}{x^i} 
\quad\text{for}\quad i = 1, \ldots, p^e.$$ 
In particular, we have 
$$ 
\widetilde{a_{p^e}} 
= \frac{\pi^*(a_{p^e})}{x^{p^e}} 
= x^{r - p^e} \left\{g(y) + x \cdot \omega(x,y)\right\} 
= \widetilde{x}^{r - p^e} \left\{g(\widetilde{y}) + \widetilde{x} \cdot \omega(\widetilde{x},\widetilde{y})\right\}. 
$$ 
(Note that we have $r \geq p^e$, since $h_x = {r}/{p^e} \geq 1$.  (cf. Proposition 6)) 
Therefore, we compute 
\begin{align*} 
\lefteqn{ 
\mathrm{Slope}_{\widetilde{h}, 
(\widetilde{x},\widetilde{y},\widetilde{z})}(\xi_{H_{\widetilde{x}}}) 
= \min\left\{ 
\frac1{p^e} 
{\mathrm{ord}_{\xi_{H_{\widetilde{x}}}}(\widetilde{a_{p^e}})},\ 
\mu(\xi_{H_{\widetilde{x}}})\right\} 
} 
\\ 
&= \min\left\{\frac{r - p^e}{p^e}, \mu_{\widetilde{x}}\right\} 
= \min\left\{\frac1{p^e}{\mathrm{ord}_{\xi_{H_x}}(a_{p^e})} - 1, 
\mu_x - 1\right\} 
\\ 
&= \min\left\{\frac1{p^e}{\mathrm{ord}_{\xi_{H_x}}(a_{p^e})}, \mu_x\right\} - 1 
= \mathrm{Slope}_{h,(x,y,z)}(\xi_{H_x}) - 1. 
\end{align*} 
Hence, if $\mathrm{Slope}_{h,(x,y,z)}(\xi_{H_x}) = \mu(\xi_{H_x})$, 
then 
$$\mathrm{Slope}_{\widetilde{h},(\widetilde{x}, 
\widetilde{y},\widetilde{z})}(\xi_{H_{\widetilde{x}}}) 
= \mu(\xi_{H_x}) - 1 = \mu(\xi_{H_{\widetilde{x}}}).$$ 
If $\mathrm{Slope}_{h,(x,y,z)}(\xi_{H_x}) < \mu(\xi_{H_x})$, then 
$$ 
\mathrm{Slope}_{\widetilde{h},(\widetilde{x},\widetilde{y}, 
\widetilde{z})}(\xi_{H_{\widetilde{x}}}) = 
\mathrm{Slope}_{h,(x,y,z)}(\xi_{H_x}) - 1 
<  \mu(\xi_{H_{\widetilde{x}}}) - 1 = \mu(\xi_{H_{\widetilde{x}}}) 
$$ 
and 
$$ 
\mathrm{In}_{\xi_{H_{\widetilde{x}}}}(\widetilde{a_{p^e}}) 
= \widetilde{x}^{r - p^e} g(\widetilde{y}) 
\text{ is not a $p^e$-th power}, 
$$ 
since $\mathrm{In}_{\xi_{H_x}}(a_{p^e}) = x^r g(y)$ is not a $p^e$-th power. 
Therefore, $\widetilde{h}$ is well-adapted at $\xi_{H_{\widetilde{x}}}$ with respect to $(\widetilde{x},\widetilde{y},\widetilde{z})$. 
 
Therefore, we conclude 
$$h_{\widetilde{x}} = \mathrm{Slope}_{\widetilde{h},(\widetilde{x},\widetilde{y},\widetilde{z})}(\xi_{H_{\widetilde{x}}}) = \mathrm{Slope}_{h,(x,y,z)}(\xi_{H_x}) - 1 = h_x - 1.$$ 
 
In case $H_x$ is bad, i.e., $h_x < \mu_x$, so is 
$H_{\widetilde{x}}$ since $h_{\widetilde{x}} 
= h_x - 1 < \mu_x - 1 = \mu_{\widetilde{x}}$.  Moreover, we compute 
\begin{align*} 
\rho_{\widetilde{x}} 
= \rho_{\widetilde{h},(\widetilde{x},\widetilde{y},\widetilde{z}),H_{\widetilde{x}}}(\widetilde{P}) 
&= 
\begin{cases} 
\mathrm{ord}_{\widetilde{P}}\left(g(\widetilde{y})\right)/p^e & \text{if }r - p^e \not\equiv 0 \text{ mod }\ p^e, \\ 
\mathrm{res\text{-}ord}^{(p^e)}_{\widetilde{P}}\left(g(\widetilde{y})\right)/p^e & \text{if }r - p^e \equiv 0 \text{ mod }\ p^e 
\end{cases} \\ 
&= 
\begin{cases} 
\mathrm{ord}_P\left(g(y)\right)/p^e 
& \text{if }r \not\equiv 0 \text{ mod }\ p^e, \\ 
\mathrm{res\text{-}ord}^{(p^e)}_P\left(g(y)\right)/p^e 
& \text{if }r \equiv 0 \text{ mod }\ p^e 
\end{cases} 
\\ 
&= \rho_{h,(x,y,z),H_x}(P) 
= \rho_x. 
\end{align*} 
In summary, we have 
$$h_{\widetilde{x}} = h_x - 1, \quad 
\rho_{\widetilde{x}} = \rho_x. $$ 
The proof for the formulas 
$$ 
h_{\widetilde{y}} = h_y, \quad 
\rho_{\widetilde{y}} = \rho_y -1 
$$ 
is similar, and left to the reader as an exercise. 
\end{proof} 
\begin{remark} The above proof is slightly sloppy in the sense that, after blow up, we may end up having $\mathrm{ord}_{\widetilde{P}}(\widetilde{a_{p^e}}) = p^e$ 
and $\mathrm{In}_{\widetilde{P}}(\widetilde{a_{p^e}}) 
= c\widetilde{x}^{p^e} + c'\widetilde{y}^{p^e}$ 
for some $(c, c')  \in k^2 \setminus \{(0,0)\}$, even under the condition $\widetilde{P} \in \pi^{-1}(P) \cap \mathrm{Sing}(\widetilde{{\mathcal R}}) \subset \widetilde{W}$ and the assumption that the invariant $\sigma$ stays the same (cf. the proof of Proposition 4 (1)).  Then we have to replace $(\widetilde{x},\widetilde{y},\widetilde{z})$ with 
$$(\widetilde{x},\widetilde{y},\widetilde{z}' 
= \widetilde{z} + c^{1/p^e}\widetilde{x} + c'^{1/p^e}\widetilde{y})$$ 
to guarantee condition (1) in \fbox{SITUATION} in \S 5.1.  Accordingly, we have to analyze $\widetilde{a_{p^e}}'$.  It is straightforward, however, to see that the same calculations hold with $\widetilde{a_{p^e}}'$.  The details are left to the reader as an exercise. 
\end{remark} 
 
\noindent \fbox{\textbf{Case}$\colon 
$ blow up with 0-dimensional center} 
 
\begin{claim} Let $W \overset{\pi}\leftarrow \widetilde{W}$ be the blow up with a 0-dimensional center $C = P = V(z,x,y) \in \mathrm{Sing}({\mathcal R})$ (cf. Proposition 6 and the description of the procedure in \S 5.3), and $\widetilde{{\mathcal R}}$ the transformation of the idealistic filtration ${\mathcal R}$ of i.f.g.\! type.  Set $Z = \{z = 0\}$.  Then a point 
$\widetilde{P} \in \mathrm{Sing}(\widetilde{{\mathcal R}}) 
\cap \pi^{-1}(P) \subset \widetilde{W}$ must be on the strict 
transform $Z'$ of $Z$, lying either in the $x$-chart or in the $y$-chart.  Assume that the invariant $\sigma$ stays the same, i.e., 
$\sigma(P) = \sigma(\widetilde{P})$.  We make the following 
three observations regarding the behavior of the invariants 
under blow up.  (We denote the strict transforms of $H_x$ and 
$H_y$ by  $H'_x$ and $H'_y$, and the exceptional divisor by $E_P$.  
Note that the pictures depict the configurations in a 2-dimensional 
manner, taking the intersection with the hypersurface $Z$ 
before blow up and with its strict transform $Z'$ after blow up. 
): 
 
\medskip 
 
\indent{\rm (1)}\quad 
The point $P$ is in case \textcircled{\footnotesize 3}, \textcircled{\footnotesize 4} or \textcircled{\footnotesize 5}. 
 
{\rm (1.1)}\quad 
Suppose $h_x < 1$.  Look at the point $\widetilde{P} = E_P \cap H'_x \cap Z'$ in the $y$-chart with a regular system of parameters $(\widetilde{x},\widetilde{y},\widetilde{z}) = (x/y,y,z/y)$.  Then the hypersurface $H'_x = H_{\widetilde{x}}$ is bad, and we have 
$$\rho_x > \rho_{\widetilde{x}}.$$ 
 
{\rm (1.2)}\quad 
Suppose $P$ is bad, and hence $E_P$ is 
also bad (cf. Lemma 4). 
Look at a point $\widetilde{P} \in (E_P \setminus H'_x) \cap Z'$ in the $x$-chart with a regular system of parameters 
$(\widetilde{x} = e,\widetilde{y},\widetilde{z}) 
= (x,y/x - c,z/x)$ for some $c \in k$.  Then we have 
$$\rho_x \geq \rho_e.$$ 
\ZU{ 
 \VCL{$H_x$}{bad} 
 \HDL{$H_y$}{} 
 \CPO{$P$} 
} 
\hskip143.5pt $\uparrow$ 
\ZU{ 
 \VLL{$H_x'$}{bad} 
 \RDL{$H_y'$}{} 
 \HOL{$E_P$}{$\mathrm{bad}_{(1.2)}$} 
 \LPO{$\widetilde{P}_{(1.1)}$} 
 \CPO{$\widetilde{P}_{(1.2)}$ \text{ or }} 
\RPO{$\widetilde{P}_{(1.2)}$} 
} 
 
\indent{\rm (2)}\quad 
The point $P$ is in case \textcircled{\footnotesize 4}.  Suppose $P$ is bad, and hence $E_P$ is also bad (cf. Lemma 4).  Look at a point $\widetilde{P} \in (E_P \setminus H'_y) \cap Z'$ with a regular system of parameters $(\widetilde{x},\widetilde{y} = e,\widetilde{z}) = (x/y - c,y,z/y)$ for some $c \in k$.  Then we have 
$$\mu_x > \rho_e.$$ 
 
\ZU{ 
 \VCL{$H_x$}{bad} 
 \HOL{$H_y$}{good} 
 \CPO{$P$} 
} 
\hskip143.5pt $\uparrow$ 
\ZU{ 
 \VLL{$H_x'$}{bad} 
 \VRL{$H_y'$}{good} 
 \HOL{$E_P$}{bad} 
 \LPO{$\widetilde{P}$ \text{ or }} 
 \CPO{$\widetilde{P}$} 
} 
 
\indent{\rm (3)}\quad 
The point $P$ is in case \textcircled{\footnotesize 5}.  Suppose $P$ is good.  Then we have 
$$\rho_x > \mu_y \quad\text{and}\quad \rho_y > \mu_x.$$ 
Look at the point $\widetilde{P} = E_P \cap H'_x \cap Z'$ in the $y$-chart 
with a regular system of parameters 
$(\widetilde{x},\widetilde{y},\widetilde{z}) = (x/y,y,z/y)$.  Since 
$\mu_{\widetilde{x}} = \mu_x$, we have as a consequence 
$$\rho_y > \mu_{\widetilde{x}}.$$ 
We draw a similar conclusion looking 
at the point $E_P \cap H'_y \cap Z'$ in the $x$-chart. 
 
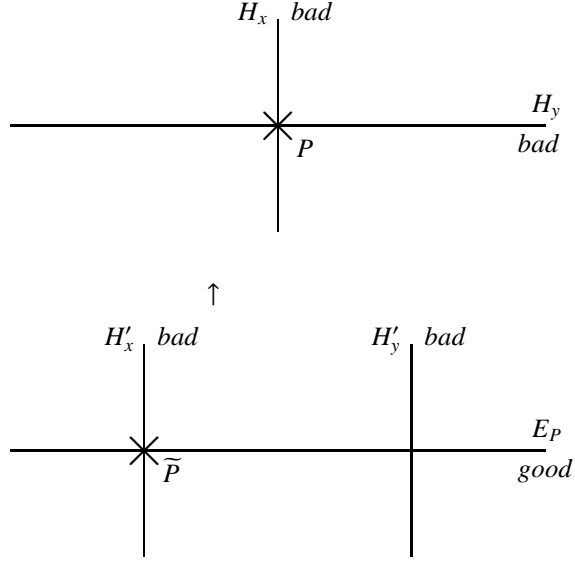
\ZU{ 
 \VCL{$H_x$}{bad} 
 \HOL{$H_y$}{bad} 
 \CPO{$P$} 
} 
\hskip143.5pt $\uparrow$ 
\ZU{ 
 \VLL{$H_x'$}{bad} 
 \VRL{$H_y'$}{bad} 
 \HOL{$E_P$}{good} 
 \LPO{$\widetilde{P}$} 
} 
\end{claim} 
\begin{proof} 
\indent{\rm (1) (1.1)}\quad 
Take (as described in \fbox{SITUATION} in \S 5.1) 
$$h = z^{p^e} + a_1z^{p^e-1} + a_2z^{p^e-2} + \cdots + a_{p^e-1}z + a_{p^e}$$ 
which is well-adapted at $\xi_{H_x}$ 
with respect to $(x,y,z)$ (cf. Proposition 5).  Set 
$$a_{p^e} = x^r \left\{g(y) + x \cdot \omega(x,y)\right\} 
\ \text{with}\ 0 \neq g(y) \in k[[y]],\ \omega(x,y) \in k[[x,y]].$$ 
Then we have 
$$h_x = \mathrm{H}(\xi_{H_x}) = \mathrm{Slope}_{h,(x,y,z)}(\xi_{H_x}) 
=\frac1{p^e} {\mathrm{ord}_{\xi_{H_x}}(a_{p^e})} 
= \frac{r}{p^e} < \mu(\xi_{H_x}).$$ 
Since $h_x < 1$ by assumption, we have $r < p^e$.  We compute 
$$ 
\widetilde{a_{p^e}} 
= \frac{a_{p^e}}{y^{p^e}} 
= \frac{(\widetilde{x}\widetilde{y})^r}{\widetilde{y}^{p^e}} 
\left\{g(\widetilde{y}) + \widetilde{x}\widetilde{y} \cdot 
\omega(\widetilde{x}\widetilde{y},\widetilde{y})\right\} 
= \widetilde{x}^r\left\{\widetilde{g}(\widetilde{y}) 
+ \widetilde{x} \cdot \widetilde{\omega}(\widetilde{x},\widetilde{y})\right\} 
$$ 
where 
$$\widetilde{g}(\widetilde{y}) = \widetilde{y}^{r - p^e} g(\widetilde{y}) 
\quad\text{and}\quad 
\widetilde{\omega}(\widetilde{x},\widetilde{y}) = \widetilde{y}^{r - p^e + 1} 
\omega(\widetilde{x}\widetilde{y},\widetilde{y}). 
$$ 
We observe that $\widetilde{h} = {h}/{y^{p^e}}$ is well-adapted at $\xi_{H_{\widetilde{x}}}$ with respect to $(\widetilde{x},\widetilde{y},\widetilde{z})$, since 
$$\mathrm{Slope}_{\widetilde{h},(\widetilde{x},\widetilde{y},\widetilde{z})}(\xi_{H_{\widetilde{x}}}) = \frac{r}{p^e} < \mu(\xi_{H_x}) = \mu(\xi_{H_{\widetilde{x}}})$$ 
and since 
$$\mathrm{In}_{\xi_{H_{\widetilde{x}}}}(\widetilde{a_{p^e}}) = \widetilde{x}^r\widetilde{g}(\widetilde{y}) =  \widetilde{x}^r \widetilde{y}^{r - p^e}g(\widetilde{y})$$ 
is not a $p^e$-th power, a fact which follows easily from the fact that 
$\mathrm{In}_{\xi_{H_x}}(a_{p^e}) = x^rg(y)$ is not a $p^e$-th power. 
Therefore, we conclude that 
\begin{align*} 
\rho_x &= \rho_{h,(x,y,z),H_x}(P) 
\\ 
&= 
\begin{cases} 
\mathrm{ord}_P\left(g(y)\right)/p^e & \text{in case }r \not\equiv 0 \text{ mod }p^e, \\ 
\mathrm{res\text{-}ord}^{(p^e)}_P\left(g(y)\right)/p^e & 
\text{in case }r \equiv 0 \text{ mod }p^e, \quad\text{i.e.,}\ \  r = 0 
\end{cases} 
\\ 
&> 
\begin{cases} 
\mathrm{ord}_{\widetilde{y}}\left(\widetilde{y}^{r - p^e}g(\widetilde{y})\right)/p^e & \text{in case }r \not\equiv 0 \text{ mod }p^e, \\ 
\mathrm{res\text{-}ord}^{(p^e)}_{\widetilde{y}}
\left(\widetilde{y}^{- p^e}g(\widetilde{y})\right)/p^e 
& \text{in case }r \equiv 0 \text{ mod }p^e,\quad \text{i.e.,}\ \  r = 0 
\end{cases} 
\\ 
&=\begin{cases} 
\mathrm{ord}_{\widetilde{y}}\left(\widetilde{g}(\widetilde{y})\right)/p^e 
& \text{in case }r \not\equiv 0 \text{ mod }p^e, \\ 
\mathrm{res\text{-}ord}^{(p^e)}_{\widetilde{y}}
\left(\widetilde{g}(\widetilde{y})\right)/p^e 
& \text{in case }r \equiv 0 \text{ mod }p^e, 
\quad\text{i.e.}\ \  r = 0 
\end{cases} 
\\ 
&= \rho_{\widetilde{h},(\widetilde{x},\widetilde{y},\widetilde{z}),H_{\widetilde{x}}}(\widetilde{P}) = \rho_{\widetilde{x}}. 
\end{align*} 
 
(Note that, even under the condition $\widetilde{P} \in \pi^{-1}(P) 
\cap \mathrm{Sing}(\widetilde{{\mathcal R}}) \subset \widetilde{W}$ 
and the assumption that the invariant $\sigma$ stays the same, there is a possibility that we may end up having $\mathrm{ord}_{\widetilde{P}}(\widetilde{a_{p^e}}) = p^e$ and $\mathrm{In}_{\widetilde{P}}(\widetilde{a_{p^e}}) = c\widetilde{y}^{p^e}$ for some $c \in k \setminus \{0\}$.  (In this case, we necessarily have $r = 0$.)  Then we have to replace $(\widetilde{x},\widetilde{y},\widetilde{z})$ with $(\widetilde{x},\widetilde{y},\widetilde{z}' = \widetilde{z} + c^{1/p^e}\widetilde{y})$ to guarantee condition (1) in \fbox{SITUATION}.  Accordingly, we have to analyze $\widetilde{a_{p^e}}'$.  It is straightforward, however, to see 
that the same calculations hold with $\widetilde{a_{p^e}}'$.) 
 
\medskip 
 
\indent{\rm (1) (1.2)}\quad 
Take (as described in \fbox{SITUATION} in \S 5.1) 
$$h = z^{p^e} + a_1z^{p^e-1} + a_2z^{p^e-2} + \cdots + a_{p^e-1}z + a_{p^e}$$ 
which is well-adapted at $P$ and $\xi_{H_x}$ simultaneously with respect to $(x,y,z)$ (cf. Proposition 5). 
 
Set 
\begin{align*} 
a_{p^e} &= \sum_{k + l \geq d} c_{kl}x^ky^l & \text{with}\ \ 
&d = \mathrm{ord}_P(a_{p^e}) > p^e\\ 
&= x^r\{g(y) + x \cdot \omega(x,y)\} & \text{with}\ \  
&r = \mathrm{ord}_{\xi_{H_x}}(a_{p^e}). 
\end{align*} 
Then we compute 
\begin{align*} 
\pi^*(a_{p^e}) 
&= \sum_{k + l \geq d}c_{kl}x^{k + l} 
\left(\frac{y}{x}\right)^l 
= x^d\left\{\sum_{k + l = d}c_{kl} 
\left(\frac{y}{x}\right)^l 
+ x \cdot 
\Omega\left(x,\frac{y}{x}\right) 
\right\} \\ 
&= x^d\left\{\phi\left(\frac{y}{x}\right) + x \cdot 
\Omega\left(x,\frac{y}{x}\right)\right\} 
\quad\text{where}\quad\phi(T) = \sum_{k + l = d}c_{kl}T^l, 
\end{align*} 
and hence 
\begin{align*} 
\widetilde{a_{p^e}} &= \frac{\pi^*(a_{p^e})}{x^{p^e}} = 
x^{d - p^e}\left\{ 
\phi\left(\frac{y}{x} - c + c\right) 
+ x \cdot 
\Omega\left(x,\frac{y}{x} - c + c\right) 
\right\} \\ 
&= \widetilde{x}^{d - p^e}\left\{\phi(\widetilde{y} + c) + \widetilde{x} \cdot \Omega(\widetilde{x},\widetilde{y} + c)\right\}. 
\end{align*} 
Moreover, just as in the analysis of the 
``\emph{Case$\colon P$ is a bad point}'' in Lemma 4, we see that 
$\widetilde{h}$ is well-adapted at $\xi_{E_P}$ with respect to $(\widetilde{x},\widetilde{y},\widetilde{z})$, since 
\begin{align*} 
\mathrm{Slope}_{\widetilde{h},(\widetilde{x},\widetilde{y},\widetilde{z})}(\xi_{H_{\widetilde{x}}}) 
&= \frac1{p^e}{\mathrm{ord}_{\xi_{H_{\widetilde{x}}}}(\widetilde{a_{p^e}})} 
= \frac{d - p^e}{p^e} 
\\ 
&=\frac1{p^e}{\mathrm{ord}_P(a_{p^e})} - 1 
\underset{\text{since }P \text{ is bad}}< \mu(P) - 1 
= \mu(\xi_{H_{\widetilde{x}}}), 
\end{align*} 
and since 
$\mathrm{In}_{\xi_{H_{\widetilde{x}}}(\widetilde{a_{p^e}}}) 
= \widetilde{x}^{d - p^e}\phi(\widetilde{y} + c)$ 
is not a $p^e$-th power, which is obvious if 
$d \not\equiv 0\ \mathrm{ mod }\ p^e$ 
and which follows from the fact 
$\mathrm{In}_P(a_{p^e}) = \sum_{k + l = d}c_{kl}x^ky^l$ 
is not a $p^e$-th power and hence 
$\phi(T) = \sum_{k + l = d}c_{kl}T^l$ is not a $p^e$-th power if 
$d \equiv 0\ \mathrm{ mod }\ p^e$. 
 
Therefore, we conclude 
\begin{align*} 
\rho_{\widetilde{x} = e} 
&= \rho_{\widetilde{h},(\widetilde{x},\widetilde{y},\widetilde{z}),H_{\widetilde{x}}} \\ 
&= 
\begin{cases} 
\mathrm{ord}_{\widetilde{y}}\left(\phi(\widetilde{y} + c)\right)/p^e 
& \text{in case }r \not\equiv 0 \text{ mod }p^e \\ 
\mathrm{res\text{-}ord}^{(p^e)}_{\widetilde{y}}\left(\phi(\widetilde{y} + c)\right)/p^e 
& \text{in case }r \equiv 0 \text{ mod }p^e 
\end{cases} 
\\ 
&\leq \frac1{p^e}\deg \phi(\widetilde{y} + c) = \frac1{p^e}\deg \phi(y). 
\end{align*} 
That is to say, we have 
$$(\star) \quad \rho_e \leq \frac1{p^e}\deg \phi(y).$$ 
On the other hand, set 
$$M = \left\{\begin{array}{ll} 
\mathrm{ord}_y\left(g(y)\right) & \text{in case }r \not\equiv 0 \text{ mod }p^e \\ 
\mathrm{res\text{-}ord}^{(p^e)}_y\left(g(y)\right) & \text{in case }r \equiv 0 \text{ mod }p^e 
\end{array}\right.$$ 
Then we conclude, for all those $(k,l)$ with $k + l = d$ and $c_{kl} \neq 0$, 
that we have 
$$ k + l = d \leq r + M, \quad\text{and}\quad k \geq r $$ 
and hence that 
$$l = d - k \leq d - r \leq (r + M) - r = M.$$ 
This implies 
$$(\star\star) \quad \deg \phi(y) = \deg \left(\sum_{k+l = d}c_{kl}y^l\right) \leq M.$$ 
{}From the inequalities $(\star)$ and $(\star\star)$, we finally conclude 
$$\rho_x = \frac{M}{p^e} \geq \frac1{p^e}\deg(\phi(y)) \geq \rho_e.$$ 
 
(Note that, even under the condition $\widetilde{P} \in \pi^{-1}(P) 
\cap \mathrm{Sing}(\widetilde{{\mathcal R}}) \subset \widetilde{W}$ 
and the assumption that the invariant $\sigma$ stays the same, there is a possibility that we may end up having $\mathrm{ord}_{\widetilde{P}}(\widetilde{a_{p^e}}) = p^e$ and $\mathrm{In}_{\widetilde{P}}(\widetilde{a_{p^e}}) = c\widetilde{x}^{p^e}$ for some $c \in k \setminus \{0\}$.  Then we have to replace $(\widetilde{x},\widetilde{y},\widetilde{z})$ with $(\widetilde{x},\widetilde{y},\widetilde{z}' = \widetilde{z} + c^{1/p^e}\widetilde{x})$ to guarantee condition (1) 
in \fbox{SITUATION} in \S 5.1. 
Accordingly, we have to analyze $\widetilde{a_{p^e}}'$.  It is straightforward, however, to see that the same calculations hold with $\widetilde{a_{p^e}}'$.) 
 
\medskip 
 
\indent{\rm (2)}\quad 
Take (as described in \fbox{SITUATION} in \S 5.1) 
$$h = z^{p^e} + a_1z^{p^e-1} + a_2z^{p^e-2} + \cdots + a_{p^e-1}z + a_{p^e}$$ 
which is well-adapted at $P$ and $\xi_{H_y}$ simultaneously with respect to $(x,y,z)$ (cf. Proposition 5).  Set 
$$a_{p^e} = \sum_{k + l \geq d} c_{kl}x^ky^l 
\quad\text{with}\quad d = \mathrm{ord}_P(a_{p^e}).$$ 
Then we see 
$$\frac{d}{p^e} = \frac1{p^e}{\mathrm{ord}_P(a_{p^e})} < \mu(P) 
\quad\text{and}\quad 
\mathrm{In}_P(a_{p^e}) \text{ is not a }p^e \text{-th power},$$ 
since $P$ is bad and since $h$ is well-adapted at $P$ with respect to $(x,y,z)$.  Then we conclude, for all those $(k,l)$ with $k + l = d$ and $c_{kl} \neq 0$, that we have 
$$\frac{\alpha + \beta}{a} = \mu(P) > \frac{d}{p^e} = \frac{k + l}{p^e} = \frac{k}{p^e} + \frac{l}{p^e} \geq \frac{k}{p^e} + \frac{\beta}{a},$$ 
since 
$$\frac{l}{p^e} \geq \frac1{p^e}{\mathrm{ord}_{\xi_{H_y}}(a_{p^e})} 
\geq \mu(\xi_{H_y}) = \frac{\beta}{a},$$ 
where the second inequality follows from the assumption that $H_y$ is good and that $h$ is well-adapted at $\xi_{H_y}$ with respect to $(x,y,z)$.  Therefore, we conclude 
$$\mu_x = \mu(\xi_{H_x}) = \frac{\alpha}{a} > \frac{k}{p^e}.$$ 
On the other hand, we compute 
\begin{align*} 
\pi^*(a_{p^e}) &= \sum_{k + l \geq d}c_{kl} 
\left(\frac{x}{y}\right)^ky^{k+l} = 
y^d\left\{\sum_{k + l = d}c_{kl}\left(\frac{x}{y}\right)^k + y \cdot 
\varOmega\left(\frac{x}{y},y\right)\right\} \\ 
&= y^d\left\{\varphi\left(\frac{x}{y}\right) + y \cdot 
\varOmega\left(\frac{x}{y},y\right)\right\}, 
\end{align*} 
where $\varphi(T) = \sum_{k + l = d}c_{kl}T^k$, 
and hence 
\begin{align*} 
\widetilde{a_{p^e}} 
&= \frac{\pi^*(a_{p^e})}{y^{p^e}} 
= y^{d - p^e}\left\{ 
\varphi\left(\frac{x}{y} - c + c\right) + y \cdot 
\varOmega\left(\frac{x}{y} - c + c,y\right)\right\} \\ 
&= \widetilde{y}^{d - p^e}\left\{\varphi(\widetilde{x} + c) + \widetilde{y} \cdot \varOmega(\widetilde{x} + c,\widetilde{y})\right\}. 
\end{align*} 
Moreover, just as in the analysis of the 
``\emph{Case$\colon P$ is a bad point}'' in Lemma 4, 
we see that $\widetilde{h}$ is well-adapted at 
$\xi_{E_P}$ with respect to $(\widetilde{x},\widetilde{y},\widetilde{z})$, 
since 
\begin{align*} 
\lefteqn{ 
\mathrm{Slope}_{\widetilde{h},(\widetilde{x}, 
\widetilde{y},\widetilde{z})}(\xi_{H_{\widetilde{y}}}) = 
\frac1{p^e} {\mathrm{ord}_{\xi_{H_{\widetilde{y}}}}(\widetilde{a_{p^e}})} 
= \frac{d - p^e}{p^e} 
} \\ 
&\qquad= \frac1{p^e}{\mathrm{ord}_P(a_{p^e})} - 1 
\underset{\text{since }P \text{ is bad}}< \mu(P) - 1 
= \mu(\xi_{H_{\widetilde{y}}}), 
\end{align*} 
and since $\mathrm{In}_{\xi_{H_{\widetilde{y}}}}(\widetilde{a_{p^e}}) = \widetilde{y}^{d - p^e}\varphi(\widetilde{x} + c)$ is not a $p^e$-th power, which is obvious if $d \not\equiv 0\ \mathrm{ mod }\ p^e$ and which follows from the fact $\mathrm{In}_P(a_{p^e}) = \sum_{k + l = d}c_{kl}x^ky^l$ is not a $p^e$-th power and hence $\varphi(T) = \sum_{k + l = d}c_{kl}T^k$ is not a $p^e$-th power if $d \equiv 0\ \mathrm{ mod }\ p^e$. 
 
Therefore, we conclude 
\begin{align*} 
&\rho_{\widetilde{y} = e} 
= \rho_{\widetilde{h},(\widetilde{x},\widetilde{y},\widetilde{z}), 
H_{\widetilde{y}}} 
= 
\begin{cases} 
\mathrm{ord}_{\widetilde{x}}\left(\varphi(\widetilde{x} + c)\right)/p^e & 
\text{if }r \not\equiv 0 \text{ mod }p^e \\ 
\mathrm{res\text{-}ord}^{(p^e)}_{\widetilde{x}}\left(\varphi(\widetilde{x} + c)\right)/p^e & 
\text{if }r \equiv 0 \text{ mod }p^e 
\end{cases} 
\\ 
&\quad 
\leq 
\frac1{p^e}\deg \varphi(\widetilde{y} + c) 
=\frac1{p^e} \deg \varphi(y) 
= \frac1{p^e}\deg\left(\sum_{k + l = d}c_{kl}T^k\right) 
< \frac{\alpha}{a} = \mu_x. 
\end{align*} 
(Note that, even under the condition $\widetilde{P} \in \pi^{-1}(P) 
\cap \mathrm{Sing}(\widetilde{{\mathcal R}}) \subset \widetilde{W}$ 
and the assumption that the invariant $\sigma$ stays the same, there is a possibility that we may end up having $\mathrm{ord}_{\widetilde{P}}(\widetilde{a_{p^e}}) = p^e$ and $\mathrm{In}_{\widetilde{P}}(\widetilde{a_{p^e}}) 
= c\widetilde{y}^{p^e}$ for some $c \in k \setminus \{0\}$. 
(In this case, we necessarily have $r = 0$.) 
Then we have to replace $(\widetilde{x},\widetilde{y},\widetilde{z})$ with $(\widetilde{x},\widetilde{y},\widetilde{z}' 
= \widetilde{z} + c^{1/p^e}\widetilde{y})$ to guarantee 
condition (1) in \fbox{SITUATION} in \S 5.1. 
Accordingly, we have to analyze $\widetilde{a_{p^e}}'$. 
It is straightforward, however, 
to see that the same calculations hold with $\widetilde{a_{p^e}}'$.) 
 
\smallskip 
 
\indent{\rm (3)}\quad 
Take (as described in \fbox{SITUATION} in \S 5.1) 
$$h = z^{p^e} + a_1z^{p^e-1} + a_2z^{p^e-2} + \cdots + a_{p^e-1}z + a_{p^e}$$ 
which is well-adapted at $P$ and $\xi_{H_x}$ 
simultaneously with respect to $(x,y,z)$ (cf. Proposition 5). 
 
Set 
$$a_{p^e} = x^r \left\{g(y) + x \cdot \omega(x,y)\right\} 
\ \text{with}\  0 \neq g(y) \in k[[y]],\ \omega(x,y) \in k[[x,y]].$$ 
Set 
$$M = 
\begin{cases} 
\mathrm{ord}_P\left(g(y)\right) 
& \text{in case }r \not\equiv 0 \text{ mod }p^e, \\ 
\mathrm{res\text{-}ord}^{(p^e)}_P\left(g(y)\right) & \text{in case }r \equiv 0 \text{ mod }p^e 
\end{cases} 
$$ 
so that 
$$\rho_x = \rho_{h,(x,y,z),H_x}(P) = \frac{M}{p^e}.$$ 
Now we compute 
\begin{align*} 
\frac{r + M}{p^e} \geq 
\frac1{p^e}{\mathrm{ord}_P(a_{p^e})} 
& \underset{\text{since }P\text{ is good}}\geq  \mu(P) 
= \frac{\alpha + \beta}{a} = \mu(\xi_{H_x}) + \frac{\beta}{a} \\ 
&\underset{\text{since }H_x\text{ is bad}}> 
\mathrm{H}(\xi_{H_x}) + \frac{\beta}{a} 
= \frac{r}{p^e} + \frac{\beta}{a}. 
\end{align*} 
Therefore, we conclude 
$$\rho_x = \frac{M}{p^e} > \frac{\beta}{a} = \mu(\xi_{H_y}) = \mu_y.$$ 
The proof for the inequality $\rho_y > \mu_x$ is identical. 
 
\smallskip 
 
This completes the proof of Claim 2. 
\end{proof} 
\begin{theorem} Let $P \in \mathrm{Sing}({\mathcal R}) \subset W$ be a point in the monomial case as described in \fbox{SITUATION} in \S 5.1.  Let $W \overset{\pi}\leftarrow \widetilde{W}$ be the blow up with center $C$ specified by the procedure described in \S 5.3, and $\widetilde{{\mathcal R}}$ the transformation of the idealistic filtration ${\mathcal R}$ of i.f.g.\! type. 
 
Then at $\widetilde{P} \in \pi^{-1}(P) \in \widetilde{W}$, one of the following holds: 
\begin{itemize} 
\item $\widetilde{P} \not\in \mathrm{Sing}(\widetilde{R})$, i.e., 
$\mathrm{Sing}(\widetilde{R}) = \emptyset$ in a neighborhood of 
$\widetilde{P}$, 
\item $\widetilde{P} \in \mathrm{Sing}(\widetilde{R})$ and the 
invariant $\sigma$ drops, or 
\item $\widetilde{P} \in \mathrm{Sing}(\widetilde{R})$ and the invariant ``$\mathrm{inv}_{\mathrm{MON}}(P)$'' strictly decreases, i.e., $\mathrm{inv}_{\mathrm{MON}}(P) > \mathrm{inv}_{\mathrm{MON}}(\widetilde{P})$. 
\end{itemize} 
 
Since the invariant ``$\mathrm{inv}_{\mathrm{MON}}$'' can not decrease infinitely many times, the procedure described in \S 5.3 terminates either with $\widetilde{P} \not\in \mathrm{Sing}(\widetilde{R})$ or with the drop of the invariant $\sigma$ after finitely many repetitions. 
 
That is to say, the basic strategy to show termination of the procedure in the monomial case with $\tau = 1$ is established. 
\end{theorem} 
 
\begin{proof} We have only to show (cf. Proposition 4) that, 
at $\widetilde{P} \in \pi^{-1}(P) \in \widetilde{W}$, assuming 
$\widetilde{P} \in \mathrm{Sing}(\widetilde{R})$ and 
$\sigma(P) = \sigma(\widetilde{P})$, we have 
$$\mathrm{inv}_{\mathrm{MON}}(P) > \mathrm{inv}_{\mathrm{MON}}(\widetilde{P}).$$ 
\noindent\emph{Case$\colon 
\dim C = 1$}. 
 
In this case, by Claim 1, it is easy to see that $P$ and $\widetilde{P}$ are in the same configuration and 
$$\mathrm{inv}_{\mathrm{MON}}(P) > \mathrm{inv}_{\mathrm{MON}}(\widetilde{P}).$$ 
\noindent\emph{Case$\colon 
\dim C = 0$, i.e., $C = P$}. 
 
\smallskip 
 
\noindent\underline{1.\ $P$ is in configuration \textcircled{\footnotesize 1} 
or \textcircled{\footnotesize 2}} 
 
\smallskip 
 
In this subcase, since $P \in \mathrm{Sing}({\mathcal R})$ is isolated, 
we see by Proposition 6 that $\mu_x = h_x < 1$ (and $\mu_y = h_y < 1$ in configuration \textcircled{\footnotesize 2}).  This implies $\mu_e < \mu_x$ ($\mu_e < \min\{\mu_x,\mu_y\}$ in configuration \textcircled{\footnotesize 2}).  Moreover, 
it is easy to see that $P$ is necessarily a good point, hence $E_P$ is also good, and that $\widetilde{P}$ is also in configuration \textcircled{\footnotesize 1} or \textcircled{\footnotesize 2}.  Now it is straightforward to see 
$$\mathrm{inv}_{\mathrm{MON}}(P) > \mathrm{inv}_{\mathrm{MON}}(\widetilde{P}).$$ 
 
\noindent\underline{2.\ $P$ is in configuration \textcircled{\footnotesize 3}} 
 
\smallskip 
 
\noindent\emph{{\rm (i)} 
$\widetilde{P} = \left(E_P \setminus H'_x\right) \cap Z'$ with 
$E_P$ being bad, and hence $\widetilde{P}$ is 
in configuration \textcircled{\footnotesize 3}}.\quad 
In this subcase, we have $\rho_x \geq \rho_e$ by Claim 2 (1.2), 
while $\mu_x > \mu_x - 1 = \mu_e$.  Therefore, we conclude 
$$\mathrm{inv}_{\mathrm{MON}}(P) = (\rho_x,0,\mu_x) >  (\rho_e,0,\mu_e) = \mathrm{inv}_{\mathrm{MON}}(\widetilde{P}).$$ 
\noindent\emph{{\rm (ii)} 
$\widetilde{P} = E_P \cap H'_x \cap Z'$ with $E_P$ being good, 
and hence $\widetilde{P}$ is in configuration 
\textcircled{\footnotesize 4}}.\quad 
In this subcase, we have $\rho_x > \rho_{\widetilde{x}}$ by Claim 2 (1.1). 
Therefore, we conclude 
$$\mathrm{inv}_{\mathrm{MON}}(P) = (\rho_x,0,\mu_x) > (\min\{\rho_{\widetilde{x}},\mu_{\widetilde{x}}\}, \max\{\rho_{\widetilde{x}},\mu_{\widetilde{x}}\}) = \mathrm{inv}_{\mathrm{MON}}(\widetilde{P}).$$ 
\noindent\emph{{\rm (iii)} 
$\widetilde{P} = E_P \cap H'_x \cap Z'$ with $E_P$ being bad, 
and hence $\widetilde{P}$ is in configuration 
\textcircled{\footnotesize 5}}.\quad 
In this case, we have $\rho_x > \rho_{\widetilde{x}}$ by Claim 2 (1.1).  Therefore, we conclude 
$$\mathrm{inv}_{\mathrm{MON}}(P) = (\rho_x,0,\mu_x) > (\min\{\rho_{\widetilde{x}},\rho_{\widetilde{y}}\}, \max\{\rho_{\widetilde{x}},\rho_{\widetilde{y}}\}) = \mathrm{inv}_{\mathrm{MON}}(\widetilde{P}).$$ 
\noindent\underline{3.\ $P$ is in configuration \textcircled{\footnotesize 4}} 
 
\smallskip 
 
\noindent\emph{{\rm (i)} 
$\widetilde{P} = \left(E_P \setminus (H'_x \cup H'_y)\right) \cap Z'$ with $E_P$ being bad, 
and hence $\widetilde{P}$ is in configuration 
\textcircled{\footnotesize 3}}.\quad 
In this case, we have $\rho_x \geq \rho_e$ and 
$\mu_x > \rho_e (\geq 0)$ by Claim 2 (1.2) and (2). 
Therefore, we conclude 
$$\mathrm{inv}_{\mathrm{MON}}(P) 
= (\min\{\rho_x,\mu_x\},\max\{\rho_x,\mu_x\}) 
> (\rho_e,0,\mu_e) = \mathrm{inv}_{\mathrm{MON}}(\widetilde{P}).$$ 
\noindent\emph{{\rm (ii)} 
$\widetilde{P} = E_P \cap H'_x \cap Z'$ with $E_P$ being good, 
and hence $\widetilde{P}$ is in configuration 
\textcircled{\footnotesize 4}}.\quad 
In this case, we have $\rho_x > \rho_{\widetilde{x}}$ by Claim 2 (1.1), while $\mu_x = \mu_{\widetilde{x}}$.  Therefore, we conclude 
\begin{align*} 
\mathrm{inv}_{\mathrm{MON}}(P) 
&= (\min\{\rho_x,\mu_x\},\max\{\rho_x,\mu_x\}) 
\\ 
&> (\min\{\rho_{\widetilde{x}},\mu_{\widetilde{x}}\}, 
\max\{\rho_{\widetilde{x}},\mu_{\widetilde{x}}\}) 
= \mathrm{inv}_{\mathrm{MON}}(\widetilde{P}). 
\end{align*} 
\noindent\emph{{\rm (iii)} 
$\widetilde{P} = E_P \cap H'_y \cap Z'$ with $E_P$ being bad, 
and hence $\widetilde{P}$ is in configuration 
\textcircled{\footnotesize 4}}.\quad 
In this case, we have $\rho_x \geq \rho_e$ by Claim 2 (1.2), while 
$\mu_x = {\alpha}/{a} > ({\alpha + \beta})/{a} - 1 
= \mu_e$ since ${\beta}/{a} = \mu_y = h_y < 1$. 
Therefore, we conclude 
\begin{align*} 
\mathrm{inv}_{\mathrm{MON}}(P) 
&= (\min\{\rho_x,\mu_x\},\max\{\rho_x,\mu_x\}) 
\\ 
&> (\min\{\rho_e,\mu_e\},\max\{\rho_e,\mu_e\}) 
= \mathrm{inv}_{\mathrm{MON}}(\widetilde{P}). 
\end{align*} 
\noindent\emph{{\rm (iv)} 
$\widetilde{P} = E_P \cap H'_x \cap Z'$ with $E_P$ being bad, 
and hence $\widetilde{P}$ is in configuration 
\textcircled{\footnotesize 5}}.\quad 
In this case, we have $\rho_x > \rho_{\widetilde{x}}$ and $\mu_x > \rho_{\widetilde{y} = e}$ by Claim 2 (1.1) and (2).  Therefore, we conclude 
\begin{align*} 
\mathrm{inv}_{\mathrm{MON}}(P) 
&= (\min\{\rho_x,\mu_x\},\max\{\rho_x,\mu_x\}) \\ 
&> (\min\{\rho_{\widetilde{x}},\rho_{\widetilde{y}}\}, \max\{\rho_{\widetilde{x}},\rho_{\widetilde{y}}\}) 
= \mathrm{inv}_{\mathrm{MON}}(\widetilde{P}). 
\end{align*} 
\noindent\underline{4.\ $P$ is in configuration \textcircled{\footnotesize 5}} 
 
\medskip 
 
\noindent\emph{{\rm (i)} 
$\widetilde{P} = \left(E_P \setminus (H'_x \cup H'_y)\right) \cap Z'$ 
with $E_P$ being bad, and hence $\widetilde{P}$ is in configuration 
\textcircled{\footnotesize 3}}.\quad 
In this case, we have $\rho_x \geq \rho_e$ and $\rho_y \geq \rho_e$ 
by Claim 2 (1.2).  We also have $\rho_x, \rho_y > 0$ since $h_y, h_x < 1$ 
with $P \in \mathrm{Sing}({\mathcal R})$.  Therefore, we conclude 
$$\mathrm{inv}_{\mathrm{MON}}(P) = (\min\{\rho_x,\rho_y\},\max\{\rho_x,\rho_y\}) > (\rho_e,0,\mu_e) = \mathrm{inv}_{\mathrm{MON}}(\widetilde{P}).$$ 
\noindent\emph{{\rm (ii)} 
$\widetilde{P} = E_P \cap H'_x \cap Z'$ with $E_P$ being good, 
and hence $\widetilde{P}$ is in configuration 
\textcircled{\footnotesize 4}}.\quad 
In this case, we have $\rho_x > \rho_{\widetilde{x}}$ and 
$\rho_y > \mu_{\widetilde{x}}$ by Claim 2 (1.1) and (3). 
Therefore, we conclude 
\begin{align*} 
\mathrm{inv}_{\mathrm{MON}}(P) 
&= (\min\{\rho_x,\rho_y\},\max\{\rho_x,\rho_y\}) 
\\ 
&> (\min\{\rho_{\widetilde{x}},\mu_{\widetilde{x}}\}, 
\max\{\rho_{\widetilde{x}},\mu_{\widetilde{x}}\}) 
= \mathrm{inv}_{\mathrm{MON}}(\widetilde{P}). 
\end{align*} 
\noindent\emph{{\rm (iii)} 
$\widetilde{P} = E_P \cap H'_x \cap Z'$ with $E_P$ being bad, 
and hence $\widetilde{P}$ is in configuration 
\textcircled{\footnotesize 5}}.\quad 
In this case, we have $\rho_x > \rho_{\widetilde{x}}$ and 
$\rho_y \geq \rho_{\widetilde{y} = e}$ by Claim 2 (1.1) and (1.2). 
Therefore, we conclude 
\begin{align*} 
\mathrm{inv}_{\mathrm{MON}}(P) 
&= (\min\{\rho_x,\rho_y\},\max\{\rho_x,\rho_y\}) 
\\ 
&> (\min\{\rho_{\widetilde{x}},\rho_{\widetilde{y}}\}, 
\max\{\rho_{\widetilde{x}},\rho_{\widetilde{y}}\}) 
= \mathrm{inv}_{\mathrm{MON}}(\widetilde{P}). 
\end{align*} 
 
\medskip 
 
This completes the proof of Theorem 1. 
\end{proof} 
 
This completes the detailed discussion of the monomial case in dimension 3, and hence completes the presentation of (the local version of) our algorithm in dimension 3. 
 
\medskip 
 
We finish this paper by making a couple of remarks. 
\begin{remark}[\textbf{Invariant whose maximum locus determines 
the center of blow up in the monomial case?}] 
The invariant ``$\mathrm{inv}_{\mathrm{MON}}$'' is only used to show effectively the termination of the procedure (in the monomial case), while the choice of the center is dictated by the study of the dimension of the singular locus (cf. Proposition 6 and the description of the procedure in 5.3).  Actually all the existing algorithms, including the one in \cite{BV3}, use the analysis of the dimension of the singular locus for the choice of the center. 
 
It would be desirable to have an invariant 
(manifested as the invariant $\Gamma$ in characteristic zero), 
which satisfies the following properties: 
 
\medskip 
 
\indent{\rm (1)}\quad 
it is upper semi-continuous, 
 
\indent{\rm (2)}\quad 
its maximum locus determines the nonsingular center of blow up for constructing the sequence of transformations for resolution of singularities, and 
 
\indent{\rm (3)}\quad 
it strictly drops after each blow up (over the center), and can not strictly decrease infinitely many times. 
 
\medskip 
 
Such an invariant would not only show the termination effectively but also dictate the choice of the center. 
\end{remark} 
 
\begin{remark} \ 
 
\indent{\rm (1)}(\textbf{Global version})\quad 
The presentation of our 
algorithm in this paper is restricted to the local version (cf. Remark 2).  However, it is not so difficult, though technical and rather lengthy, to make some adjustments in order for us to turn the local version into the global version.  The detailed discussion of these adjustments will be published elsewhere. 
 
\smallskip 
 
\indent{\rm (2)}(\textbf{Embedded resolution of singularities of a surface 
$X \subset W$ with $\dim W = 3$})\quad Consider a surface $X \subset W$ 
embedded in a nonsingular ambient space $W$ of $\dim W = 3$.  Then we can establish embedded resolution of singularities for $X \subset W$ (cf. Problem 2) in the following manner: 
 
\indent{\rm (i)}\quad 
The problem of embedded resolution of singularities for $X \subset W$ is reduced to the problem of resolution of singularities for the triplet 
$(W,({\mathcal I}_X, 1),\emptyset)$ (cf. Lemma 1). 
 
\indent{\rm (ii)}\quad 
The problem of resolution of singularities for the triplet 
$(W,  ({\mathcal I}_X  ,1),\emptyset)$ is equivalent to 
the problem of resolution of singularities for $(W,{\mathcal R},\emptyset)$ where the idealistic filtration of i.f.g.type is given by ${\mathcal R} = {\mathcal G}(({\mathcal I}_X,1))$ (cf. \S 2. Overview). 
 
\indent{\rm (iii)}\quad 
The problem of resolution of singularities for $(W,{\mathcal R},\emptyset)$ is solved by the global version of our algorithm in $\dim W = 3$. 
 
\smallskip 
 
\indent{\rm (3)}(\textbf{Embedded resolution of singularities of a surface 
$X \subset W$ with $\dim W$ arbitrary})\quad 
We can establish the global version of our algorithm for resolution 
of singularities of the triplet $(W,{\mathcal R},E)$ with $\dim W$ arbitrary, 
as long as it satisfies the condition 
$\tau(P) \geq \dim W - 2$ for all $P \in W$.  As an application, we can give an alternative proof for embedded 
resolution of a surface $X \subset W$ with $\dim W$ arbitrary, 
a theorem established by \cite{BV2} and \cite{CJS}.  We only describe the outline of our alternative proof below (while the details will be published elsewhere): 
 
\indent{\rm (i)}\quad 
We consider the Hilbert-Samuel function ${\rm HS}$ over $X$, 
and its maximum value $M = \max_{x \in X}\{{\rm HS}(x)\}$. 
 
\indent{\rm (ii)}\quad 
Construct a covering $X = \bigcup_{\lambda \in \Lambda}U_{\lambda}$, 
and a pair $({\mathcal I}_{\lambda},a_{\lambda})$ of an ideal 
${\mathcal I}_{\lambda}$ over $U_{\lambda}$ with level 
$a_{\lambda} \in {\mathbb Z}_{> 0}$ for each $\lambda \in \Lambda$ 
(in \'etale topology using the result due to Aroca (cf. \cite{EV}) 
or in Zariski topology using the result by \cite{BM}) satisfying 
the following property: 
 
$(\spadesuit)$ the singular locus of the pair coincides with 
the maximum locus of the Hilbert-Samuel function, i.e., 
$$\mathrm{Sing}({\mathcal I}_{\lambda},a_{\lambda}) 
= \{P \in U_{\lambda} \mid {\rm HS}(P) = M\},$$ 
and this relation persists through \emph{any} sequence of 
transformations (and smooth morphisms). 
 
\indent{\rm (iii)}\quad 
For each $\lambda \in \Lambda$, consider the 
idealistic filtration of i.f.g.type ${\mathcal R}_{\lambda}' 
= {\mathcal G}(({\mathcal I}_{\lambda},a_{\lambda}))$. 
We take its 
${\mathcal R}$adical-$\&$-${\mathcal D}$ifferential saturation 
${\mathcal R}_{\lambda} := {\mathcal R}{\mathcal D} 
({\mathcal R}_{\lambda}')$ (cf. \cite{K}). 
Then by the result of \cite{BraGV} we see that 
the ${\mathcal R}_{\lambda}$'s patch together, i.e., 
${\mathcal R}_{\lambda}|_{U_{\lambda} \cap U_{\mu}} 
= {\mathcal R}_{\mu}|_{U_{\mu} \cap U_{\lambda}}$, 
to give rise to the idealistic filtration ${\mathcal R}$ 
over $W$ such that ${\mathcal R}|_{U_{\lambda}} = {\mathcal R}_{\lambda}$ 
for all $\lambda \in \Lambda$, that $\mathrm{Sing}({\mathcal R}) = \{P \in W 
\mid {\rm HS}(P) = M\}$, and that this relation persists through any sequence of transformations (and smooth morphisms).  Note that the ${\mathcal R}$-$\&$-${\mathcal D}$ saturation gives the ``biggest'' idealistic filtration of i.f.g.type with the property $(\spadesuit)$. 
 
\indent{\rm (iv)}\quad 
Observe that, for this idealistic filtration ${\mathcal R}$, 
we have $\tau(P) \geq \dim W - 2$ 
for all $P \in W$ by the result in \cite{BraV1}\cite{BraV2}. 
 
\indent{\rm (v)}\quad 
Apply the global version of our algorithm to obtain resolution of 
singularities for $(W,{\mathcal R},E_M = \emptyset)$, which implies 
the strict decrease of (the maximum value of) the Hilbert-Samuel function. 
 
\indent{\rm (vi)}\quad 
Repeat the procedure.  (Note that in the middle of the repetition of 
the procedure, the boundary $E_M$, which is the union of all 
the exceptional divisors created so far, may not be empty.) 
 
\indent{\rm (vii)}\quad 
Since the value of the Hilbert-Samuel function can not strictly decrease 
infinitely many times, the procedure must come to an end after finitely 
many repetitions, providing embedded resolution of singularities for 
$X \subset W$. 
\end{remark} 
\end{subsection} 
\end{section} 
 

\end{document}